\documentclass[12pt,a4paper]{article}

\usepackage{ifthen} 
\usepackage{calc}  
\usepackage[german,english]{babel} 
\usepackage[all]{xy}
\usepackage{amsmath}
\usepackage{amscd}
\usepackage{amssymb}
\usepackage{latexsym}
\usepackage{enumerate}
\usepackage{theorem}
\usepackage{mathrsfs}
\usepackage{helvet}
\usepackage{lastpage} 
\usepackage{prettyref} 
\usepackage[pagebackref]{hyperref} 

\input xy 

\newboolean{finalversion} 
\setboolean{finalversion}{true}

\theoremheaderfont{\scshape}

\newrefformat{chap}{\hyperref[{#1}]{Chapter~\ref*{#1}}}
\newrefformat{sec}{\hyperref[{#1}]{Section~\ref*{#1}}}

\newtheorem{theorem}{Theorem}[section]
\newrefformat{thm}{\hyperref[{#1}]{Theorem~\ref*{#1}}}
\newtheorem{definition}[theorem]{Definition}
\newrefformat{def}{\hyperref[{#1}]{Definition~\ref*{#1}}}
\newtheorem{lemma}[theorem]{Lemma}
\newrefformat{lem}{\hyperref[{#1}]{Lemma~\ref*{#1}}}
\newtheorem{proposition}[theorem]{Proposition}
\newrefformat{prop}{\hyperref[{#1}]{Proposition~\ref*{#1}}}
\newtheorem{corollary}[theorem]{Corollary}
\newrefformat{cor}{\hyperref[{#1}]{Corollary~\ref*{#1}}}
{\theorembodyfont{\upshape}
\newtheorem{examplecore}[theorem]{Example}}
\newrefformat{ex}{\hyperref[{#1}]{Example~\ref*{#1}}}

\newtheorem{remark}[theorem]{Remark}
\newrefformat{rem}{\hyperref[{#1}]{Remark~\ref*{#1}}}

\newcommand{\Spec}{\ensuremath{\operatorname{Spec}}}

\newcommand{\codim}{\ensuremath{\operatorname{codim}}}

\newenvironment{proof}{\noindent\textsc{Proof:}}{\hspace*{\fill}
$\blacksquare$\par\vspace{.1cm}} 

\newenvironment{example}{\begin{examplecore}}{\hspace*{\fill}
$\square$\par\vspace{.1cm}\end{examplecore}}

\newcommand{\mylabel}[1]{\label{#1}\ifthenelse{\boolean{finalversion}}{
  }{\marginpar{\tiny #1}}}

\title{Valuation theoretic methods in the birational geometry of algebraic varieties}

\author{Stefan G\"unther}

\date{December, 2015}

\begin{document}

\maketitle
\begin{abstract} In this paper, we give a valuation formula for rational top differential forms of function fields in characteristic zero  for arbitrary Abhyankar places generalizing the classical valuation of rational top differential forms at prime divisors. This enables us to defne log discrepancies $a(X,\Delta,\nu)$\, for log pairs $(X,\Delta)$\, for arbitrary Abhyankar places $\nu$.\\
If $\nu$\, is an Abhyankar place of dimension greater than zero, we restrict rational top differential forms $\omega$\, to rational top differential forms $\overline{\omega}$\, of the residue field $\kappa(\nu)$\, of $\nu$\,, generalizing the classical restriction of top differential forms with a simple pole along a smooth divisor.\\
This opens up the door to generalize the classical adjunction machinery to arbitrary Abhyankar places.
\end{abstract} 
\tableofcontents

\section{Notations and Conventions}
We work over the field of complex numbers $\mathbb C$\,. In this paper we will always be concerned with fields $K/\mathbb C$ that have finite transcendence degree over $\mathbb C$\,. There is then always a complete normal integral variety $X$ such that $K\subseteq \overline{K(X)}$\, where the last expression denotes the algebraic closure of the function field of $X$.\\
In case $K$ is finitely generated over $\mathbb C,$\, by $\mbox{Mod}(K/\mathbb C)$\, we denote the partially ordered set of all complete normal models of the function field $K=K(X)/\mathbb C.$\, We say that $Y>X$\, iff there is a necessarily unique proper birational morphism $p: Y\longrightarrow X$\,. We use the expression "for a sufficiently high model $X\in \mbox{Mod}(K/\mathbb C)"$\, to say that there is an $Y\in \mbox{Mod}(K/\mathbb C)$\, such that for all $X>Y$\, the statement ... holds true.\\
By $R(K(X)/\mathbb C)$\, we denote the Zariski-Riemann-variety of the function field $K(X)/\mathbb C$\,. This is a locally ringed space which is the projective limit of all  integral complete normal schemes $(Y,\mathcal O_Y)$\, with $K(Y)=K(X)$\, in the category of locally ringed spaces. Its points are the valuations of the function field $K(X)/\mathbb C$\,.
 (see \cite{Guenther}[chapter 5, pp. 8-37]).  The underlying topological space is a noetherian quasicompact space. If $A\subset K$\, is a $\mathbb C$-subalgebra, we denote by $R\Spec A$\, the set of all $\nu\in R(K/\mathbb C)$\, with $A\subset A_{\nu}$\,.\\
   For $\nu\in R(K(X)/\mathbb C)$\, we denote by
   \begin{enumerate}[1]
   \item $A_{\nu}$\, the corresponding valuation ring which is the stalk of the structure sheaf $\mathcal O_{R(K(X)/\mathbb C)}$\, at the point $\nu$\,;
   \item  $\mathfrak{m}_{\nu}$\,  the maximal ideal of $A_{\nu}$\, which is the set of all $a\in A_{\nu}$\, with $\nu(a)>0$\,;
   \item  $\Gamma_{\nu}=K(X)^*/A_{\nu}^*$\,  the value group of $A_{\nu}$\,;
   \item   $\kappa_{\nu}=\kappa(\nu)$\,  the residue field $A_{\nu}/\mathfrak{m}_{\nu}$\, of $A_{\nu}$\, which is a not necessarily finitely generated extension of the base field $\mathbb C$; 
   \item  $\dim(\nu)$\,  the dimension of the valuation $\nu$\, which is by definition equal to $\mbox{trdeg}(\kappa_{\nu}/\mathbb C)<\infty$\,; 
   \item  $r(\nu)$\,  the rank of $\nu$\, which is by definition equal to the Krull dimension of the not nessesarily noetherian ring $A_{\nu}$\,;
   \item  $rr(\nu)$\,  the rational rank of $\nu$\, which is equal to the dimension of the finite dimensional $\mathbb Q$-vector space $\Gamma_{\nu}\otimes_{\mathbb Z}\mathbb Q$\,.
   \end{enumerate}
 If $(K,\nu)$\, is a valued field and $\gamma_1,...,\gamma_n$\, is a $\mathbb Q$-basis for $\Gamma_{\nu}\otimes_{\mathbb Z}\mathbb Q$\, and $y_1,...,y_n\in K$\, satisfy $\nu(y_i)=\gamma_i,i=1,...,n$ we sloppily say that $(y_1,...,y_n)$ is a $\mathbb Q$-basis for $(K,\nu)$\,. In the same way we use the expression "$y_1,...,y_n$ is a $\mathbb Z$-basis for $(K,\nu)$".\\  
    We consider the following subspaces of $R(K/\mathbb C)$\,.
   \begin{enumerate}
   \item By $R^{k,l}(X/\mathbb C)\subset R(X/\mathbb C)$\, we denote the subset of all valuations $\nu\in R(X/\mathbb C)$\, with $r(\nu)=k, rr(\nu)=l$\,.
   \item By $R_{Ab}(X/\mathbb C)\subset R(X/\mathbb C)$\, we denote the subset of all Abhyankar-places of $K(X)/\mathbb C$\,. 
   \item By  $R^{k,l}_{Ab}(K(X)/\mathbb C)=R_{Ab}(K(X)/\mathbb C)\cap R^{k,l}(K(X)/\mathbb C)$\, we denote the subspace of all Abhyankar places of $K(X)/\mathbb C$\, of rank $k$ and rational rank $l$\,.
   \item Then, $R^{1,1}_{Ab}(K(X)/\mathbb C)$\, is the space of all discrete algebraic rank one valuations of $K(X)/\mathbb C$ which corresponds to the set of all prime divisors $E$ of the function field. We sometimes abbreviate the notation for this space simply by $R^{cl}(K(X)/\mathbb C)$\,.
   \end{enumerate}
    For $X\in \mbox{Mod}(K/\mathbb C)$\, and $\nu\in R(K/\mathbb C)$\, we denote by $c_X(\nu)$\, the center of the valuation $\nu$\, on the model $X$, which is the unique scheme point $\eta\in X$\, such that $\mathcal O_{X,\eta}\subset A_{\nu}$\,.\\
   If $X\in \mbox{Mod}(K/\mathbb C)$\, and $V\subset X$\, is a Zariski-closed subset, by $\overline{V}\subset R(K(X)/\mathbb C)$\, we denote the set of all valuations $\nu$\, such that $c_X(\nu)\in V$\,. The set $\overline{V}$\, is then Zariski-closed in $R(K(X)/\mathbb C)$\,.\\
   If $\omega\in \Lambda^{max}\Omega^1(K/\mathbb C)$\, is a rational top differential form of the function field $K/\mathbb C$\,and $X\in \mbox{Mod}(K/\mathbb C)$,\, then $K_X^{\omega}$\, denotes the divisor of zeroes and poles on $X$ of the rational section of the canonical reflexive sheaf corresponding to $\omega$\,.\\
   By a log pair $(X,D)$\, we understand a pair consisting of a normal complete variety $X$ and an $\mathbb R$-Weil divisor $D=\sum_id_iD_i, d_i\in \mathbb R$\, such that the $\mathbb R$-divisor $K_X+D$\, is $\mathbb R$-Cartier. By a log variety we mean a log pair $(X,D)$\, such that all $0< d_i\leq 1$\,. In this case $D$ is called a boundary.\\
   For a normal variety $X$, we denote the $\mathbb R$-vector space of all $\mathbb R$-Weil divisors by $\text{WDiv}_{\mathbb R}(X)$\,.\\
By a birational divisor ($b$-divisor) in the sense of Shokurov we understand a valuation function $\mathcal D: R^{cl}(K/\mathbb C)\longrightarrow \mathbb Z$\, such that for each model $X$ of $K$ the set $$\{\nu\in R^{cl}(X/\mathbb C)\mid  \dim c_X(\nu)=\dim X-1\,, \mathcal D(\nu)\neq 0\}$$
is finite.\\
In particular, $\mathcal K^{\omega}$\, denotes the canonical b-divisor associated to the rational top differential form $\omega$\, and $\mathcal A(X,D)$\, denotes the discrepancy b-divisor associated to a log pair $(X,D)$\,.\\
In awareness of the existence of higher Kaehler differential modules (see \cite{Promotion}), we denote the classical Kaehler differentials for an extension $f: A\longrightarrow B$\, of commutative rings by $\Omega^1(B/A)$\, .\\
The classical differential of a ring element $b\in B$\, we denote by $d^1b$\,.
\section{Introduction}
Basically, the study of varieties up to birational equivalence is the study of algebraic function fields. Valuation theory is one means to do this. The study of  discrete algebraic rank one valuations can be considered as the study of Cartier divisors on some birational model $X$ of $K=K(X)$\,. The study of linear series on varieties is actually equivalent to looking for rational functions having a value bounded below at some finite set of discrete algebraic rank one valuations.\\
Even if one is interested in particular birational models, for instance the so called minimal models, the process of finding these involves apriori an infinite set of birational models and the definitions of the standard classes of singularities considered in the Log Minimal Model Program involve considering divisors on all sufficiently high birational models of a given variety $X$. The adequate notion for doing this is Shokurov's notion of $b$-divisors, the standard examples of which are the canonical b-divisor of a rational top differential form and the discrepancy b-divisors of a log pair $(X,D)$\,.\\
There are two points that make it seem reasonable to extend the birational study of algebraic varieties to arbitrary valuations. First, there is a well defined object, depending only on the function field, called the Zariski-Riemann variety $R(K/\mathbb C)$ of $K/\mathbb C$, which is a locally ringed space, all of whose local rings are the valuation rings of the function field. It dominates each model of the function field. So whenever arguments involve considering infinitely many birational models of a variety $X$ at one time, one could try to implement  these arguments directly on the Riemann variety or a countable birational limit object in the category of locally ringed spaces and then use the geometric properties of $R(K/\mathbb C)$\, such as e.g quasicompactness  to carry on. For instance, a $b$-Cartier-divisor on  the function field $(K/\mathbb C)$\, is simply a Cartier divisor on the locally ringed space $R(K(X)/\mathbb C)$ (see \cite{Guenther}[chapter 5, Proposition 5.26]). \\
Secondly, if one is able to extend b-divisors, in particular the canonical b-divisor and the discrepancy b-divisors to all valuations of $K$, especially to discrete algebraic rank $n$ valuations, this opens up the possibility to run inductive arguments on the dimension of $X$.\\
We will show in this work that the classical subject of general valuation theory has  a wider scope of applications than commutative algebra and desingularization theory.\\
In section three we briefly review the theory of generalized Laurent-and power series fields attached to a totally ordered abelian group $\Gamma$\, and construct in an adhoc way topological Kaehler differentials that allow for termwise formal differentiation of generalized Laurent series. \\
The key point that we prove is that under an embedding over $\mathbb C$\,, $j: K(X)\hookrightarrow \mathbb C((\Gamma)),$\, where the rational rank of $\Gamma$\, equals the transcendence degree of $K(X)/\mathbb C$,\, the Kaehler differentials of $K(X)/\mathbb C$\, embed into the topological Kaehler differentials of $\mathbb C((\Gamma))$\, (see \prettyref{lem:L15}). That is, if $f\in K(X)$\, and $x_1,...,x_n$\, is a transcendence basis of $K(X)/\mathbb C$\, and 
$$d^1f=\sum_if_i\cdot d^1x_i,$$
 then under the embedding $j$,\, the functions $f_i$\, correspond to the formal partial derivatives in the generalized Laurent series field $\mathbb C((\Gamma))$\,.\\
We prove a valuation formula for rational top differential forms 
$$\omega\in \Lambda^{\text{max}}\Omega^1_{top}(\mathbb C((\Gamma))/\mathbb C)$$ with respect to the canonical valuation $\nu$\, on $\mathbb C((\Gamma))$\, with value group $\Gamma$\,.\\
This will serve as a preparatory step for our general valuation formula for Abhyankar places of function fields.\\
The basic idea is very elementary. Suppose, you want to valuate top differential forms $\omega=f(x)\cdot d^1x$\, on the affine line $\mathbb A^1_{\mathbb C}$\, at zero. Let $t=x^n\cdot u$\, $n\in \mathbb Z, u\in \mathbb \mathcal O^*_{\mathbb A^1,0}$\, a unit. Let us write $\omega=f_t(x)\cdot d^1t$\,. Since there is an embedding $\mathbb C(x)\hookrightarrow \mathbb C((x))$\, we may write $t$ as a Laurent series
$$t=a_n\cdot x^n+a_{n+1}\cdot x^{n+1}+...+a_N\cdot x^N+... .$$
 Then, formally 
$$d^1t=\partial^1t/\partial^1x\cdot d^1x\,\, \text{and}\,\, f_t(x)\cdot \partial^1t/\partial^1x=f(x).$$
 If we define $\nu_0(\omega)=\nu_0(f)+\nu_0(x)=\nu_0(f)+1$\, we see that 
\begin{gather*}\nu_0(f_t(x))+\nu_0(t)=\nu_0(\frac{f(x)}{\partial^1t/\partial^1x})+\nu_0(t)\\
=\nu_0(f(x))-(n-1)+n=\nu_0(f(x))+1
\end{gather*} this formula is independent of the choosen local parameter $t$ at zero. This easy observation is the basic philosophy for proving a  valuation formula for rational top differential forms at general Abhyankar places.\\
 The reason, why we restrict to Abhyankar places is that we work with embeddings into fields of generalized Laurent series and we want the rank of the topological Kaehler differentials of these fields to be equal to the transcendence degree of our function field. If $K(X)\hookrightarrow \mathbb C((x))$\, is the embedding given by an arc on a model of $K(X)$\,, which corresponds to a discrete nondivisorial valuation of rank one, the topological differentials $\Omega^1_{top}(\mathbb C((x))/\mathbb C)$\, have rank one and we cannot write for $\mbox{trdeg}(K(X)/\mathbb C)\geq 2 $\, rational top differential forms of $K(X)/\mathbb C$\, as such forms on $\mathbb C((x))$\,; the image form under the embedding of fields simply becomes zero.\\
Observe that, for divisorial valuations $\nu_E$ of function fields, there is always a complete model $X$ such that $\nu_E$ has divisorial center on $X$ and then the local ring $\mathcal \mathcal O_{X,E}$\, is the discrete valuation ring and you can valuate such top differential forms as rational sections of the canonical sheaf on $X$. If the Abhyankar place $\nu$\, is arbitrary, the corresponding valuation ring $A_{\nu}$\, is nonnoetherian and there is no model such that this becomes a local ring on a model. If you write $A_{\nu}$\, as the direct limit of its noetherian local subrings, which are local rings on different models of the function fields, you must consider rational sections of canonical sheaves on all models and then somehow pass to a limit, which, at least to myself, is not clear to exist in $\Gamma_{\nu}$\,. So straight forward generalizations of divisorial valuations of top differential forms will not be very sucsessfull.\\
In section five, we prove the general valuation formula for $\nu(\omega),$\, where $\omega$\, is a rational top differential form of a function field $K(X)/\mathbb C$\, and $\nu\in R_{Ab}(K(X)/\mathbb C)$\, is an Abhyankar place. The proof requires some commutative algebra and takes the main part of this section. We compare our valuation of rational top differential forms in case of $\nu\in R^{cl}(K(X)/\mathbb C)$\, to the classical valuation. Namely, we always have in this case 
$$\nu(\omega)=\nu_{cl}(\omega)+1\in \mathbb Z.$$
 Thus one might say that our valuation is some kind of "log valuation".\\
In section six, we generalize the classical Poincar$\acute{e}$-residue map, that tells us how to restrict a rational differential form $\omega$\, on a smooth variety $X$ to a smooth hyperplane $H$ if $\omega$\, has a simple pole along $H$\,. We construct a residue map
$$\Lambda^{n}\Omega^1(K(X)/\mathbb C)_{\nu=0}\longrightarrow \Lambda^{n-k}\Omega^1(\kappa(\nu)/\mathbb C),\,\, n=\mbox{trdeg}(K(X)/\mathbb C),\,\, k=r(\nu),$$
where $\nu$\, is an Abhyankar place of $K(X)/\mathbb C$\, with $rr(\nu)=k$\, and the first quantity in the displayed map denotes the set of all rational top differential forms $\omega$\, with $\nu(\omega)=0$\,. We show that in case that $\nu\in R^{cl}(K(X)/\mathbb C)$\,  this map specializes to the classical Poincar$\acute{e}$-residue map.\\
In section seven, we generalize the definition of log discrepancy of a log pair $(X,D)$\, to arbitrary Abhyankar places. We prove that if $(X,D)$\, is a klt (lc) log variety, then the generalized log discrepancy is positive (nonnegative) for all $\nu\in R_{Ab}(K(X)/\mathbb C)$\,.\\
In section seven, we use our generalized Poincar$\acute{e}$-residue map to define adjunction for arbitrary log canonical centers of an lc log variety $(X,D)$\,. As an application, we show  as an example how a particular case of the nowadays well known adjunction theorem (see \cite{Kollar}[chapter 4.1, Theorem 4.9, p. 158] can be proved using higher Abhyankar places.\\
We also generalize the well know monotonicity lemma (see \cite{Matsuki}[Lemma 9-1-3, pp.320-321]) for log flips to arbitrary Abhyankar places.\\
In a final outlook we propose some future applications to the log minimal model program of our theory.
 \section{Review of valuation theory}
  We will use in this paper only well known facts about valuation theory of function fields as can be found in \cite{ZaSam}[Volume II, chapter VI], which we will review for the convenience of the reader. For a good readable account of the basics of valuation theory, see \cite{Val}. We will only be concerned with Krull valuations of a field $K$, or, more generally of the relative situation of a field extension $K/k$\,, $k$ being a base field of characteristic zero.\\
   We recall that a Krull valuation is a homomorphism of abelian groups 
   $$\nu:K^*=K\backslash \{0\}\longrightarrow \Gamma,$$ $K^*$\, with the multiplication and $\Gamma$\, being totally ordered, such that for all $f,g\in K$\, one has $$\nu(f+g)\geq \rm{min}(\nu(f),\nu(g)),$$ where $\nu(0)$\, is considered to be $\infty$\,.\\
    The set $A_{\nu}\subset K$\, consisting of all $f\in K$\, such that $\nu(f)\geq 0$\, (including $f=0$\,) is easily seen to be a ring, not necessarily noetherian, called the valuation ring of $\nu$\, and the subset $\mathfrak{m}_{\nu}\subset A_{\nu}$\, consisting of elements $f$\, with $\nu(f)>0$\, is the maximal ideal.
     The field $k_{\nu}:=A_{\nu}/\mathfrak{m}_{\nu}$\, is the usual residue field and the image of the homomorphism $\nu$\, in $\Gamma$\, is denoted by $\Gamma_{\nu}$\, and is called the value group of $\nu$\,.\\
      Two valuations are said to be equivalent, if the associated valuation rings are equal.\\
    Valuation rings inside a field $K$\, can be characterized as being maximal subrings of $K$ with respect to the partial order of domination between local rings, i.e., the relation $(A,m)<(B,n)$\, iff there is a local homomorphism $(A,m)\longrightarrow (B,n)$\, being the identity on generic points (see \cite{Val}[chapter 1, p.1]) .\\
     From this maximality property a valuation ring $A_{\nu}$\, has the property that for each $f\in K^*$\, at least one of the two elements $f$ and $f^{-1}$\, is in $A_{\nu}$\,. It follows, that the abelian group $K^*/A_{\nu}^*$\, is totally ordered with respect to the relation $$[f]<[g]\,\, \mbox{if}\,\,[gf^{-1}]\in A_{\nu}.$$
     There is always  a canonical isomorphism of totally ordered abelian groups $$(*):K^*/A_{\nu}^*\cong \Gamma_{\nu}\,\, \mbox{sending}\,\, [f]\,\, \mbox{to}\,\, \nu(f).$$
      If $L\subset K$\, is any subfield, we may restrict $\nu$\, to $L$ and obtain a valuation of $L$, denoted by $\nu\mid_L$\,.\\
      It is known (see \cite{ZaSam}[Vol.I, chapter 7, Theorem 11] that each valuation $\nu$ of a field $K$ can be extended to a valuation $\mu$\, of any field extension $K\subset L$\,, algebraic or not, such that $\mu\mid_K=\nu$\,.\\
A valuation $\nu$\, of $K$\, will be called  over a subfield $k\subset K$\,, if $\nu\mid_k$\, is the trivial valuation, i.e. $\nu(c)=0$\, for all $c\in k^*$\,. $k$ will usually be the fixed base field.  If we assume the restriction to be trivial, we will simply write $\nu\in R(K/k)$\,.\\ 
We denote the dimension of the $\mathbb Q$-vector space $\Gamma_{\nu}\otimes_{\mathbb Z}\mathbb Q$\, by $rr(\nu)$\, and call it the rational rank of the valuation $\nu$. The Krull dimension of the valuation ring $A_{\nu}$\, will be called the rank of $\nu$\,, denoted by $r(\nu)$\,. We always have an inequality $r(\nu)\leq rr(\nu)$\,. The transcendence degree $\mbox{trdeg}(k_{\nu}/k)$\, will be called the dimension of $\nu$\, denoted by $\dim(\nu)$\,. A valuation with $\dim(\nu)=0$\, is also called rational.\\
A segment $\Delta$\, of a totally ordered abelian group $\Gamma$\, is a symmetric subset such that 
$$ \alpha,\beta\in \Delta, \alpha <\gamma< \beta \vee \beta<\gamma<\alpha \Rightarrow \gamma\in \Delta.$$
By \cite{Val}[chapter 1.2, Proposition 1.6, p.4] there is a one-to-one inclusion reversion correspondence between segments of $\Gamma_{\nu}$\, and ideals of $A_{\nu}$\,. If $\mathfrak{a}\subset A_{\nu}$\, is an ideal, then the associated segment $\Delta(\mathfrak{a})$\, is defined by 
$$\Delta(\mathfrak{a}):=\{\gamma\in \Gamma\mid \gamma<\nu(a)\,\,\forall a\in \mathfrak{a}\vee \gamma >-\nu(a)\forall a\in \mathfrak{a}\}.$$
If $\Delta\subset \Gamma$\, is a segment, then the associated ideal $\mathfrak{a}(\Delta)$\, is defined by 
$$\mathfrak{a}(\Delta):=\{a\in A_{\nu}\mid \nu(a)> \delta\,\,\forall \delta\in \Delta\}.$$
We recall that the order rank of a totally ordered abelian group $\Gamma$\, is the length of a maximal chain of convex subgroups
$$0\subsetneq \Gamma_1\subsetneq \Gamma_2\subsetneq ...\subsetneq \Gamma_n=\Gamma.$$
By \cite{Val}[chapter 1.2, Corollary, Theorem 1.7,p.5], the Krull dimension of $A_{\nu}$\, equals the order rank of $\Gamma_{\nu}.$\\ 
If $K$ happens to be of finite transcendence degree over $k$, all the above defined numbers are finite and there is always the fundamental Abhyankar inequality $$rr(\nu)+\dim(\nu)\leq \rm{trdeg}(K/k).$$ A place $\nu$\, for which equality holds is called an Abhyankar place. If $K$ happens to be finitely generated over $k$, then for an Abhyankar place $\nu$\,, the value group $\Gamma_{\nu}$\, is known to be a finite (free) $\mathbb Z$-module and the residue field $k_{\nu}:=A_{\nu}/\mathfrak{m}_{\nu}$\, is known to be finitely generated over $k$. Both properties may fail in general for arbitrary places of a function field $K/k$\, and such places always exist if $\rm{trdeg}(K/k)\geq 2$\,.\\
It is known that every totally ordered abelian group $\Gamma$\, with finite rational rank can be order imbedded into some $\mathbb R^n$\,with the lexicographic product order of the usual order on the factors $\mathbb R$\,. We always have $n\geq r(\Gamma)$\, and it is known that the smallest $n\in \mathbb N$\, for which there exists such an embedding is equal to the  rank, i.e. if $\Gamma=\Gamma_{\nu}$, then $n$ equals the Krull dimension of $A_{\nu}$\,, hence the rank of $\nu$\,. \\
We will consider in this paper exclusively fields $K$ of finite transcendence degree over a base field $k$\,. We denote by $R^{n,l}(K/k)$\, the subset of all valuations $\nu$\, of $R(K/k)$\, with $r(\nu)=n$\, and $rr(\nu)=l$\,. We denote by $R_{Ab}^{n,l}(K/k)\subset R^{n,l}(K/k)$\, the subset of all such Abhyankar places. \\
The space of all algebraic discrete rank one valuations ($R^{1,1}_{Ab}(K/k)$) we will simply denote by $R^{cl}(K/k)$\,.
\\
If $A_{\nu}\subset K$\, is a valuation ring and $A_{\nu_1}\subset k_{\nu}$\, is another valuation ring, we can form the composed ring $$A_{\nu\circ \nu_1}:=p^{-1}(A_{\nu_1})$$
 where $p:A_{\nu}\longrightarrow k_{\nu}$\, is the canonical residue map. As the notation suggests, $A_{\nu\circ \nu_1}$\, is again a valuation ring of $K$ and the corresponding valuation is called the composition of $\nu$\, with $\nu_1$\,, denoted by $\nu\circ \nu_1$\,. Because $$A_{\nu\circ \nu_1}\backslash\mathfrak{m}_{\nu\circ \nu_1}\subset A_{\nu}\backslash\mathfrak{m}_{\nu},$$
  there is a canonical surjection $$K^*/A_{\nu\circ \nu_1}^*\longrightarrow K^*/A_{\nu}^*,$$
   whose kernel is canonically isomorphic to $p^{-1}(k_{\nu_1}^*)/p^{-1}(A_{\nu_1}^*)$\, which is again canonically isomorphic to $k^*_{\nu_1}/A_{\nu_1}^*$\,. \\
   Using the canonical isomorphisms $(*)$\,, we get a canonical exact sequence of ordered abelian groups 
   $$(**)\,\,0\longrightarrow \Gamma_{\nu_1}\longrightarrow \Gamma_{\nu\circ \nu_1}\longrightarrow \Gamma_{\nu}\longrightarrow 0,$$ such that $\Gamma_{\nu_1}\longrightarrow \Gamma_{\nu\circ \nu_1}$\, is an order inclusion and $\Gamma_{\nu_1}$\, is a  convex subgroup of $\Gamma_{\nu\circ \nu_1}$\, in the sense that if 
   $$\alpha,\beta\in \Gamma_{\nu_1}, \gamma\in \Gamma_{\nu\circ \nu_1}\,\, \mbox{and}\,\, \alpha<\gamma<\beta,\,\, \mbox{then}\,\, \mbox{also}\,\, \gamma\in \Gamma_{\nu_1}.$$ The order on $\Gamma_{\nu}$\, is isomorphic under these isomorphisms to the induced quotient order on $\Gamma_{\nu\circ \nu_1}/\Gamma_{\nu_1}$\,.\\ Every totally ordered abelian group must necessarily be a torsion free $\mathbb Z$-module, hence every such sequence possesses a (noncanonical) order preserving splitting $$\Gamma_{\nu}\longrightarrow \Gamma_{\nu\circ \nu_1}.$$
    If $\Gamma_{\nu}$\, happens to possess a finite $\mathbb Z$-basis $\gamma_i, i=1,...,n$\,, then any choice of elements $t_1,...,t_n$\, with $\nu(t_i)=\gamma_i$\,, which we sloppily call a $\mathbb Z$-basis of $\nu$\,, induces a splitting of $(**)$\,. Namely, writing $\gamma\in \Gamma_{\nu}$\, as $[f]\in K^*/A_{\nu}^*$\, under the above isomorphism $(*)$\,, we have $$\nu(f)=\sum_in_i\nu(t_i),\,\,\text{for some}\,\, n_i\in \mathbb Z.$$
We first send $[f]$\, to 
$$\nu_1(p(f\cdot\prod t_i^{-n_i}))\in \Gamma_{\nu_1}\,\, \mbox{(observe}\,\, f\cdot \prod_it_i^{-n_i}\in A_{\nu}^*).$$ This is indeed a group homomorphism. We define the isomorphism $$\phi(\underline{t}):\Gamma_{\nu\circ \nu_1}\cong \Gamma_{\nu}\oplus \Gamma_{\nu_1}$$ by sending $[f]$\, to $(\nu(f), \nu_1(p(f\cdot \prod_it_i^{-n_i})))$\,.\\
For later usage, we prove for the lack of reference the following
\begin{lemma}\mylabel{lem:L100} Let $\nu\in R(\overline{K(X)}/\mathbb C)$\, be an arbitrary nonarchimedian valuation. Then $\Gamma_{\nu}$\, has naturally the structure of a $\mathbb Q$-vector space.
\end{lemma}
\begin{proof} Let $t_1,...,t_k\in K(X)$\, be a $\mathbb Q$-basis for $\Gamma_{\nu}$\,. For each $n\in \mathbb N$\,, let $L_n=K(X)(t_1^{\frac{1}{n}}, ...t_k^{\frac{1}{n}})$\,. Let $\nu_n$\, be the restriction  of $\nu$\, to $L_n$.\, Obviously, we must have $\nu(t_i)=n\cdot \nu_n(t_i^{\frac{1}{n}}), i=1,...,k$\,. Let 
$$\gamma=\sum_{i=1}^k\frac{p_i}{q_i}\cdot \nu(t_i)\in \Gamma_{\nu}\otimes_{\mathbb Z}\mathbb Q,\,\, p_i,q_i\in \mathbb Z$$ be arbitrary. Put $N=q_1\cdot ...\cdot q_k$\, and $t=\prod_{i=1}^kt_i^{\frac{p_i}{q_i}}\in L_N$\,. Then 
$$\nu_N(t)=\sum_{i=1}^kp_i\cdot \nu_N(t_i^{\frac{1}{q_i}})=\sum_{i=1}^k\frac{p_i}{q_i}\cdot \nu(t_i)=\gamma.$$ Thus $\gamma\in \Gamma_{\nu_N}\subseteq \Gamma_{\nu}$\,. Since $\gamma$\, was arbitrary, we conclude $\Gamma_{\nu}\cong \Gamma_{\nu}\otimes_{\mathbb Z}\mathbb Q$\, and the assertion follows.
\end{proof}
For the lack of reference, we prove the following
\begin{lemma}
\mylabel{lem:L1}
Let $K/\mathbb C$\, be a function field and $\overline{K}$\, its algebraic closure. Let $\nu$ an arbitrary Krull valuation of $K$, being trivial on $\mathbb C$\, and $\mu$ be an extension of $\nu$\, to $\overline{K}$\,. Then there exists a $\mathbb Q$ basis $t_1,...,t_k, k=rr(\nu)$ for $\nu$ plus a compatible system of roots $t_i^{\frac{1}{n}}\in \overline{K},\,\, \forall i=1,...,k\,\,\forall n\in \mathbb N$\,.
\end{lemma}
\begin{proof} We consider $\overline{K}$ as a field extension over $\mathbb R$\, which has again finite transcendence degree. Let  $R\subset \overline{K}$\, with $[\overline{K}:R]=2$\,be the  real closed subfield. We may choose a $\mathbb Q$-basis for $\nu$\, in $R\cap K$, since the restriction of $\nu$\, to $R\cap K$ has the same rational rank. Moreover, we may choose $t_1,...,t_k$\, to be positive, i.e.  $t_i\geq 0$\,in $R$. This is possible since either $t$ or $-t$ is positive and the sign does not effect the value of $t$. Since $R$ is real closed and all $t_i$ are positive, there exist in $R$ for each $n$ an $n$th root of $t_i$\,.\\ If we moreover require that $t_i^{\frac{1}{n}}>0$\,, then these roots are uniquely determined because the quotient of two such roots must be a real positive root of unity, hence must be $1$.\\
 By the same reasoning, we must have $(t_i^{\frac{1}{nm}})^m=t_i^{\frac{1}{n}}$\, and we can define for each $r\in \mathbb Q$ unambigously $t_i^{r}$\, and we have $t_i^r\cdot t_i^s=t_i^{r+s}$\,.
\end{proof}
We will make essential use of the following
\begin{theorem} \mylabel{thm:T2}
 (Embedding Theorem of Kaplansky)\\
 Let $(K,\nu)/\mathbb C$\, be an arbitrary valued field  such that $\nu$  has dimension zero, $\dim(\nu)=0,$\, i.e. $\kappa(\nu)= A_{\nu}/\mathfrak{m}_{\nu}=\mathbb C$\,. Let $\Gamma_{\nu}$\, be the value group. Then there is an embedding of fields over $\mathbb C$: $\phi: K\hookrightarrow \mathbb C((\Gamma_{\nu}))$\,  into the field of generalized Laurent series with exponents in $\Gamma_{\nu}$\, (see the next section) such that the canonical valuation of $\mathbb C((\Gamma_{\nu}))$\, restricted via $\phi$\, to $K$\, gives back the valuation $\nu$\,.
 \end{theorem}
 \begin{proof}   The assumption of \cite{Kaplansky}[ §3, Theorem 5, p.312 and §4, Theorem 8, p.318] are all satisfied. That there is an embedding into a field of Laurent series without certain factor sets  follows from the existence of a compatible set of roots of a $\mathbb Q$-basis for $\nu$\, proven in \prettyref{lem:L1}. For more details see \cite{Kaplansky}[ Paragraph 4].
 \end{proof}
 We will only need this theorem in the case where $K=K(X)$ is an algebraic function field over $\mathbb C$ or $K=\overline{K(X)}$\, is its algebraic closure and where $\nu$\, is an Abhyankar place of dimension zero.
 \paragraph{Examples of zero dimensional valuations}
 \begin{example}\mylabel{ex:E60}(Monomial valuations)\\
Suppose $(A,\mathfrak{m},k)$\, is a regular local $k$-algebra essentially of finite type and the residue field $k$ is algebraically closed ( and of characteristic zero). Then $(A,\mathfrak{m},k)$\, is a smooth local $k$-algebra and there are elements  $r_1,...,r_n\in A$\, such that the induced morphism 
$$\phi:k[x_1,...,x_n]_{(x_1,...,x_n)}\longrightarrow A, x_i\mapsto r_i, i=1,...,n$$ is etale (see \cite{SGA}[SGA I, Expose II, Definition 1.1, p.29]). The parameters $r_1,r_2,...,r_n$\,  then necessarily form  a regular sequence generating the maximal ideal $\mathfrak{m}$\,. Let $\widehat{A}$\, be the completion of $A$ along the maximal ideal. As $\phi$\, is etale, we have 
$$\widehat{A}\cong \widehat{k[x_1,x_2,...x_n]_{(x_1,x_2,...,x_n)}}\cong k[[x_1,x_2,...,x_n]].$$
We have thus an inclusion $A\hookrightarrow k[[x_1,...,x_n]]$\, where $r_i$\, maps to $x_i$\, for $i=1,...,n$\,.  Now choose positive real numbers $c_i>0$\, and if 
$$p(r_1,...,r_n)=\sum_I\underline{x}^I$$ is a homogenous polynomial define 
$$\nu(p)=\mbox{min}_I(\prod_{j=1}^ni_j\cdot c_j)>0\,\,\text{where}\,\, I=(i_1,i_2,...,i_n).$$
 Now if $p$\, is a power series in $\widehat{A}$\,, define $$\nu(p)=\mbox{inf}_n\nu(p_n),$$
 where $p_n$\, is the homogeneous part of $p$ of degree $n$. If $c=\mbox{min}_{i=1}^n c_i$\,, then \\
$\nu(p_n)\geq nc$\,, which tends to infinity. Thus the above infemum is for every power series actually a minimum. It follows that  $\nu(p)\in \mathbb Z(c_1,...,c_n)\subset \mathbb R$\,. It is clear that $\nu(p+q)\geq \mbox{min}(\nu(p),\nu(q))$\, as power series are added term by term. Similarly, there is a monomial $m_p$\, occurring in the power series of $p$ and a a monomial $m_q$\, in the power series of $q$ such that $\nu(p)=\nu(m_p)$\, and $\nu(q)=\nu(m_q)$\,. Then the monomial $m_pm_q$\, occurs in the power series of $pq$\, and is there the monomial with minimal value. Thus $$\nu(pq)=\nu(m_pm_q)=\nu(m_p)+\nu(m_q)=\nu(p)+\nu(q).$$
 We  can therefore extend $\nu$\, to $k((x_1,...,x_n))$\, by $\nu(\frac{p}{q})=\nu(p)-\nu(q)$\, to get a valuation of $k((x_1,...,x_n))$\,. Restricting to $K(A)\subset k((x_1,...x_n))$\, we get a valuation of $K(A)$\, with center above $A$\,. It has finitely generated value group. If $c_1,...,c_n$\, are linearly independent over $\mathbb Q$\,  we get a zero dimensional Abhyankar place of rank one and rational rank $n$.
\end{example}
 \begin{example}\mylabel{ex:E61} (Flag valuations)
These are the nonnoetherian valuations that are most commonly known to nonspecialists.\\
Let $K/k$\, be a function field and $X$ be a complete model of $K$. Consider a flag of subvarieties
$$X=D_0\supsetneq D_1\supsetneq ...\supsetneq D_n\supsetneq 0$$
such that
\begin{enumerate}[i]
\item each $D_i, i=0,...,n$\, is integral.
\item The local ring $\mathcal O_{D_{i},\eta_{D_{i+1}}}$\, is one dimensional and regular.
\end{enumerate}
It follows that $\mathcal O_{D_i,\eta_{D_{i+1}}}$\, is a discrete algebraic rank one valuation  ring with corresponding valuation $\nu_i$\, in the function field $k(D_i)$ for all $i=1,....,n$\,.\,
We can form the successive composition $\nu:=\nu_0\circ \nu_1\circ ...\circ \nu_{n-1}$\, and get a valuation of the function field $K(X)$\,. Since under composition the rational rank and the rank are additive  (see \cite{Val}[chapter 1.2, Corollary to Proposition 1.11, p.9]), we have $rr(\nu)=n$\, and $r(\nu)=n$\, and we have got an Abhyankar place in $R^{n,n}_{Ab}(K(X\ \mathbb C))$\,. The value group $\Gamma_{\nu}$\, is isomorphic to $\mathbb Z^n$\, but in a noncanonical way! Such an isomorphism depends on the choices of local parameters for the discrete algebraic valuation rings in the function fields $k(D_i)$\,.
\end{example}
 \begin{example}\mylabel{ex:E62}(Mixed valuations) In example (2) we might as well take an incomplete flag 
$$D_0=X\supsetneq D_1\supsetneq ...\supsetneq D_k$$
 in order to construct an Abhyankar place $\nu\in R^{k,k}_{Ab}(K(X)/\mathbb C)$\,. It has dimension $\dim(\nu)=\mbox{trdeg}(\kappa(\nu)/k)=n-k$\,. We can take an arbitrary monomial valuation $\nu_0$\, of the function field $\kappa(\nu)$\, (example (1)). and form the composed valuation $\mu:=\nu\circ \nu_0$\,.\\ By the additivity of rank and rational rank (see \cite{Val}[chapter 1.2, Corollary to Proposition 1.11, p.9]), we then get an Abhyankar place in $R^{k+1,n}_{Ab}(K(X)/\mathbb C)$\,.
\end{example}
\section{Fields of generalized Laurent series}
 Recall that for an arbitrary totally ordered abelian group $\Gamma$\, and any base field $k$, one may construct the field of generalized Laurent series $k((\Gamma))$\, with respect to $\Gamma$\,. Its elements are formal possibly infinite linear combinations $f=\sum_{\gamma\in \Gamma}a_{\gamma}z^{\gamma}$\, with $a_{\gamma}\in k$\, such that the support of $f$, being the subset $S_f\subset \Gamma$\, of all $\gamma\in \Gamma$\, such that $a_{\gamma}\neq 0$\,, is a well ordered set with respect to the ordering on $S_f$\, induced by the given one on $\Gamma$\,. This is equivalent to the fact that the totally ordered set $S_f$\, satisfies the descending chain condition or is free of lower culmination points with respect to the order topology on $\Gamma$\,. \\
  In particular, for each $f\in k((\Gamma))$, there exists a smallest $\gamma_0\in \Gamma$\, for which $a_{\gamma}\neq 0$\,. Addition is defined component wise and multiplication is the usual multiplication of power series under the multiplication rule $z^{\gamma}\cdot z^{\delta}=z^{\gamma+\delta}$\, for $\gamma,\delta\in \Gamma$\,.\\
   The condition on the sets $S_f$\, to be well ordered guarantees that the expressions defining the coefficients of a product, are finite sums of elements of $k$\,. \\
   If $\Gamma =\mathbb Z$\,, then $k((\Gamma))\cong k((x))$\,, the usual field of Laurent series.\\
 It is straight forward to check that a power series $f=\sum_{\gamma\geq 0}a_{\gamma}z^{\gamma}$\, with $a_0\neq 0$\, is invertible and therefrom one deduces that $k((\Gamma))$\, is indeed a field. There is a natural valuation $\nu$\, on $k((\Gamma))$\, with value group $\Gamma$\, defined by $$\nu(f)=\min_{a_{\gamma}\neq 0,\gamma\in \Gamma}\gamma .$$ The subring 
    $$k[[\Gamma]]:=\{f\in k((\Gamma))\mid a_{\gamma}=0\ \mbox{for}\ \gamma<0\}$$ is the corresponding valuation ring which is called the ring of generalized power series and 
    $$\mathfrak{m}=\{f\in k[[\Gamma]]\mid a_0=0\}$$ the valuation ideal.\\
 Now suppose that we are given an order preserving exact sequence of totally ordered abelian groups $$0\longrightarrow \Gamma'\longrightarrow \Gamma\longrightarrow \Gamma''\longrightarrow 0,$$
 such that $\Gamma'$\, is a convex subgroup of $\Gamma$\, and $\Gamma''$\, gets the induced quotient total order. We consider the valuation ring $k[[\Gamma]]$ with canonical valuation $\mu$\, such that $\Gamma=\Gamma_{\mu}$. By general valuation theory, there is a prime ideal $p_{\Gamma'}$\, in $k[[\Gamma]]$\, corresponding to the convex subgroup $\Gamma'$\,. It consists of all $f\in k[[\Gamma]]$\, such that $\nu(f)\notin \Gamma'$\,, i.e., that the smallest nonzero coefficient $a_{\gamma_0}$\, of $f$  has $\gamma_0\notin \Gamma'$\,.\\ It is  generated as an ideal in $k[[\Gamma]]$\, by the set\, $\{z^{\gamma}\mid  \gamma\notin \Gamma'\}$\,.\\
  Again by general valuation theory, the residue ring $k[[\Gamma]]/p_{\Gamma'}$\, is a valuation ring corresponding to a valuation $\nu_1$\, with value group $\Gamma'=\Gamma_{\nu_1}$\,. \\
  \begin{proposition}\mylabel{prop:P60}
  With  notation as above, we have the following statements.
  \begin{enumerate}[1]
  \item $k[[\Gamma]]/p_{\Gamma'}\cong k[[\Gamma']].$\,
  \item $k[[\Gamma]]_{p_{\Gamma'}}\cong k((\Gamma_{\geq 0}\cup \Gamma'))$\, is a valuation ring corresponding to a valuation $\nu$\, with value group $\Gamma''=\Gamma_{\nu}$\,.
  \end{enumerate}
  \end{proposition}
  \begin{proof}
  \begin{enumerate}[1]
  \item
   Indeed, writing $$f=\sum_{\gamma\in \Gamma}a_{\gamma}z^{\gamma}\in k[[\Gamma]]$$ and deleting all terms $a_{\gamma}z^{\gamma}$\, with $\gamma\notin \Gamma'$\, we get a power series $$p(f)=\sum_{\gamma\in \Gamma'}a_{\gamma}z^{\gamma}\in k[[\Gamma']].$$
    One checks that this is a ring homomorphism as $\Gamma'$\, is convex and the kernel of $p$ is the ideal generated by $\{z^{\gamma}\mid \gamma\notin \Gamma'\}$\,, which is precisely $p_{\Gamma'}$ .\\
    \item  Furthermore, by general valuation theory, the localization $k[[\Gamma]]_{p_{\Gamma'}}$\, is again a valuation ring with value group $\Gamma''$\,. The localization is with respect to the multiplicatively closed subset of all power series having smallest nonzero coefficient in $\Gamma'$\, and as every power series with smallest nonzero coefficient $\gamma=0$\, is invertible, this is the subring $k((\Gamma_{\geq 0}\cup \Gamma'))$\, consisting of all Laurent series with smallest coefficient in the subset $\Gamma_{\geq 0}\cup \Gamma'$\,, $\Gamma_{\geq 0}$\, denoting the semigroup of nonnegative elements of $\Gamma$\,.\\
      The valuation ideal $\mathfrak{m}_{\Gamma''}$\, corresponds to all Laurant series with lowest coefficient in $\Gamma_{\geq 0}\backslash \Gamma'_{\geq 0}$\,. Then as $\Gamma_{\nu_{\Gamma''}}=k((\Gamma))^*/A_{\nu_{\Gamma'}}^*$\,, it is clear that this value group is isomorphic to $\Gamma''$\,. The value $$\nu_{\gamma''}(f=\sum_{\gamma\in \Gamma}a_{\gamma}z^{\gamma})$$ is then  equal to $p(\nu(f))=p(\gamma_0)$\,.\\
\end{enumerate}
 The situation is described in the following commutative diagrams.
 
 \[\begin{array}{ccc}
 A_{\mu} & \hookrightarrow & k[[\Gamma_{\mu}]] \\
 \Big\downarrow & &\Big\downarrow  \\   
 k(\mu) & \hookrightarrow & k[[\Gamma_{\nu_1}]]
 
 \end{array}\]
 \\
 \[\begin{array}{ccc}
 A_{\mu} & \hookrightarrow & k[[\Gamma_{\mu}]]\\
 \Big\downarrow & & \Big \downarrow \\
 A_{\nu} & \hookrightarrow & k[[\Gamma_{\mu}]]_{p_{\Gamma_{\nu_1}}}.
 \end{array}\]
 \end{proof}
 \subsection{K\"ahlerdifferentials for generalized power series rings}
Let $\Gamma$ be a totally ordered abelian group and as above $\mathbb C((\Gamma))$ be the generalized Laurent series field consisting of all formal Laurent series $\sum_{\gamma\in  \Gamma}a_{\gamma}z^{\gamma}$\, such that the set of all $\gamma$\, such that $a_{\gamma}\neq 0$\, is  well ordered with respect to the ordering of $\Gamma$\,. We want to define in a convenient way the topological K\"ahler differentials $\Omega^1_{top}(\mathbb C[[\Gamma]]/\mathbb C)$\,  with respect to the natural valuation topology of $\mathbb C[[\Gamma]]$\,. There are several ways to do this. One way is to consider the valuation topology as a generalized preadic topology (see \cite{Promotion}[chapter 5.2, pp. 13-39]) and then use the general theory developed  there in order to define the complete topological K\"ahler differentials (see \cite{Promotion}[chapter 6.3, pp.86-96]). We here prefer to use an adhoc definition in order to get what we want, namely that this module is finitely generated over $\mathbb C[[\Gamma]]$\, and that we are allowed to formally differentiate generalized power series. If we fix an isomorphism $\Gamma\cong \mathbb Z^n$\, with basis $e_1,...,e_n$, $\mathbb Z^n$\, with the induced order, we then can write each element $c\in \mathbb C[[\Gamma]]$\, as 
$$c=\sum_{\underline{m}\in \mathbb N_0^n}a_{\underline{m}}\cdot \underline{x}^{\underline{m}},$$ where we use multiindex notation $\underline{x}^{\underline{m}}=x_1^{m_1}\cdot x_2^{m_2}\cdot ...\cdot x_n^{m_n}$. In the usual power series expansion $\sum_{\gamma\in \Gamma}a_{\gamma}z^{\gamma}$\, write $\gamma =\sum_in_i\cdot e_i$\, and put $z^{e_i}=x_i$. Then, we want to get the usual rules for working with differentials
$$d^1\sum_{\underline{m}\in \mathbb N^n}a_{\underline{m}}\cdot \underline{x}^{\underline{m}}=\sum_{i=1}^n\partial^1/\partial x_i(c)\cdot d^1x_i,$$ where we use formal differentiation. Observe that we do not get the classical power series ring in $n$ variables. Note that for each $i$,
$$\partial^1/\partial^1x_i(c)=\sum_{\underline{m}\in \mathbb Z^n}m_i\cdot a_{\underline{m}}\cdot \underline{x}^{\underline{m}-1^i}$$
is again in $\mathbb C((\Gamma))$\,, where the notation $\underline{m}-1^i$\, shall indicate that we have subtracted one at the $i^{th}$\, place. For, if $S=\{m^j_i\cdot a_{\underline{m^j}}\}$\, is a subset of indices that are nonzero, corresponding to the set of exponents $\{\underline{m^j}-1^i\}$\,, then the coefficients in $c$ in front of $\underline{x}^{\underline{m^j}}, j\in S$\, are nonzero, and thus there is an index $j_0$\, such that $\underline{m^{j_0}}$\, is the least element. But then, $\underline{m^{j_0}}-1^i$\, is the smallest element of $S$. Thus, the set of all $\gamma$\, such that the corresponding coefficient in the power (Laurent-) series of $\partial^1/\partial^1x_i(c)$\, is nonzero is well ordered and  the formally differentiated power (Laurent-) series belongs to $\mathbb C((\Gamma))$\,.\\
\begin{definition}\mylabel{def:D10} Let $\Gamma$\, be a totally ordered abelian group together with a fixed isomorphism 
$$j:\Gamma\cong \oplus_{i=1}^n\mathbb K\cdot e_i,$$
with either $\mathbb K=\mathbb Z$\, or $\mathbb K=\mathbb Q$\,.
 We consider the $\mathbb C$-algebra $\mathbb C[[\Gamma]]$\, of generalized power series with exponents in $\Gamma$\, or its quotient field, the field of generalized Laurent series, $\mathbb C((\Gamma))$\,. Putting $z_i=z^{e_i}$\,, under the above isomorphism, we can write each generalized power- (Laurent)-series $f(z)=\sum_{\gamma\in \Gamma}a_{\gamma}z^{\gamma}$\, as 
$$f(z_1,...,z_n)=\sum_{\underline{n}\in \mathbb K^n}a_{\underline{n}}\underline{z}^{\underline{n}},$$ where we use multiindex notation.Then we  can formally differentiate a generalized power series.\\ We put 
\begin{gather*}\Omega^1_{top}(\mathbb C[[\Gamma]]/\mathbb C) = C[[\Gamma]]\otimes_{\mathbb C}\mathbb C[[\Gamma]]/M_{FD,j},\,\,\mbox{resp},\\
 \Omega^1_{top}(\mathbb C((\Gamma))/\mathbb C) =\mathbb C((\Gamma))\otimes_{\mathbb C}\mathbb C((\Gamma))/M'_{FD,j},
 \end{gather*}
where we regard $\Omega^1_{top}$\, as a $\mathbb C[[\Gamma]]$-module (resp. as a $\mathbb C((\Gamma))$-vector space) via the first tensor factor and where $M_{FD}$\,, resp. $M'_{FD}$\, is the $\mathbb C[[\Gamma]]$-submodule, (resp. $\mathbb C((\Gamma))$-subvector space) generated by the expressions 
$$\langle 1\otimes f-(\sum_{i=1}^n\partial^1/\partial^1z_i(f)\otimes z_i)\mid f\in \mathbb C((\Gamma))\rangle.$$
We put $d^1f:=\overline{1\otimes f} $,\,where the bars denote residues module $M_{FD,j}$, resp. $M'_{FD,j}$\, such that the above  expressions just read as the usual  transformation rules for usual K\"ahler differentials of formal or convergent power series.
\end{definition}
\begin{remark} Of course, the basic example we have in mind is where either $\Gamma=\Gamma_{\nu}$\, with $\nu\in R_{Ab}(K(X)/\mathbb C)$\, is an Abhyankar place of dimension zero of an algebraic function field $K(X)/\mathbb C$\, or $\Gamma=\Gamma_{\overline{\nu}}$\, where $\overline{\nu}\in R_{Ab}(\overline{K(X)}/\mathbb C)$\, is an Abhyankar place of dimension zero on the algebraic closure of an algebraic function field over $\mathbb C$\,. In the first case, we have $\Gamma_{\nu}\cong \mathbb Z^n$\, and in the second case we have $\Gamma_{\overline{\nu}}\cong \mathbb Q^n$\,. Since we want to work with $\mathbb Q$-basis for $\Gamma_{\nu}$\, in $K(X)$\,, we pass to the algebraic closure $\overline{K(X)}$\, and an extension $\overline{\nu}$\, of $\nu$\, to $\overline{K(X)}$\, and the $\mathbb Q$-basis for $\nu$\, then gives an isomorphism $\Gamma_{\overline{\nu}}\cong \mathbb Q^n$\,.
\end{remark}
We prove the following easy
\begin{lemma}\mylabel{lem:L2050} With notations as above, the module $M_{FD,j}$\,, resp. the vector space $M'_{FD,j}$\, contains the submodule (subvector space) $M_L, M'_L$\, generated by  all "Leibnitz-relations" 
 $$\langle 1\otimes (f\cdot g)-f\otimes g-g\otimes f\mid f,g\in \mathbb C((\Gamma))\rangle,f,g\in \mathbb C[[\Gamma]],\,\, \mbox{resp.},\,\, \in \mathbb C((\Gamma))$$  such that we have  canonical surjections
 \begin{gather*}\Omega^1(\mathbb C[[\Gamma]]/\mathbb C)\twoheadrightarrow \Omega^1_{top}(\mathbb C[[\Gamma]]/\mathbb C)\,\,\mbox{and}\\
 \Omega^1(\mathbb C((\Gamma))/\mathbb C)\twoheadrightarrow \Omega^1_{top}(\mathbb C((\Gamma))/\mathbb C).
 \end{gather*}
 \end{lemma}
 \begin{proof} We treat the case of power series, the case of Laurent series is identical.\\
 We have 
 \begin{gather*} 1\otimes (f\cdot g)-\sum_{i=1}^n\partial^1/\partial^1z_i(f\cdot g)\otimes z_i\in M_{FD,j};\,\, (1)\\
 f\otimes g-\sum_{i=1}^nf\cdot \partial^1/\partial^1z_i(g)\otimes z_i\in M_{FD,j};\,\, (2)\\
 g\otimes f-\sum_{i=1}^ng\cdot \partial^1/\partial^1z_i(f)\otimes z_i\in M_{FD,j}\,\, (3).
 \end{gather*}
 If for each $i=1,...,n$\, we view $f,g$ as formal power series in the vairable $x_i$\,, then the usual rule for differentiation of a product gives
 $$\partial^1/\partial z_i(f\cdot g)=f\cdot \partial^1/\partial^1z_i(g)+g\cdot \partial^1/\partial^1z_i(f), i=1,...,n.\,\, (*_i)$$
 Now subtracting $(2)$ and $(3)$ from $(1)$\,, we get
 \begin{gather*}1\otimes (f\cdot g)-f\otimes g-g\otimes f=\\
 \sum_{i=1}^n(\partial^1/\partial^1z_i(f\cdot g)-f\cdot \partial^1/\partial^1z_i(g)-g\cdot \partial^1/\partial^1z_i(f))\otimes z_i\,\,\,\mbox{mod}(M_{FD,j}).
 \end{gather*} But the sum on the right hand side is zero because of $(*_i), i=1,...,n$.
 \end{proof}
 We have the following 
 \begin{lemma}\mylabel{lem:L47}With notations as above and fixing the isomorphism $j$, there are canonical isomorphisms
 \begin{gather*}\Omega^1_{top}(\mathbb C[[\Gamma]]/\mathbb C)\cong \bigoplus_{i=1}^n\mathbb C[[\Gamma]]\cdot d^1z_i\,\,\mbox{and}\\
 \Omega^1_{top}(\mathbb C((\Gamma))/\mathbb C)\cong \bigoplus_{i=1}^n\mathbb C((\Gamma))\cdot d^1z_i.
 \end{gather*}
 \end{lemma}
 \begin{proof} We give the proof in the case of formal generalized power series, the case of Laurent series being identical.
 We write down two homomorphisms of $\mathbb C[[\Gamma]]$-modules
 \begin{gather*}\phi: \bigoplus_{i=1}^n\mathbb C[[\Gamma]]\cdot d^1z_i\longrightarrow \Omega^1_{top}(\mathbb C[[\Gamma]]/\mathbb C)\\
 \sum_{i=1}^nf_i\cdot d^1z_i\mapsto \overline{\sum_{i=1}^nf_i\otimes z_i}\,\,\mbox{and}\\
 \psi: \mathbb C[[\Gamma]]\otimes_{\mathbb C}\mathbb C[[\Gamma]]\longrightarrow \bigoplus_{i=1}^n\mathbb C[[\Gamma]]\cdot d^1z_i;\\
 \sum_jf_j\otimes g_j\mapsto\sum_jf_j\cdot \sum_{i=1}^n\partial^1/\partial^1z_i(g_j)\cdot d^1z_i=\\
 \sum_{i=1}^n(\sum_{j=1}^nf_j\cdot \partial^1/\partial^1z_i(g_j))\cdot d^1z_i.
 \end{gather*}
 We want to show that the homomorphism $\psi$\, factors through $\Omega^1_{top}(\mathbb C[[\Gamma]]/\mathbb C)$\,, i.e., we have to show that each generator of the submodule $M_{FD,j}$\,  maps to zero. So take a generator 
 $$1\otimes f-\sum_{i=1}^n\partial^1/\partial^1z_i(f)\otimes z_i.$$ Obviously, 
 \begin{gather*}\psi(1\otimes f)=\sum_{i=1}^n\partial^1/\partial^1z_i(f)\otimes d^1z_i;\\
 \psi(1\otimes z_i)=1\otimes z_i
 \end{gather*}
 and thus, by linearity
 \begin{gather*}\psi(\sum_{i=1}^n\partial^1/\partial^1z_i(f)\otimes z_i)=\sum_{i=1}^n\partial^1/\partial^1z_i(f)\cdot \psi(1\otimes z_i)\\
 =\sum_{j=1}^n\partial^1/\partial^1z_i(f)\cdot 1\otimes z_i
 \end{gather*} and we get
 $$\psi(1\otimes f-\sum_{i=1}^n\partial^1/\partial^1z_i(f)\otimes z_i)=0,$$
  as was to be shown. Hence we get a homomorphism of $\mathbb C[[\Gamma]]$-modules
  $$\overline{\psi}: \Omega^1_{top}(\mathbb C[[\Gamma]]/\mathbb C)\longrightarrow \bigoplus_{i=1}^n\mathbb C[[\Gamma]]\cdot d^1z_i.$$
  We calculate the effect of both compositions $\overline{\psi}\circ \phi$\, and $\phi\circ \overline{\psi}$\, on generating elements:
  \begin{gather*}\overline{\psi}\circ \phi(d^1z_i)=\overline{\psi}(\overline{1\otimes z_i})\\
  =\psi(1\otimes z_i)=d^1z_i\,\,\mbox{and}\\
  \phi\circ \overline{\psi}(\overline{f\otimes g})=\phi(f\cdot \sum_{i=1}^n\partial^1/\partial^1z_i(g)\cdot d^1z_i)=\sum_{i=1}^n\overline{f\cdot \partial^1/\partial^1z_i(g)\otimes z_i}\\
  =\overline{f\otimes g}
  \end{gather*}
  the last equality because of the definition of the module $M_{FD,j}$. 
  So the required isomorphism is proven.
 \end{proof}
 \begin{remark} For each choice of $\mathbb K$-basis $x_1,...,x_n$\,, we get an order isomorphism $(\Gamma,<)\cong (\mathbb K^n,<_{\underline{x}}),$\, where the order on $\mathbb K^n$\, is defined by transport of structure. If $y_1,...,y_n$\, is another $\mathbb K$-basis we get another ring of generalized Laurent series $\mathbb C((\mathbb K^n,<_{\underline{y}}))$\, and both rings are canonically isomorphic to $\mathbb C((\Gamma))$\,. One cannot say, that these are the same rings. We can write each $x_i$ as a power series $x_i=\sum_{\gamma}a_{\gamma}^iz^{\gamma}$\, and could try to embed the ring $\mathbb C((\mathbb K^n,<_{\underline{x}}))$\, into $\mathbb C((\Gamma))$\, by forming Laurent series $p(x_1(z),x_2(z),...,x_n(z))\in \mathbb C((\Gamma))$\,. It is not clear at all, how to define these iterated series. First, we cannot use the order topology, since $\sum_{\gamma\in \Gamma}a_{\gamma}z^{\gamma}\in \mathbb C((\Gamma))$\, does not satisfy in general $\lim_{a_{\gamma}\neq 0}\gamma=\infty,$\, there are in general upper culmination points in $\Gamma$\,. So, as a formal sum, it is not clear at all that $p(x_1(z),...,x_n(z))$\, really exists, i.e., that the coefficients of the iterated series are finite sums. In the case of generalized power series with exponents in $\mathbb Z^n$\,, the arguement in \cite{Bochner}[Chapter I, Paragraph 1.1 Calculus of formal power series, pp.3-10] of constructing iterated power series breaks down because there is no appropriate order function on $(\mathbb K^n,<_{\underline{x}})$\,. In general, we do not have $\mathbb K^n_{>_{\underline{x}}0}=\mathbb K_{>0}^n$\,. For these and related questions, see \cite{Hoeven}. Thus, we should denote $\Omega^1_{top}(\mathbb C((\Gamma))/\mathbb C)$\, more correctly as $\Omega^1_{top}(\mathbb C((\mathbb Z,<_{\underline{x}}))/\mathbb C)$\,. A different choice of $\mathbb Z$-basis $y_1,...,y_n$\, then gives an isomorphic module (but not the same module).
 \end{remark} 
 By \prettyref{thm:T2}, there is an embedding of fields over $\mathbb C$\,, $$j_{\nu}:K(X)\hookrightarrow \mathbb C((\Gamma_{\nu})),$$
 such that the natural valuation $\nu: \mathbb C((\Gamma_{\nu}))\longrightarrow \Gamma_{\nu}$\, is an extension of $\nu: K(X)\longrightarrow \Gamma_{\nu}$\,. Let 
$$\overline{\nu}: \overline{K(X)}\longrightarrow \Gamma_{\overline{\nu}}$$
 be an extension of $\nu$\, to the algebraic closure  of $K(X)/\mathbb C$\,. By the same theorem, there is an embedding over $\mathbb C$\,,
$$j_{\overline{\nu}}: \overline{K(X)}\hookrightarrow \mathbb C((\Gamma_{\overline{\nu}}))$$ such that the following diagram is commutative:
\begin{equation*}
\begin{CD} K(X)@>j_{\nu}>> \mathbb C((\Gamma))\\
          @VVjV                  @VV\mathbb C((i)) V\\
         \overline{K(X)} @>j_{\overline{\nu}}>> \mathbb C((\Gamma_{\overline{\nu}})). 
\end{CD}
\end{equation*}
This follows because by \cite{Kaplansky}[ §3, Theorem 5, p.312 and §4, Theorem 8, p.318] the maximal immediate extension of $K(X)$\, is contained in the maximal immediate extension of $\overline{K(X)}$\, and both maximal immediate extensions are isomorphic to the Laurent series rings on the right hand side of the commutative square.
 \begin{lemma}\mylabel{lem:L15} Let $(K(X),\nu)/\mathbb C$\, be a valued function field with value group $\Gamma_{\nu},$\, such that $\nu$ has maximal rational rank $rr(\nu)=\mbox{trdeg}(K(X)/\mathbb C)=n$\,.  Let 
 $$j_{\nu}:(K(X),\nu)\hookrightarrow (\mathbb C((\Gamma_{\nu})),\nu_{can})$$ be an embedding of fields over $\mathbb C$\, that respects the valuations. Then, there is an isomorphism
 $$\beta_{K(X)}:\Omega^1(K(X)/\mathbb C)\otimes_K\mathbb C((\Gamma_{\nu}))\stackrel{\cong}\longrightarrow \Omega^1_{top}(\mathbb C((\Gamma_{\nu}))/\mathbb C).$$
  If $(\overline{K(X)},\overline{\nu})\supset (K(X),\nu)$\, is an extension of $\nu$\, to the algebraic closure,  and there is a commutative diagram of embeddings as above, we get  a commutative diagram of isomorphisms  compatible with the canonical embedding 
  $$\mathbb C((i)):\mathbb C((\Gamma_{\nu}))\hookrightarrow \mathbb C((\Gamma_{\overline{\nu}}))$$ coming from the inclusion of value groups $i:\Gamma_{\nu}\hookrightarrow \Gamma_{\overline{\nu}},$\,
\begin{equation}
\begin{CD}
\Omega^1(K(X)/\mathbb C)\otimes_{K(X)}\mathbb C((\Gamma_{\nu}))@>\beta_{K(X)}>> \Omega^1_{top}(\mathbb C((\Gamma_{\nu}))/\mathbb C))\\
@VVV             @VVV\\
\Omega^1(\overline{K(X)}/\mathbb C)\otimes_{\overline{K(X)}}\mathbb C((\Gamma_{\overline{\nu}}))@>\beta_{\overline{K(X)}}>> \Omega^1_{top}(\mathbb C((\Gamma_{\overline{\nu}}))/\mathbb C).
\end{CD}
\end{equation}
\end{lemma}
\begin{proof} As for the first part, fix a $\mathbb Z$-basis $x_1,...,x_n\in K(X)$\, (we know that $\nu$\, is an Abhyankar place). Identifying $x_1,...,x_n$\, with their images in $\mathbb C((\Gamma_{\nu})),$\, we get an isomorphism $j: \Gamma_{\nu}\cong \mathbb Z^n$\,. We take the topological K\"ahler differentials $\Omega^1_{top}(\mathbb C((\Gamma_{\nu}))/\mathbb C)$\, with respect to the isomorphism $j$\,. By \ref{lem:L47}, we get an isomorphism of $\mathbb C((\Gamma_{\nu}))$-vector spaces
$$\phi: \Omega^1_{top}(\mathbb C((\Gamma_{\nu}))/\mathbb C)\stackrel{\cong}\longrightarrow \bigoplus_{i=1}^n\mathbb C((\Gamma_{\nu}))\cdot d^1x_i.$$
Also, the elements $x_1,...,x_n$\, form a transcendence basis for $K(X)/\mathbb C$\, and therefore an $K(X)$-vector spase basis for $\Omega^1(K(X)/\mathbb C)$\,: there is an isomorphism
$$\psi:\Omega^1(K(X)/\mathbb C)\stackrel{\cong}\longrightarrow \bigoplus_{i=1}^nK(X)\cdot d^1x_i.$$  
Identifying the above modules of K\"ahler differentials with this free presentations, the canonical homomorphism
$$\Omega^1(K(X)/\mathbb C)\stackrel{\Omega^1(j_{\nu}/\mathbb C)}\longrightarrow \Omega^1(\mathbb C((\Gamma_{\nu}))/\mathbb C)\twoheadrightarrow \Omega^1_{top}(\mathbb C((\Gamma_{\nu}))/\mathbb C)$$
 is nothing but the homomorphism
$$\bigoplus_{i=1}^nj_{\nu}\cdot d^1x_i: \bigoplus_{i=1}^nK(X)\cdot d^1x_i\longrightarrow \bigoplus_{i=1}^n\mathbb C((\Gamma_{\nu}))\cdot d^1x_i.$$
If we tensor the left hand side over $K(X)$ with $\mathbb C((\Gamma_{\nu}))$\,, we get the required isomorphism.\\
As for the second part, this follows easily from the commutative diagram of embeddings of valued fields
\begin{equation*}
\begin{CD} K(X)@>j_{\nu}>> \mathbb C((\Gamma))\\
          @VVjV                  @VV\mathbb C((i)) V\\
         \overline{K(X)} @>j_{\overline{\nu}}>> \mathbb C((\Gamma_{\overline{\nu}})). 
\end{CD}
\end{equation*}
 For, a typical element of $\Omega^1(K(X)/\mathbb C)\otimes_{K(X)}\mathbb C((\Gamma_{\nu}))$\, is of the form 
$$d^1f_1\otimes p_1+... d^1f_k\otimes p_k\,\, \mbox{with}\,\, f_i\in K(X)\,\, \mbox{and}\,\, p_i\in \mathbb C((\Gamma)), i=1,...,k.$$ 
Via the homorphism $\beta_{K(X)}$\, it is mapped to 
$$p_1\cdot d^1j_{\nu}(f_1)+...+ p_k\cdot d^1j_{\nu}(f_k).$$ The right vertical homomorphism  is $\Omega^1_{top}(\mathbb C((i))/\mathbb C)$\, where $i:\Gamma_{\nu}\hookrightarrow \Gamma_{\overline{\nu}}$\, is the inclusion of value groups. So the image of our typical element is 
$$\mathbb C((i))(p_1)\cdot d^1\mathbb C((i))\circ j_{\nu}(f_1)+...+\mathbb C((i))(p_k)\cdot d^1\mathbb C((i))\circ j_{\nu}(f_k)\,\, (*).$$
The left vertical arrow is $\Omega^1(j/\mathbb C)\otimes \mathbb C((i)),$\, so $d^1f_1\otimes p_1+...+d^1f_k\otimes p_k$\, is mapped to 
$$d^1j(f_1)\otimes \mathbb C((i))(p_1)+...+d^1j(f_k)\otimes \mathbb C((i))(p_k).$$
The image of this element under $\beta_{\overline{K(X)}}$\, is
$$\mathbb C((i))(p_1)\cdot d^1j_{\overline{\nu}}\circ j(f_1)+...+\mathbb C((i))(p_k)\cdot d^1j_{\overline{\nu}}\circ j(f_k)\,\, (**).$$ Comparing the elements $(*)$\, and $(**)$\, everything follows from $\mathbb C((i))\circ j_{\nu}=j_{\overline{\nu}}\circ j.$\,
\end{proof}
\subsection{The valuation formula for fields of generalized Laurent series}
\begin{proposition}\mylabel{prop:P1} Let $\Gamma$ be a totally ordered finite dimensional $\mathbb Q$-vector space such that the generalized power series ring $\mathbb C[[\Gamma]]$\, has the same Krull dimension as $\dim_{\mathbb Q}(\Gamma)$\, and let $\gamma_1,...,\gamma_n$\, be a $\mathbb Q$-basis for $\Gamma$\,  and $y_1,...,y_n\in \mathbb C((\Gamma))$\, be elements with $\nu(y_i)=\gamma_i$\,. Let 
$$\omega=f\cdot d^1y_1\wedge d^1y_2\wedge ...\wedge d^1y_n\in \Lambda^n\Omega^1_{top}(\mathbb C((\Gamma))/\mathbb C),$$
 which is a rank one vector space over $\mathbb C((\Gamma))$\,. If we define a valuation
$$\nu: \Lambda^n\Omega^1_{top}(\mathbb C((\Gamma))/\mathbb C)\longrightarrow \Gamma=\Gamma_{\nu}$$
by $$\nu(\omega)=\nu(f)+\gamma_1+\gamma_2+...+\gamma_n\in \Gamma,$$
 then this is independent of the choosen $\mathbb Q$-basis $(\gamma_1,...,\gamma_n)$\, and independent of the choosen elements $y_1,...,y_n$\,. Then this valuation satisfies the usual rules
$$\nu(\omega_1+\omega_2)\geq \min \{\nu(\omega_1),\nu(\omega_2)\}\,\,\text{and}\,\, \nu(g\cdot \omega)=\nu(g)+ \nu(\omega).$$  
\end{proposition}
 We first prove an easy step of \prettyref{prop:P1}, namely the following 
 \begin{lemma}\mylabel{lem:L2016} Let $\nu\in R_{Ab}(\overline{K(X)}/\mathbb C)$\, be an Abhyankar place of maximal rational rank. Then the above valuation formula holds if we exchange a $\mathbb Q$-basis
 $$(y_1,...,y_n)\,\, \mbox{to}\,\, (y_1^{m_1},y_1^{m_2},...,y_n^{m_n})\,\, \mbox{for}\,\, m_i\in \mathbb Q\backslash\{0\},\,\, i=1,...,n.$$
 \end{lemma}
 \begin{proof}
By formal differentiation calculus we have 
\begin{gather*}d^1y_1^{-1}\wedge d^1y_2\wedge ....\wedge d^1y_n= \frac{-1}{y_1^2}d^1y_1\wedge d^1y_2\wedge ...\wedge d^1y_n\,\,\text{and}\\
\nu(y_1^{-1})+\nu(y_2)+...\nu(y_n)=-\nu(y_1)+\nu(y_2)+...+\nu(y_n)\\
=\nu(\frac{-1}{y_1^2})+\nu(y_1)+\nu(y_2)+...\nu(y_n)
\end{gather*} and, more generally for $m_i\in \mathbb \mathbb Q\backslash\{0\},$\, 
\begin{gather*}d^1y_1^{m_1}\wedge d^1y_2^{m_2}\wedge ...\wedge d^1y_n^{m_n}\\
=m_1\cdot y_1^{m_1-1}\cdot m_2\cdot y_2^{m_2-1}\cdot ...\cdot m_n\cdot y_n^{m_n-1}\cdot d^1y_1\wedge  ...\wedge d^1y_n\,\,\text{and}\\
\nu(y_1^{m_1})+\nu(y_2^{m_2})+...+\nu(y_n^{m_n})=\sum_{i=1}^nm_i\cdot \nu(y_i)\\
=\sum_{i=1}^n(m_i-1)\cdot \nu(y_i)+\sum_{i=1}^n\nu(y_i)=\\
\nu(\prod_{i=1}^nm_i\cdot y_i^{m_i-1}\cdot d^1y_1\wedge d^1y_2\wedge ...\wedge d^1y_n)
\end{gather*}
\end{proof}
\begin{proof} Turning back to the proof of our theorem, we may assume without loss of generality by the exchange principle for $\mathbb Q$-basis (which is not more than the basis exchange principle for $\mathbb Q$-vector spaces)  that for each 
 $$i=1,...,n,\,\,\, (y_1,...,y_i,x_{i+1},...,x_n)$$ is also a $\mathbb Q$-basis. Then, it suffices to consider each change of $\mathbb Q$-basis 
 $$(y_1,...,y_i,y_{i+1},x_{i+2},...,x_n)\,\, \mbox{to}\,\, (y_1,...,y_i,x_{i+1},...,x_n).$$ The $\mathbb Q$-basis $B_i:(y_1,...,y_i, x_{i+1},...,x_n)$\, determines an order isomorphism $\Gamma_{\nu}\cong (\mathbb Q^n, <_{B_i})$\, which induces an isomorphism of generalized power series fields $\mathbb C((\Gamma_{\nu}))\cong \mathbb C(((\mathbb Q^n, <_{B_i})))$\,. Via this isomorphism we may view $y_{i+1}$\, as an element of the latter ring and  there is a generalized power series representation 
 $$y_{i+1}=p_{i+1}(y_1,y_2,...,y_k,x_{k+1},...,x_n).$$ We now calulate differentials in $\Lambda^n\Omega_{top}(\mathbb C((\Gamma))/\mathbb C)$\, where we use for the construction of this module the basis $B_i$\,.
We have 
\begin{gather*}d^1y_1\wedge...\wedge d^1 y_i\wedge d^1p_{i+1}(y_1,...,y_i,x_{i+1},...,x_n)\wedge d^1x_{i+2}\wedge...\wedge d^1x_n\\
=d^1y_1\wedge ...\wedge d^1y_i\wedge(\sum_{j=1}^i\partial^1p_{i+1}/\partial^1y_j\cdot d^1y_j+\sum_{j=i+1}^n\partial^1p_{i+1}/\partial^1x_j\cdot d^1x_j)\\
\wedge d^1x_{i+2}\wedge ...\wedge d^1x_n\\
=\partial^1 p_{i+1}/\partial^1 x_{i+1}\cdot d^1y_1\wedge...\wedge d^1y_i\wedge d^1x_{i+1}\wedge ...\wedge d^1x_n.
\end{gather*}
 All we have to show is that
\begin{gather*}\nu(y_1)+...+\nu(y_i)+\nu(\partial^1p_{i+1}(y_1,...,y_i,x_{i+1},...,x_n)/\partial^1x_{i+1})\\
+\nu(x_{i+1})+\nu(x_{i+2})+...+\nu(x_{n})\\
=\nu(y_1)+...+\nu(y_{i+1})+\nu(x_{i+2})+...+\nu(x_n),\,\,\text{or}\\
\nu(\partial^1p_{i+1}/\partial^1x_{i+1})+\nu(x_{i+1})=\nu(p_{i+1}).
\end{gather*}
Because $y_1,...,y_i,y_{i+1},x_{i+2},...,x_n$\, is a $\mathbb Q$-basis, we know that the least nonzero value $\gamma_0$ for which the coefficient $a_{\gamma_0}$\, of $p_{i+1}$\, is nonzero is of the form 
$$\gamma_0=r_1e_1+r_2e_2+...+r_ie_i+r_{i+1}e_{i+1}+...+r_ne_n\,\, \text{with}\,\, r_{i+1}\neq 0.$$
 Otherwise $e_1,...,e_i,\gamma_0,e_{i+2},...e_n$\, is not a $\mathbb Q$-basis. Then, by formal differentiation, the least nonzero $\delta_0$\, for which the coefficient in $\partial^1p_{i+1}/\partial^1x_{i+1}$\, is nonzero is 
$$\nu(\partial^1p_{i+1}/\partial^1x_{i+1})=\delta_0=r_1e_1+...+r_ie_i+(r_{i+1}-1)e_{i+1}+r_{i+2}e_{i+2}+...+r_ne_n.$$
Since $\nu(x_{i+1})=e_{i+1},$\, we get the claim.
\end{proof}
\section{Valuation of top differential forms of function fields}
\subsection{Abhyankar-places of dimension zero}
 We turn now to our main question  of valuating rational top differential forms of a function field $K(X)/\mathbb C$\, at an Abhyankar place $\nu\in R_{Ab}^{n,n}(K(X)/\mathbb C),\,\, n=\dim(X)$, for which we have $\dim(\nu)=0$\,. The idea is to embed $K(X)/\mathbb C$\, into a generalized Laurent series field $\mathbb C((\Gamma_{\nu}))$\,. This is always possible by \prettyref{thm:T2}. 
 Obviously, if $(f_1,...,f_n)\in K(X)^n$\, is a $\mathbb Q$-basis for $\Gamma_{\nu}$\,, then $\overline{\nu}(f_1),...,\overline{\nu}(f_n)$\, forms a $\mathbb Q$-vector space basis for $\Gamma_{\overline{\nu}}$\, via the natural inclusion 
$\Gamma_{\nu}\hookrightarrow \Gamma_{\overline{\nu}}$\,.
\begin{theorem}\mylabel{thm:T48} Let $K(X)/\mathbb C$\, be an algebraic function field of transcendence degree $n$\, and $\nu\in R^{n}_{Ab}(K(X)/\mathbb C)$\, be an Abhyankar place of dimension zero. Let $\omega\in \Lambda^n\Omega^1(K(X)/\mathbb C)$\, be a rational top differential form. Let $t_1,...,t_n\in K(X)$\, be an arbitrary $\mathbb Q$-basis for $\nu$\, and write
$$\omega=f\cdot d^1t_1\wedge d^1t_2\wedge ...\wedge d^1t_n.$$ Then, if we define 
$$\nu(\omega):=\nu(f)+\nu(t_1)+\nu(t_2)+...+\nu(t_n)\in \Gamma_{\nu},$$
then this value is independent of the choosen $\mathbb Q$-basis $t_1,...,t_n\in K(X)$\,.
\end{theorem}
\begin{proof}The idea is to embed $K(X)/\mathbb C$\, into a generalized Laurent series field $\mathbb C((\Gamma_{\nu}))$\,. This is always possible by \prettyref{thm:T2}. That is, there is an embedding of fields over $\mathbb C,$\, 
$$j_{\nu}:K(X)\hookrightarrow \mathbb C((\Gamma_{\nu})),$$
 such that the natural valuation $\nu: \mathbb C((\Gamma_{\nu}))\longrightarrow \Gamma_{\nu}$\, is an extension of $\nu: K(X)\longrightarrow \Gamma_{\nu}$\,. Let 
$$\overline{\nu}: \overline{K(X)}\longrightarrow \Gamma_{\overline{\nu}}\cong\Gamma_{\nu}\otimes_{\mathbb Z}\mathbb Q$$
 be an extension of $\nu$\, to the algebraic closure of $K(X)/\mathbb C$\,. By the same theorem, there is an embedding over $\mathbb C$\,,
 $$j_{\overline{\nu}}: \overline{K(X)}\hookrightarrow \mathbb C((\Gamma_{\overline{\nu}}))$$ The idea is to pass to the algebraic closure $\overline{K(X)}/\mathbb C$\, and a valuation $\overline{\nu}\in R^n_{Ab}(\overline{K(X)}/\mathbb C)$\, extending $\nu$\,. Then, since 
$$\Omega^1(\overline{K(X)}/\mathbb C)\cong \Omega^1(K(X)/\mathbb C)\otimes_{K(X)}\overline{K(X)},$$ we may view $\omega$ as a rational top differential form of $\overline{K(X)}$\, and prove the statement for arbitrary $\mathbb Q$-vector space basis $t_1,...,t_n\in \overline{K(X)}$ for $\overline{\nu}$\,. This then implies the statement of the theorem, since via the embedding $\Gamma_{\nu}\hookrightarrow \Gamma_{\overline{\nu}}$\, the independence of the value $\nu(\omega)$\, will follow from the independence of the value $\overline{\nu}(\omega)\in \Gamma_{\overline{\nu}}$\,.\\We have the following commutative diagram of embeddings of valued fields
\begin{equation*}
\begin{CD} K(X)@>j_{\nu}>> \mathbb C((\Gamma))\\
          @VV\subset V                  @VV\subset V\\
         \overline{K(X)} @>j_{\overline{\nu}}>> \mathbb C((\Gamma_{\overline{\nu}})). 
\end{CD}
\end{equation*}
Then 
$$j_{\overline{\nu}}(x_1),...,j_{\overline{\nu}}(x_n)\,\, \text{as well as}\,\, j_{\overline{\nu}}(y_1),...,j_{\overline{\nu}}(y_n)$$  give trivializations 
$$\Gamma_{\overline{\nu}}\cong (\mathbb Q^n,<_{\underline{x}})\,\,\mbox{and}\,\, \Gamma_{\overline{\nu}}\cong (\mathbb Q^n,<_{\underline{y}})$$ in $\mathbb C[[\Gamma_{\overline{\nu}}]]$\,. By \prettyref{lem:L15}, there is a canonical isomorphism
$$\alpha(j_{\overline{\nu}}):\,\,\Lambda^n\Omega^1(\overline{K(X)}/\mathbb C)\otimes_{\overline{K(X)}}\mathbb C((\Gamma_{\overline{\nu}}))\stackrel{\cong}\longrightarrow \Lambda^n\Omega^1_{top}(\mathbb C((\Gamma_{\overline{\nu}}))/\mathbb C)$$ 
and via this isomorphism we may view a given rational top differential form of $\overline{K(X)}/\mathbb C$\, as a rational top differential form of $\mathbb C((\Gamma_{\overline{\nu}}))/\mathbb C$\,. Let 
$$\omega= f\cdot d^1x_1\wedge d^1x_2\wedge ...\wedge d^1x_n=g\cdot d^1y_1\wedge d^1y_2\wedge ...\wedge d^1y_n$$ 
be two represenations of $\omega$\, with $f,g\in \overline{K(X)}$\,. Then 
$$\alpha(j_{\overline{\nu}}(\omega))=j_{\overline{\nu}}(f)\cdot d^1j_{\overline{\nu}}(x_1)\wedge ...\wedge d^1j_{\overline{\nu}}(x_n)=j_{\overline{\nu}}(g)\cdot d^1j_{\overline{\nu}}(y_1)\wedge ...\wedge d^1j_{\overline{\nu}}(y_n)$$ 
are two presentations of $\alpha(j_{\overline{\nu}})(\omega).$ By \prettyref{prop:P1}, we have 
\begin{gather*}\overline{\nu}(\alpha(j_{\overline{\nu}}(\omega))=\overline{\nu}(j_{\overline{\nu}}(f))+\overline{\nu}(j_{\overline{\nu}}(x_1))+...\overline{\nu}(j_{\overline{\nu}}(x_n))=\\
\overline{\nu}(j_{\overline{\nu}}(g))+\overline{\nu}(j_{\overline{\nu}}(y_1))+...+\overline{\nu}(j_{\overline{\nu}}(y_n)),
\end{gather*}
where we called the canonical valuation on $\mathbb C((\Gamma_{\overline{\nu}}))$\, simply again $\overline{\nu}$\,. Since 
$$j_{\overline{\nu}}:\overline{K(X)}\hookrightarrow \mathbb C((\Gamma_{\overline{\nu}}))$$ respects the valuations, we get
$$\overline{\nu}(\omega)=\overline{\nu}(f)+\overline{\nu}(x_1)+...+\overline{\nu}(x_n)=\overline{\nu}(g)+\overline{\nu}(y_1)+...+\overline{\nu}(y_n)$$
which is what we wanted to prove.
\end{proof} 
\subsection{The case of general Abhyankar places}
We will prove the following general
\begin{theorem}\mylabel{thm:T49} 
Let $K(X)/\mathbb C$\, be a function field of transcendence degree $n$ over $\mathbb C$\, and $\nu\in R_{Ab}(K(X)/\mathbb C)$\, be an Abhyankar place of dimension $k$\,. Let $\omega\in \Lambda^n\Omega^1(K(X)/\mathbb C)$\, be a rational top differential form. Choose a $\mathbb Q$-basis $(t_1,...,t_{n-k})$\, for $\nu$\, and lifts $x_{n-k+1}, ...,x_n\in K(X)$\, of a transcendence basis $\overline{x_{n-k+1}}, ...\overline{x_n}\in \kappa(\nu)/\mathbb C$\,. Write 
$$(+)\,\,\omega=f\cdot d^1t_1\wedge d^1t_2\wedge ...\wedge d^1t_{n-k}\wedge d^1x_{n-k+1}\wedge ...\wedge d^1x_n.$$
 Then if we define 
$$\nu(\omega)=\nu(f)+\nu(t_1)+...+\nu(t_{n-k})\in \Gamma_{\nu},$$ then this value is independent of the choosen $\mathbb Q$-basis $t_1,...,t_{n-k}$\,, independent of the transcendence basis $\overline{x_{n-k+1}},...,\overline{x_n}$\, of $\kappa(\nu)/\mathbb C$\, and independent of the choosen lifts $x_{n-k+1},...,x_n\in A_{\nu}$\,.
\end{theorem} 
\begin{proof}
First, we show that for a fixed $\mathbb Q$-basis $t_1,...,t_{n-k}$\, and a fixed transcendence basis $\overline{x_{n-k+1}},...,\overline{x_n}$\, the value $\nu(\omega)$\, is independend of the choosen lifts $x_{n-k+1},...,x_n\in A_{\nu}$\,. There is an Abhyankar place $\mu\in R_{Ab}^{k,k}(\kappa(\nu)/\mathbb C)$\, such that $\overline{x_{n-k+1}},...,\overline{x_n}\in \kappa(\nu)$\, is a $\mathbb Q$-basis for $\mu$\,. Then, as above, arbitrary lifts $x_{n-k+1},...,x_n$\, plus an arbitrary $\mathbb Q$-basis $t_1,...,t_{n-k}$\, for $\nu$\, form a $\mathbb Q$-basis for $\nu\circ \mu$\,. By \prettyref{thm:T49} the value $\nu\circ \mu(\omega)$\, is independend of these choices. Associated to the composed valuation  $\nu\circ\mu$,\, there is the exact sequence
$$(*)\,\,0\longrightarrow \Gamma_{\mu}\stackrel{j}\longrightarrow \Gamma_{\nu\circ \mu}\stackrel{p}\longrightarrow \Gamma_{\nu}\longrightarrow 0,$$
 Thus $\nu(\omega)=p(\nu\circ \mu(\omega))$\, and we get that the value $\nu(\omega)$\, is independent of the choice of the lifts of our transcendence basis $\overline{\underline{x}}$\,. This is what we wanted to show.\\
We then show the independence of $\nu(\omega)$ of the choosen transcendence basis $\overline{x_{n-k+1}},...,\overline{x_n}$\,.\\ By the exchange principle for transcendence basis for function fields, given two such basis 
$$\overline{x_{n-k+1}},...,\overline{x_n}\,\,\mbox{and}\,\, \overline{y_{n-k+1}},...,\overline{y_n}$$ we may assume by a reordering that $\overline{y_{n-k+j}}=\overline{x_{n-k+j}}$\, for $j=2,...,k$\,. Then $\overline{y_{n-k+1}}$\, is algebraic over $\mathbb C(\overline{x_{n-k+1}},...,\overline{x_n})$\, that is, there is a polynomial equation 
$$(\overline{*})\,\,\overline{p}(\overline{y_{n-k+1}})=\overline{y}_{n-k+1}^m+a_{m-1}(\overline{\underline{x}})\cdot \overline{y}_{n-k+1}^{m-1}+...+a_1(\overline{\underline{x}})\cdot \overline{y}_{n-k+1}+a_0(\overline{\underline{x}})=0,$$
 where 
 $$a_i(\overline{x}_{n-k+1},...\overline{x}_n)\in \mathbb C((\overline{x}_{n-k+1},...\overline{x}_n)),\,\, i=0,1,...,m-1$$ are rational functions in the algebraically independent variables $\overline{x}_j, j=n-k+1,...,n$\,.\\ We fix arbitrary lifts $x_{n-k+1},...,x_n\in A^*_{\nu}\subset K$\, of $\overline{x}_{n-k+1},...,\overline{x}_n$\, and use these lifts to lift the polynomial $\overline{p}$\, to a polynomial 
 $$(*)\,\,p(y_{n-k+1})=y_{n-k+1}^m+a_{m-1}(\underline{x})\cdot y_{n-k+1}^{m-1}+...+a_1(\underline{x})\cdot y_{n-k+1}+a_0(\underline{x})=0,$$ where the coefficients are now in $A_{\nu}^{*}$\,. Let $L/K$\, be the finite algebraic field extension obtained by adjoining all solutions of $p$ to $K$. By elementary valuation theory, there is an Abhyankar place $\nu'\in R_{Ab}(L/\mathbb C)$\, of the same rank, rational rank and the same dimension as $\nu$\, such that $\nu'\mid_K=\nu$\,. By a standard arguement, there must be $i\neq j$\, such that 
 $$\nu'(a_i(\underline{x})\cdot y_{n-k+1}^i)=\nu'(a_j(\underline{x})\cdot y_{n-k+1}^j),$$ i.e., 
 $$\nu'(y_{n-k+1}^{i-j})=\nu(a_j(\underline{x}))-\nu(a_i(\underline{x}))=0,$$ since the residues of $a_i(\underline{x})$\, and $a_j(\underline{x})$\, in $\kappa(\nu)$\, are nonzero rational functions in $\mathbb C((\overline{\underline{x}}))$\,. Thus $\nu'(y_{n-k+1})=0$\, and $y_{n-k+1}\in A_{\nu'}^*$\, for each solution $y_{n-k+1}$\, of $p$. It follows, that if we take the residues of the $m$ solutions of $p$, we get the $m$ (different) solutions of $\overline{p}$\, and we choose $\xi=y_{n-k+1}$\, in such a way that $\overline{\xi}=\overline{y}_{n-k+1}$\,. It follows that from
$$\Omega^1(L/\mathbb C)=\Omega^1(K(X)/\mathbb C)\otimes_{K(X)}L\cong\sum_{j=n-k+1}^nL\cdot d^1x_j$$
 there is a presentation
$$d^1y_{n-k+1}=\sum_{j=n-k+1}^{n}f_j\cdot d^1x_j\in \Omega^1(L/\mathbb C),\,\, f_j\in L.$$
  We now calculate the coefficients $f_j\in L$\, of the differential $d^1y_{n-k+1}$\,. From $(*)$\,, we get
 \begin{gather*}
 d^1y_{n-k+1}^m+d^1(a_{m-1}(\underline{x})\cdot y_{n-k+1}^{m-1})+...+d^1(a_1(\underline{x})\cdot y_{n-k+1})+d^1a_0(\underline{x})=0;\\
 m\cdot y_{n-k+1}^{m-1}\cdot d^1y_{n-k+1}+(m-1)\cdot y_{n-k+1}^{m-2}\cdot a_{m-1}(\underline{x})\cdot d^1y_{n-k+1}+....\\+a_1(\underline{x})\cdot d^1y_{n-k+1}=
 -y_{n-k+1}^{m-1}\cdot d^1a_{m-1}(\underline{x})-...-y_{n-k+1}\cdot d^1a_1(\underline{x})-d^1a_0(\underline{x});\\
 p'(y_{n-k+1})\cdot d^1y_{n-k+1}=-\sum_{i=0}^{m-1}y_{n-k+1}^i\cdot d^1a_i(\underline{x})\\
 \mbox{and}\\
 d^1y_{n-k+1}=\frac{-\sum_{i=1}^{m-1}y_{n-k+1}^i\cdot d^1a_i(\underline{x})}{p'(y_{n-k+1})}.
 \end{gather*} 
 The functions $x_{n-k+1},...,x_n$\, are algebraically independent and we can write 
 $$d^1a_i(\underline{x})=\sum_{j=n-k+1}^nb_{ij}d^1x_j,\,\, b_{ij}\in \mathbb C((\underline{x}))\subset A_{\nu},\,\, i=0,..., m-1.$$
 Inserting this into the last equation, we get 
 \begin{gather*}(**)\,\,d^1y_{n-k+1}=-\sum_{i=0}^{m-1}\frac{\sum_{j=n-k+1}^nb_{ij}d^1x_j}{p'(y_{n-k+1})}=\\
 \sum_{j=n-k+1}^n(\sum_{i=0}^m\frac{-b_{ij}}{p'(y_{n-k+1})})d^1x_j.
 \end{gather*}
 Thus we get with the above notation that 
 $$f_j=\sum_{i=0}^m\frac{-b_{ij}}{p'(y_{n-k+1})}.$$
 The residue of the polynomial $p'(y_{n-k+1})$\, is just $\overline{p}'(\overline{y_{n-k+1}})\in \kappa(\nu)$\, which cannot be zero since $\overline{p}$\, was choosen to be a minimal irreducible polynomial equation for $\overline{y_{n-k+1}}$\, and we are in characteristic zero.  We thus obtain that the differential $d^1y_{n-k+1}$\, is an $L$-linear combination of the differentials $d^1x_j$\, with coefficients in $A_{\nu'}$\,. The equation $(**)$\, then holds in the $A_{\nu'}$-module $\Omega^1(A_{\nu'}/\mathbb C)$\, and we can take residues under the natural homomorphism
 $$\Omega^1(A_{\nu'}/\mathbb C)\twoheadrightarrow \Omega^1(\kappa(\nu')/\mathbb C),$$ and get 
 $$ (\overline{++})\,\,d^1\overline{y_{n-k+1}}=\sum_{j=n-k+1}^n\overline{f_j}\cdot d^1\overline{x_j}.$$ 
 Turning back to our original question of the independence of the choosen transcendence basis $\overline{\underline{x}}$\,, we have replaced $\overline{x_{n-k+1}}$\, by $\overline{y_{n-k+1}}$\, and have choosen appropriate arbitrary lifts $x_{n-k+1},...,x_n\in  A^*_{\nu}\subset K$\, and $y_{n-k+1}\in A^*_{\nu'}$\, which we are free to do so by the first step of the proof. In the representation of (+) of our rational top differential form $\omega$\, we replace $x_{n-k+1}$\, by $y_{n-k+1}$\, and write 
 $$(+')\,\,\omega=g\cdot d^1t_1\wedge ...\wedge d^1t_{n-k}\wedge d^1y_{n-k+1}\wedge d^1x_{n-k+2}\wedge ...\wedge d^1x_n$$ and we get that $g\cdot \partial^1y_{n-k+1}/\partial^1x_{n-k+1}=f$\,. But by our above calculations, 
 $$\partial^1y_{n-k+1}/\partial^1x_{n-k+1}=f_{n-k+1}\in A_{\nu'}.$$
 In the equation $(\overline{++})$\,, the coefficient $\overline{f_{n-k+1}}$\, cannot be zero, since then, we had a $\kappa(\nu)$-linear dependence relation of the differentials 
 $$d^1\overline{y_{n-k+1}}, d^1\overline{x_{n-k+2}},...,d^1\overline{x_n},$$
  contradictory to the fact that $\overline{y_{n-k+1}},\overline{x_{n-k+2}},...,\overline{x_n}$\, are a transcendence basis for $\kappa(\nu)/\mathbb C$\,. Hence $\partial^1y_{n-k+1}/\partial^1x_{n-k+1}=f_{n-k+1}$\, is a unit in $A_{\nu'}$\, and we have $\nu'(f_{n-k+1})=0$\,. Thus the value of $\nu(\omega)=\nu'(\omega)$\, keeps unchanged if we pass from the presentation $(+)$\, to the presentation $(++)$\,. This is what we wanted to prove. \\
  We isolate the lastly proven facts in a separate 
\begin{lemma}\mylabel{lem:L52} With notation as above, let $\nu\in R_{Ab}(K/\mathbb C)$\, be an Abhyankar place of dimension $n-k,$\, $t_1,...,t_k$\, be a $\mathbb Q$-basis for $\Gamma_{\nu}$\, and $\overline{x_{k+1}},...,\overline{x_n}$\, and $\overline{y_{k+1}},...,\overline{y_n}$\, be two transcendence basis for $\kappa(\nu)/\mathbb C$\,. Choose arbitrary lifts $x_{k+1},...,x_n\in A_{\nu}^*$\, of $\overline{x_{k+1}},...,\overline{x_n}$\,. Then, fix minimal irreducible polynomials 
$$\overline{f_j}(\overline{y_j})=\overline{y_j}^{m_j}+ a_{m_j-1}(\overline{\underline{x}})\cdot \overline{y_j}^{m_j-1}+...+a_1(\overline{\underline{x}})\cdot \overline{y_j}+a_0(\overline{\underline{x}})=0,$$ and lift the equations $\overline{f_j}$\, via the lifts $x_j, j=k+1,...,n$\, to polynomial equations
$$f_j(y_j)=y_j^{m_j}+a_{m_j-1}(\underline{x})\cdot y_j^{m_j-1}+...+a_1(\underline{x})\cdot y_j+a_0(\underline{x})=0,$$ Then, after passing to a finite valued algebraic field extension $(L,\nu')/(K,\nu)$\, where the polynomials $f_j$\, completely decompose, we choose solutions $\xi_j=y_j$\, of $f_j$\, such that $\overline{\xi_j}=\overline{y_j}.$\, If we write a given top differential form $\omega\in \Lambda^n\Omega^1(K/\mathbb C)\subset \Lambda^n\Omega^1(L/\mathbb C)$\, as 
\begin{gather*}\omega=f\cdot d^1t_1\wedge ...\wedge d^1t_k\wedge d^1x_{k+1}\wedge ...\wedge d^1x_n\,\,\text{and}\\
\omega=\text{Jac}(\underline{t}\cup\underline{x}/\underline{t}\cup\underline{y})\cdot f\cdot d^1t_1\wedge ...\wedge d^1t_k\wedge d^1y_{k+1}\wedge ...\wedge d^1y_n,
\end{gather*}
then $\nu'(\text{Jac}(\underline{t}\cup\underline{x}/\underline{t}\cup\underline{y}))=0$\, and we have 
\begin{gather*}\overline{\text{Jac}(\underline{t}\cup\underline{x}/\underline{t}\cup\underline{y})}=\\
\text{Jac}(\overline{\underline{x}}/\overline{\underline{y}})\in \kappa(\nu')=\kappa(\nu).
\end{gather*}
\end{lemma}
\begin{proof} The proof follows easily from the preceeding arguements and by applying the exchange principle.
\end{proof}
We finally show that the value $\nu(\omega)$\, is independent of the choosen $\mathbb Q$-basis $(t_1...,t_{n-k})$\,. We choose an arbitrary Abhyankar place $\mu\in R_{Ab}(\kappa(\nu)/\mathbb C)$\, of dimension zero. By the Abhyankar equality it has rational rank $k$\,. We form the composed valuation $\nu\circ \mu\in R_{Ab}(K(X)/\mathbb C)$\, which is then of dimension zero and, by the additivity of the  rank of rational rank $n$\,. By the fundamental exact sequence
$$(*)\,\,0\longrightarrow \Gamma_{\mu}\stackrel{j}\longrightarrow \Gamma_{\nu\circ \mu}\stackrel{p}\longrightarrow \Gamma_{\nu}\longrightarrow 0,$$ the values $\nu\circ \mu(t_1),\nu\circ\mu(t_2),...,\nu\circ \mu(t_{n-k})$\, are $\mathbb Q$-independend in $\Gamma_{\nu\circ \mu}$\,. Take an arbitrary $\mathbb Q$-basis $\overline{x_{n-k+1}},...,\overline{x_n}\in \kappa(\nu)$\, for the valuation $\mu$\,. Then, by well known valuation theory, they form a transcendence basis for $\kappa(\nu)/\mathbb C$\,. Choose arbitrary lifts $x_{n-k+1},...,x_n\in A_{\nu\circ\mu}\subset A_{\nu}$\,. We are free to choose this particular transcendence basis and these particular lifts in our presentation of $\omega$\, by the first two steps of our proof. From the fundamental exact sequence and elementary linear algebra it follows that $t_1,...,t_{n-k},x_{n-k+1},...,x_n$\, form a $\mathbb Q$-basis for $\Gamma_{\nu\circ \mu}$\,. By \prettyref{thm:T48} and \prettyref{thm:T49} it follows that the value 
$$\nu\circ \mu(\omega)=\nu\circ\mu(f)+\sum_{i=1}^{n-k}\nu\circ\mu(t_i)+\sum_{j=n-k+1}^n\nu\circ\mu(x_j)\in \Gamma_{\nu\circ\mu}$$ is independent of the $\mathbb Q$-basis $(t_1,...,t_{n-k},x_{n-k+1},...,x_n)$\,. Then, $p(\nu\circ\mu(\omega))\in \Gamma_{\nu}$\, is independent of all choices. But this is precisely 
$$\nu(f)+\nu(t_1)+...+\nu(t_{n-k})=\nu(\omega)$$
 since $x_{n-k+1},...,x_n$\, are units in $A_{\nu}$\, and thus $p(\nu\circ \mu(x_j))=\nu(x_j)=0$\,. Thus  we have finally shown that the value $\nu(\omega)$\, is independent of the $\mathbb Q$-basis $t_1,...,t_{n-k}$\, for $\nu$\,.\\
\end{proof}
\begin{example} Let $X$ be an integral normal variety of dimension $n$\, and $x\in X$\, a closed point such that the local ring $\mathcal O_{X,x}$\, is regular. Let $(r_1,r_2,...,r_n)$\, be a regular sequence in $\mathfrak{m}_x$\, such that $\mathfrak{m}_x=(r_1,r_2,...,r_n)$\, and put $X'=\mbox{Bl}_xX$\, with exceptional prime divisor $E\subset X'$\,. We want to calculate 
$$\nu_E(d^1r_1\wedge d^1r_2\wedge ...\wedge d^1r_n)$$
 according to our valuation formula. We have $\nu_E(r_i)=1$\,. Put $r_i=r_1\cdot \frac{r_i}{r_1}, \forall i\geq 2$\,. Observe that the residues $\overline{r_i/r_1}\in \kappa(\nu_E); i=2,...,n$\, form a transcendence basis over $\mathbb C$\,, by the regularity assumption on our sequence $(r_1,r_2,...,r_n)$\,. We calculate
\begin{gather*}d^1r_1\wedge d^1r_2\wedge ...\wedge d^1r_n=d^1r_1\wedge d^1(r_1\cdot r_2/r_1)\wedge d^1(r_1\cdot r_3/r_1)\wedge ...\wedge d^1(r_1\cdot r_n/r_1)\\
=d^1r_1\wedge ((r_2/r_1)d^1r_1+r_1d^1(r_2/r_1))\wedge ...\wedge ((r_n/r_1)d^1r_1+r_1d^1(r_n/r_1))\\
=r_1^{n-1}d^1r_1\wedge d^1(r_2/r_1)\wedge ...\wedge d^1(r_n/r_1).
\end{gather*}
The last expression of $\omega$\, is well adapted for our purpose. We may take $r_1$ as a local parameter (``a $\mathbb Q$-basis for $\Gamma_{\nu_E}\cong \mathbb Z$\,) for $A_{\nu_E}$\, and $(\overline{r_2/r_1},...,\overline{r_n/r_1})$\, as a transcendence basis for $\kappa_{\nu_E}/\mathbb C$\,. We have $\nu(r_1)=1$\, and thus 
$$\nu(\omega)=\nu(r_1^{n-1}d^1r_1\wedge d^1(r_2/r_1)\wedge ...\wedge d^1(r_n/r_1))=(n-1)\cdot \nu(r_1)+\nu(r_1)=n.$$
We see that, in disaccordance with the classical thruth that $\nu(\omega)=n-1$\, we get $\nu(\omega)=n$\,.
\end{example}
\begin{example}\mylabel{ex:E49} Let $X$ be a normal integral scheme and $x\in X$\, a closed point, $\mathfrak{q}=(x_1,...,x_n)$\, be a parameter ideal in the local ring $\mathcal O_{X,x}$\,. Let $X^n$\, be the normalization of the blowing up $X':=\text{Bl}_{\mathfrak{q}}X$\,. Let $E_1,...,E_k$\, be the prime divisors of the exceptional locus of $p^n: X^n\longrightarrow X$\, with corresponding discrete algebraic valuations $\nu_i, i=1,...,k$\,.Let 
$$\omega=d^1x_1\wedge d^1x_2\wedge ...\wedge d^1x_n\in \Lambda^n\Omega^1(K(X)/\mathbb C).$$ We will prove that for all $i=1,...,k$\, we have $\nu_i(\omega)=\sum_{j=1}^n\nu_i(x_j)$\,.\\
First of all it is known by general commutative algebra that the graded ring 
$\mbox{gr}_{\mathfrak{q}}\mathcal O_{X,x}$\, is isomorphic to the graded ring $A/\mathfrak{q}[T_1,...,T_n]$\, in $n$ indeterminates. The isomorphism sends $\overline{x_j}\in \mathfrak{q}/\mathfrak{q}^2$\, to $T_j, j=1,...,n$\,. The exceptional divisor $E'$\,, which is thus irreducible but nonreduced on $X'$\, is locally given by one of the equations $x_j=0$\,. Let $E_i\subset X^n$\, be one of the components of the exceptional divisors of the morphism $p^n$\,. Then $E_i\longrightarrow E'_{red}$\, is a finite morphism and via the inclusion $\mathfrak{m}_i\subset \mathcal O_{X',E'_{red}}\hookrightarrow \mathcal O_{X^n,E_i}$\, we have $x_j\in \mathfrak{m}_i, j=1,...,n$\,. Since we know that  $\kappa(E'_{red})\cong \mathbb C(X_1,...,X_{n-1})$\, where each $X_j$\, corresponds to the variable, say $\frac{T_j}{T_n}$\,which is the residue of $\frac{x_j}{x_n}\in \mathcal O_{X',E'_{red}}$\, we see that for $j=1,...,{n-1}$\, the algebraic functions $\frac{x_j}{x_n}$\, are units in $\mathcal O_{X^n,E_i}$\, and $\frac{T_j}{T_n}=\frac{\overline{x_j}}{\overline{x_n}}$\, form a transcendence basis for $\kappa(E_i)\supset \kappa(E'_{red})$\, for $i=1,...,k$\,. We then have for all $i=1,...,k$\, $\nu_i(x_j)=\nu_i(x_n)$\,. Writing 
\begin{gather*}\omega=d^1x_1\wedge d^1x_2\wedge d^1x_n=\\
x_n^{n-1}\cdot d^1\frac{x_1}{x_n}\wedge d^1\frac{x_2}{x_n}\wedge...\wedge d^1\frac{x_{n-1}}{x_n}\wedge d^1x_n,
\end{gather*}
we get by our valuation formula for Abhyankar places that 
$$\nu_i(\omega)=(n-1)\cdot \nu_i(x_n)+\nu_i(x_n)=n\cdot \nu_i(x_n)=\sum_{j=1}^n\nu_i(x_j).$$
\end{example}
Generalizing the above example, we now give a formula for the valuation of a top differential form on a possibly singular variety, that comprises the valuation of a differential form on exceptional divisors over a local (singular) ring that extends the well known formula for a regular local ring.
\begin{proposition}
\mylabel{prop:P2}
 Let $z_1,...,z_n$\, be a transcendence basis for the function field $K/\mathbb C$\, and let $\nu\in R_{Ab}^{k,-}(K/\mathbb C)$\, be such that the subvector space of $\Gamma_{\nu}\otimes_{\mathbb Z}\mathbb Q$\, spanned by $\nu(z_1),...,\nu(z_n)$\, is equal to $\Gamma_{\nu}\otimes_{\mathbb Z}\mathbb Q$\,. Let it be spanned by, say $\nu(z_1),...,\nu(z_k)$\,. Write $$n_j\nu(z_j)=\sum_{i=1}^k n_{ji}\nu(z_i)\,\,,n_j,n_{ji}\in \mathbb Z, j=k+1,...,n.$$ Assume that the residues in $k_{\nu}$ of  $$\frac{z_j^{n_j}}{\prod_{i=1}^kz_i^{n_{ji}}}, j=k+1,...,n $$ form a transcendence basis for $k_{\nu}$\,. Then $$\nu(dz_1\wedge ...\wedge dz_n)=\nu(z_1....z_n).$$
\end{proposition}
\begin{proof} We put $L=K(z_j^{\frac{1}{n_j}}, j=1,...,k)$\,. There is an Abhyankar place $\nu'$\, of $L$ above $\nu$\, and we have $\nu'(z_j^{\frac{1}{n_j}})=\frac{\phi_{L/K}(\nu(z_j))}{n_j}$\,, where $\phi_{L/K}: \Gamma_{\nu}\hookrightarrow \Gamma_{\nu'}$\, is the inclusion map of value groups. Replacing $z_j$\, by $z_{j}':=z_j^{\frac{1}{n_j}}, j=1,...,k$\,, we have 
 $$\nu'(z_i)=\sum n_{ij}\nu(z_j'),\,\, i=k+1,...,n.$$
  We put 
  $$z_i'=z_i\prod_{j=1}^k(z_{j}')^{-n_{ij}}\in L,\,\, i=k+1,...,n.$$
   For $1\leq j\leq k$\, we have $d^1z_j=d^1(z_j')^{n_j}=n_j\cdot (z_j')^{n_j-1}d^1z_j'$\, and for $k+1\leq i\leq n$\, we have 
   $$d^1z_i=\prod_{j=1}^k(z_{j}')^{n_{ij}}d^1z_i'+z_i'\sum_kn_{ik}\cdot(\prod_{j\neq k} (z_j')^{n_{ij}})\cdot (z_k')^{n_{ik}-1}\cdot d^1z_k'$$\, by usual differentiation. Hence we get
\begin{gather*}d^1z_1\wedge d^1z_2\wedge ...\wedge d^1z_n\\
=\prod_{j=1}^kn_j\cdot \prod_{j=1}^k(z_j')^{n_j-1}\cdot \prod_{j=1}^k\prod_{i=k+1}^n(z_j')^{n_{ij}}d^1z_1'\wedge d^1z_2'\wedge ...\wedge d^1z_n'\\
=\prod_{j=1}^kn_j\cdot \prod_{j=1}^k(z_j')^{n_j-1+\sum_{i=k+1}^nn_{ij}}\cdot d^1z_1'\wedge ...\wedge d^1z_k'\wedge d^1z_{k+1}'\wedge ...\wedge d^1z_n'.
\end{gather*}
We apply \prettyref{thm:T49}. Obviously, $z_1',...,z_k'$\, form a $\mathbb Q$-basis for $\nu'$\, and $z_{k+1}',...,z_n'$\, form by assumption a transcendence basis for $\kappa(\nu)$\, and also for $\kappa(\nu')$\,. The valuation formula for $\omega=d^1z_1\wedge ...\wedge d^1z_n$\, then gives
\begin{gather*}
\nu(\omega)=\nu'(\prod_{j=1}^k(z_j')^{n_j-1+\sum_{i=k+1}^nn_{ij}})+\sum_{j=1}^k\nu'(z_j')\\
=\sum_{j=1}^k(n_j-1)\nu'(z_j')+\sum_{j=1}^k\nu'(z_j')+\sum_{j=1}^k\nu'((z_j')^{\sum_{i=k+1}^nn_{ij}})\\
=\sum_{j=1}^k\nu'((z_j')^{n_j})+\nu'(\prod_{i=k+1}^n\prod_{j=1}^k(z_j')^{n_{ij}})\\
=\sum_{j=1}^k\nu'(z_j)+\sum_{i=k+1}^n\nu'(z_i)+\sum_{i=k+1}^n\nu'(\frac{\prod_{j=1}^k(z_j')^{n_{ij}}}{z_i})\\
=\sum_{j=1}^n\nu'(z_j)=\nu(z_1\cdot ...\cdot z_n)
\end{gather*}
which was to be proved. Observe, that by assumption $\nu'(\prod_{j=1}^k(z_j')^{n_{ij}})=\nu(z_i)$\, such that the value of the fractions is zero.
Finally we give the comparison between the classical valuation $\nu_E^{cl}(\omega)$\, of rational top differential form $\omega\in \Lambda^n\Omega^1(K(X)/\mathbb C)$\, and our valuation $\nu_E(\omega)$\, viewing $\nu_E$\, is an Abhyankar place in $R^{1,1}_{Ab}(K(X)/\mathbb C)$\,.
\end{proof}
\begin{lemma}\mylabel{lem:L1966} With notation as above, there is always an equality
$$\nu_E^{cl}(\omega)+1=\nu_E(\omega).$$
\end{lemma}
\begin{proof} Fix the discrete algebraic rank one valuation $E$.\, Find a smooth complete model $X\in \mbox{Mod}(K(X)/\mathbb C)$\, such that $\nu_E$\, has divisorial center on $X$, such that  $\overline{c_X(\nu_E)}=E\subset X$\, is nonsingular. Fix a closed point $x\in E$\, and let $r_1=0$\, be a local equation for $E$ at $x\in X$\,. Complete $r_1$\, to a regular sequence $(r_1,r_2,...,r_n)\in \mathfrak{m}_x$\, generating $\mathfrak{m}_x$\,. Then $r_2,...,r_n$\, are units in $\mathcal O_{X,E}$\, and its residues $\overline{r_2},..., \overline{r_n}$\, form a transcendence basis in the reside field $\kappa(E)$. 
Write $\omega=f\cdot d^1r_1\wedge d^1r_2\wedge ...\wedge d^1r_n$\,. Then, the rational section
$$d^1r_1\wedge d^1r_2\wedge ...\wedge d^1r_n\in \mathcal O_X(K_X)$$
is regular around $x\in X$\, and there are no zeros in an open  neighbourhood  $U$\, of $x\in X$\,. Then $\mbox{div}_U(\omega)=\mbox{div}(f)$\, and 
$$\nu_E^{cl}(\omega)=\mbox{div}_{\mathcal O_{X,E}}(\omega)=\mbox{div}(f).$$
We may apply our valuation formula and find 
$$\nu_E(\omega)=\nu_E(f)+\nu_E(r_1)=\nu_E(f)+1$$
since $r_1$\, is a local parameter in $\mathfrak{m}_{X,E}$\,.
\end{proof}
\begin{remark} We might thus say that our valuation of rational top differential forms at Abhyankar places is a log valuation.
\end{remark}
\section{The generalized Poincar$\acute{\mbox{e}}$-residue map for  Abhyankar places}
The adjunction setting consists classically of a log pair $(X,D)$\, and a prime divisor $S\subset X$ such that $a(X,D,\nu_S)=0$\,. One asks for a divisor $D_S\subset S$\,, called the different of $D$ on $S$ such that $K_S+D_S= (K_X+D)\mid_S$\,, which is the adjunction formula.\\
In order to make sense of the adjunction formula, we have to fix a canonical rational top differential form $\omega$\, on $K(X)/\mathbb C$\, such that $\omega$\, has a simple pole along $S$ , and have to restrict $\omega$\, to a canonical rational top differential form $\overline{\omega}$\, on $K(S)/\mathbb C$\,. In case $S$ is normal, one then defines $$D_S:=(K_X^{\omega}+D)\mid_S- K_S^{\overline{\omega}}.$$
In order to define $\overline{\omega}$\, on $K(S)$\, we may assume that $X$ and $S\subset X$ are  smooth in a neighbourhood around the generic point of $S$.\\
Now, if $S\subset X$\, is a smooth divisor on a smooth variety it is well known that there is an isomorphism $$\mathcal O_S(K_S)\cong \mathcal O_X(K_X+S)\mid_S(*)$$ obtained from taking top exterior powers in the cotangential sequence 
$$0\longrightarrow \mathcal O_X(-S)/\mathcal O_X(-2S)\longrightarrow \Omega^1_X\mid_S\longrightarrow \Omega^1_S\longrightarrow 0.$$  We want to clarify, where a top differential form with a simple pole along $S$ is mapped to in $\mathcal O_S(K_S)$\, via the inverse of the above isomorphism. \\
Choosing locally an etale morphism $f:(X,x)\longrightarrow (\mathbb A^n,0)$\, such that $S$ is given by $f^{\sharp}(x_1)=0$\,, we are reduced to the case $X=\mathbb A^n$\, and $S$ being the coordinate hyperplane $(x_n=0)$\, in $\mathbb A^n_k$\,.\\
 From the isomorphism  
 $$\Lambda^n(\Omega^1(\mathbb A^n_k/k))\otimes_{k[x_1,...,x_n]}\mathcal O_S(S)\cong \Lambda^{n-1}(\Omega^1(\mathbb A^{n-1}_k/k))$$  we find that a rational top differential form  $$\omega=f\cdot d^1x_1\wedge d^1x_2\wedge ...\wedge d^1x_n$$\, has to be written as a tensor product $$(fx_n)\cdot d^1x_1\wedge ...\wedge d^1x_{n-1}\wedge d^1x_n\otimes \frac{1}{\overline{x_n}},$$ where $\frac{1}{\overline{x_n}}$\, is to be viewed as a local section of $\mathcal O_S(S)$\,. Then via the inverse of the isomorphism
 $$(*)\,\,\Lambda^{n-1}\Omega^1(S/k)\otimes \mathcal O_S(-S)\cong \Lambda^n\Omega^1(X/k)\mid S,$$ where 
 $$\overline{f}d^1x_1\wedge ...\wedge d^1x_{n-1}\otimes g\,\, \text{is sent to}\,\, fd^1x_1\wedge ...\wedge d^1x_{n-1}\wedge d^1g,$$ the rational top differential form $\omega$  is then  sent to 
 $$\overline{\omega}:=\overline{fx_n}\cdot d^1\overline{x_1}\wedge ...\wedge d^1\overline{x_{n-1}}$$ being a rational top differential form on $\mathbb A_k^{n-1}$\,.\\ This will be the starting point for our considerations .\\
We start this section with an auxiliary lemma about Kaplansky embeddings of algebraic function fields into generalized Laurent series fields of composed valuations.
\begin{lemma}\mylabel{lem:L51} Let $K(X)/\mathbb C$\, be an algebraic function field of transcendence degree $n>1$\, and $\nu\circ \mu\in R_{Ab}^n(K(X)/\mathbb C)$\, a composed Abhyankar place of dimension $n$ with $\mu\in R_{Ab}^{n-k}(\kappa(\nu)/\mathbb C)$\,. Then, there is a commutative diagram of embeddings of valued fields
\begin{equation*}
\begin{CD} A_{\nu\circ \mu}@>>> A_{\nu}@>\Phi>> A_{\nu_{can}}@>>>\mathbb C((\Gamma_{\nu\circ \mu})) \\
          @VV\mbox{res}V      @VV\mbox{res}V              @VV\mbox{res}V\\
         A_{\mu} @>>> \kappa(\nu) @>\phi>>      \mathbb C((\Gamma_{\mu})),
\end{CD}
\end{equation*}
where $A_{\nu_{can}}$\, denotes the  valuation ring inside $\mathbb C((\Gamma_{\nu\circ \mu}))$\, of the valuation 
$$\nu_{can}: \mathbb C((\Gamma_{\nu\circ\mu}))\longrightarrow \Gamma_{\nu\circ \mu}\twoheadrightarrow \Gamma_{\nu}.$$
The same holds true for the algebraic closure $\overline{K(X)}/\mathbb C$\, and composed Abhyankar places in $R_{Ab}^n(\overline{K(X)}/\mathbb C)$\,.
\end{lemma}
\begin{proof} We give the proof in the case of function fields, the proof for the algebraic closure carries over verbatim if one replaces isomorphisms $\Gamma\cong \mathbb Z^n$\, by isomorphisms $\Gamma\cong \mathbb Q^n$\, and works with $\mathbb Q$-basis instead.\\
 Let us describe the valuation ring $A_{\nu_{can}}$\, inside $\mathbb C((\Gamma_{\nu\circ\mu}))$\,. To ease notation, we drop the subscript $(-)_{can}$\, if possible. Fixing a $\mathbb Z$-basis  $t_1,...,t_k$\,for $\nu$\, and lifts $x_{k+1},...,x_n$\, of a $\mathbb Z$-basis $\overline{x_{k+1}},...,\overline{x_n}\in \kappa(\mu)$\,for $\mu$\, we get an isomorphism 
$$(*)\,\,\Gamma_{\nu\circ\mu}\cong \Gamma_{\nu}\oplus\Gamma_{\mu}\cong \mathbb Z^k\oplus \mathbb Z^{n-k}.$$  Using the isomorphism $(*),$\, we can write each $f\in \mathbb C((\Gamma_{\nu\circ \mu}))$\, as 
$$f=f(t_1,...,t_k, x_{k+1},...,x_n)=\sum_{\underline{r}\in \mathbb Z^k, \underline{s}\in \mathbb Z^{n-k}}a_{\underline{r},\underline{s}}\cdot \underline{t}^{\underline{r}}\cdot \underline{x}^{\underline{s}}.$$
Now $f\in A_{\nu}$\, precisely means that the first nonzero coefficient $a_{\underline{r},\underline{s}}$\, has $\underline{r}\in \mathbb Z^k_{\geq 0}\cong \Gamma_{\nu,\geq 0}$\, and this then implies that each nonzero $a_{\underline{r},\underline{s}}$\, has $\underline{r}\in \mathbb Z^k_{\geq 0}$\,. The maximal ideal $\mathfrak{m}_{\nu}\subset A_{\nu}$\, then consists of all generalized Laurent series $f$ with least nonzero $a_{\underline{r},\underline{s}}$\, with $\underline{r}\in \mathbb Z^k_{>0}\cong \Gamma_{\nu,>0}$\,.  The residue homomorphism 
$$A_{\nu}\twoheadrightarrow \kappa(\nu)\cong A_{\nu}/\mathfrak{m}_{\nu}$$
 then identifies as the map obtained by putting all $t_i, i=1,...,k$\, to zero. We obtain a generalized Laurent series in $x_{k+1},...,x_n$\, and an isomorphism $\kappa(\nu)\cong \mathbb C((\Gamma_{\mu}))$\,. \\
Now, there is a canonical injection $j:A_{\nu}\hookrightarrow A_{\nu_{can}}$\, with $j(\mathfrak{m}_{\nu})\hookrightarrow \mathfrak{m}_{\nu_{can}}$\, inducing an injection of fields over $\mathbb C$\,,
$$\kappa(\nu)=A_{\nu}/\mathfrak{m}_{\nu}\hookrightarrow \kappa(\nu_{can})=A_{\nu_{can}}/\mathfrak{m}_{\nu_{can}}.$$ We have to show that this is an inclusion of valued fields. For $\overline{f}\in \kappa(\nu), \overline{f}\neq 0$\,, choose a lift $f\in A_{\nu}^*$. Then $f$ maps to a Laurent series 
$$f(\underline{t},\underline{x})=\sum_{\underline{n}\in \mathbb Z^k,\underline{m}\in \mathbb Z^{n-k}}a_{\underline{n},\underline{m}}\cdot\underline{t}^{\underline{n}}\cdot\underline{x}^{\underline{m}}$$ with the minimal nonzero $a_{\underline{n},\underline{m}}$\, having $\underline{n}=\underline{0}$\, since $f$ was a unit in $A_{\nu}$\,. The series $f(\underline{t},\underline{x})$\, maps to the generalized Laurent series $\overline{f}(\underline{x})$\, obtained by putting each variable $t_i$\, to zero. Then $\mu_{can}(\overline{f}(\underline{x}))$\, is under the isomorphism $(*)$\, equal to $\underline{m}_0$\, where $a_{\underline{m_0}}$\, is the least nonzero coefficient in the Laurent series expansion. Now, the value $\nu\circ \mu(f)$\, lies in the subgroup $\Gamma_{\mu}\subset \Gamma_{\nu\circ \mu}$\, and this value is preserved under the embedding $j$, where $\nu_{can}\circ\mu_{can}(j(f))$\,, which lies in $\Gamma_{\mu_{can}},$\, is just the same as the value $\mu_{can}(\overline{j(f)})$\,. 
\end{proof}
We come now to the main theorem of this section.
\begin{theorem}\mylabel{thm:T11}(Generalized Poincar$\acute{\mbox{e}}$-residue map)\\
 Let $K(X)/\mathbb C$\, be a function field and $\nu\in R_{Ab}(K(X)/\mathbb C)$\, be an Abhyankar place of dimension $n-k<\dim(X)=n$\,. Let $\omega\in \Lambda^n\Omega^1(K(X)/\mathbb C)$\, be a rational top differential form with $\nu(\omega)=0$\,.  We can write 
$$(*)\,\,\omega =f\cdot d^1t_1\wedge d^1t_2\wedge ...\wedge d^1t_k\wedge d^1x_{k+1}\wedge d^1x_{k+2}\wedge ...\wedge d^1x_n$$ such that $(t_1,...,t_k)$\, is a $\mathbb Q$-basis for $\Gamma_{\nu},$\, $\nu(x_{k+1})=\nu(x_{k+2})=...=\nu(x_n)=0$\, and the residues modulo $\mathfrak{m}_{\nu}$\, $\overline{x_{k+1}},\overline{x_{k+2}},...\overline{x_n}$\, form a transcendence basis for $\kappa_{\nu}/\mathbb C$\,. Then the rational top differential form 
$$\overline{\omega}=\overline{ft_1\cdot ...\cdot t_k}\cdot d^1\overline{x_{k+1}}\wedge d^1\overline{x_{k+2}}\wedge ...\wedge d^1\overline{x_n}\in \Lambda^{n-k}\Omega^1(\kappa_{\nu}/\mathbb C)$$ is up to a constant $c\in \mathbb C$\, independent of the choosen representation $(*)$\,.\\
It thus defines a well defined $b$-divisor $\mathcal K^{\overline{\omega}}$\, of the function field $\kappa_{\nu}/\mathbb C$\,.
\end{theorem}
\begin{remark}
\mylabel{rem:R5}
 Observe that if $\nu(\omega)=0$\, then by \prettyref{thm:T49} 
 $$\nu(\omega)=\nu(ft_{1}\cdot ...\cdot t_k)=0\,\, \text{and the residue}\,\, \overline{f\cdot t_{1}\cdot ...\cdot t_k}$$ makes sense.
\end{remark}
\begin{proof} The arguement proceeds in three steps. 
\begin{enumerate}[1]
\item First we prove by fixed $\mathbb Q$-basis $(t_1,...,t_k)$\, and fixed transcendence basis $\overline{x_{k+1}},...,\overline{x_n}$\, of $\kappa(\nu)/\mathbb C$\, the independence of $\overline{\omega}$\, of the choosen lifts 
$$x_{k+1},...,x_n\in A_{\nu}^*\subset K.$$
 If $y_{k+1},...,y_n$\, is another set of lifts, by the exchange principle, we may assume that $y_{k+2}=x_{k+2},...,y_n=x_n$\,. We fix an Abhyankar place $\mu\in R_{Ab}^{n-k}(\kappa(\nu)/\mathbb C)$\, such that $\overline{x_{k+1}},...,\overline{x_n}$\, form a $\mathbb Q$-basis of $\mu$\, and form the composed valuation $\nu\circ \mu$\, which is then an Abhyankar place of dimension zero. If $x_{k+1},...,x_n$\, is any set of lifts of this $\mathbb Q$-basis to $A_{\nu}^*\subset K$\,, then $t_1,...,t_k, x_{k+1},...,x_n$\, is  by general valuation theory a $\mathbb Q$-basis for $\nu\circ \mu$\,. By \prettyref{thm:T2}, there is an embedding of fields $K\hookrightarrow \mathbb C((\Gamma_{\nu\circ \mu}))$\, such that the restriction of the canonical valuation on the latter gives the valuation $\nu\circ \mu$\, of $K/\mathbb C$\,. As explained in the introduction, we may use the above $\mathbb Q$-basis to obtain an isomorphism 
$$(*)\,\,\Gamma_{\nu\circ \mu}\otimes_{\mathbb Z}\mathbb Q\cong \Gamma_{\mu}\otimes_{\mathbb Z}\mathbb Q\oplus \Gamma_{\nu}\otimes_{\mathbb Z}\mathbb Q\cong \mathbb Q^{n-k}\oplus \mathbb Q^k\cong \mathbb Q^n.$$
We may extend the total order on $\Gamma_{\nu\circ \mu}$\, to a total order on $\Gamma_{\nu\circ \mu}\otimes_{\mathbb Z}\mathbb Q$\, in a natural way and get an  embedding of valued fields 
$$\mathbb C((\Gamma_{\nu\circ \mu}))\hookrightarrow \mathbb C((\Gamma_{\nu\circ \mu}\otimes_{\mathbb Z}\mathbb Q)).$$ By \prettyref{lem:L100} there is an Abhyankar place $\nu'\circ \mu'\in R_{Ab}^n(\overline{K(X)}/\mathbb C)$\, restricting to $\nu\circ \mu$\, on $K(X)$\, with value group isomorphic to $\Gamma_{\nu\circ \mu}\otimes_{\mathbb Z}\mathbb Q$\,. Hence, we get an embedding of valued fields 
$$\Phi:K(X)\hookrightarrow \mathbb C((\Gamma_{\nu'\circ \mu'}))\cong \mathbb C((\mathbb Q^n,<)).$$ On $\mathbb C((\Gamma_{\nu'\circ \mu'}))$\, there is another valuation which we denote by the same letter $\nu'$ given by 
$$\nu': \mathbb C((\Gamma_{\nu'\circ \mu'}))\longrightarrow \Gamma_{\nu'\circ \mu'}\twoheadrightarrow \Gamma_{\nu'}.$$ The restriction of $\nu'$\, to $K(X)$\, is given by the original valuation $\nu$\,. 
Turning to the original question of indenpendence of $\overline{\omega}$\, of the lifts, we first remark, that we get a homomorphism 
$$\Omega^1(\Phi):\,\,\Lambda^n\Omega^1(K(X)/\mathbb C)\hookrightarrow \Lambda^n\Omega^1_{top}(\mathbb C((\Gamma_{\nu'\circ \mu'}))/\mathbb C))$$ 
such that we can view $\omega$\, as a top differential form in the latter module, namely
$$\Omega^1(\Phi)(\omega)=\Phi(f)\cdot d^1\Phi(t_1)\wedge ...\wedge d^1\Phi(t_k)\wedge d^1\Phi(x_{k+1})\wedge ...\wedge d^1\Phi(x_n).$$  By \prettyref{lem:L47}, we get as well an injective homomorphism
$$\Lambda^{n-k}\Omega^1(\phi):\,\,\Lambda^{n-k}\Omega^1(\kappa(\nu)/\mathbb C)\hookrightarrow \Lambda^{n-k}\Omega_{top}^1(\mathbb C((\Gamma_{\mu}))/\mathbb C),$$
where 
$$\phi: \kappa(\nu)\hookrightarrow \mathbb C((\Gamma_{\nu}))$$ is the above embedding of fields over $\mathbb C$ so that we can view each of the top differential forms $\overline{\omega}$\, of $\kappa(\nu)/\mathbb C$\, obtained by a presentation of $\omega$\, as a top differential form of $\mathbb C((\Gamma_{\mu}))/\mathbb C.$\, Because of injectivity, it then suffices to prove equality up to a constant complex scalar factor for the forms $\Omega^1(\phi)(\overline{\omega})$\,. Write
\begin{gather*}\omega=f\cdot d^1t_1\wedge ...\wedge d^1t_k\wedge d^1x_{k+1}\wedge ...\wedge d^1x_n;\\
\overline{\omega}^1=\overline{ft_1...t_k}\cdot d^1\overline{x_{k+1}}\wedge ...\wedge d^1\overline{x_n};\\
\omega=g\cdot d^1t_1\wedge ...\wedge d^1t_k\wedge d^1y_{k+1}\wedge d^1x_{k+2}\wedge ...\wedge d^1x_n;\\
\overline{\omega}^2=\overline{gt_1...t_k}\cdot d^1\overline{y_{k+1}}\wedge d^1\overline{x_{k+2}}\wedge ...\wedge d^1\overline{x_n}.
\end{gather*}
We identify these top differential forms with their images in the corresponding topological Kaehler differential modules over the fields of generalized Laurent series. Then, $y_{k+1}$\, has an expression 
$$y_{k+1}=y_{k+1}(\underline{t},\underline{x})=\sum_{\underline{r}\in \mathbb Q^k,\underline{s}\in \mathbb Q^{n-k}}a_{\underline{r},\underline{s}}\cdot \underline{t}^{\underline{r}}\cdot \underline{x}^{\underline{s}},$$ where the first nonzero $a_{\underline{r},\underline{s}}$\, has $\underline{r}=0$\, since $y_{k+1}\in A_{\nu}^*$\,. That the residue of $y_{k+1}$\, is $\overline{x_{k+1}}$\, then reads precisely that the initial term of the Laurent series is $x_{k+1}$\,.\\
We have with the above notation $g\cdot \partial^1y_{k+1}/\partial^1x_{k+1}=f$\, and we have to show that
$$\overline{\partial^1y_{k+1}/\partial x_{k+1}}=\partial^1\overline{y_{k+1}}/\partial^1\overline{x_{k+1}}.$$
This is also meaningfull because the power series $\partial^1y_{k+1}/\partial^1x_{k+1}$\, has leading term equal to $1$ and is thus a unit in $A_{\nu}$\, and exchanging taking residues and partial differentiation after $x_{k+1}$\, amounts to either first putting each $t_i$\, to zero and then differentiate or vice versa which is the same as one easily sees.\\
\item In the second step, we prove independence of the top differential form $\overline{\omega}$\, of the choosen transcendence basis $\overline{x_i}, i=k+1,...,n$\,. Again, by the exchange priniciple, if $\overline{y_i}, i=k+1,...,n$\, is another transcendence basis, we may assume after reordering that $\overline{y_i}=\overline{x_i}, i=k+2,...,n$\,. The proof runs along the  lines of the proof of \prettyref{thm:T49}. Let 
$$\overline{f}(\overline{y_{k+1}})=\overline{y_{k+1}}^m+a_{m-1}(\overline{\underline{x}})\cdot \overline{y_{k+1}}^{m-1}+ ...a_1(\overline{\underline{x}})\cdot \overline{y_{k+1}}+a_0(\overline{\underline{x}})=0$$
 be a minimal irreducible polynomial with coefficients in $\mathbb C(\overline{\underline{x}})$\,. We choose arbitrary lifts $x_i\in A_{\nu}^*$\, of $\overline{x_i}\in \kappa(\nu)$\, which give a lift of $\overline{f}$\, to 
 $$f(y_{k+1})=y_{k+1}^m+a_{m-1}(\underline{x})\cdot y_{k+1}^{m-1}+...+a_1(\underline{x})\cdot y_{k+1}+a_0(\underline{x})=0.$$
 As in the proof of \prettyref{thm:T49} we conclude that after passing to a finite valued algebraic extension $(L,\nu')/(K,\nu)$\, that $\nu'(y_{k+1})=0$\, for any solution $y_{k+1}\in L$\, of $f$. We may furthermore choose the solution  $\xi=y_{k+1}$\, such that $\overline{\xi}=\overline{y_{k+1}}$\,. We calculated the differential $d^1y_{k+1}$\, as 
 $$d^1y_{k+1}=\frac{-\sum_{i=0}^{m-1}y_{k+1}^i\cdot d^1a_i(\underline{x})}{f'(y_{k+1})}.$$ 
 The functions $x_{k+1},...,x_n$\, are algebraically independent and we can write 
 $$d^1a_i(\underline{x})=\sum_{j=k+1}^nb_{ij}d^1x_j,\,\, b_{ij}\in \mathbb C(\underline{x})\subset A_{\nu},\,\, i=0,..., m-1.$$
 Inserting this into the last equation, we get 
 \begin{gather*}(**)\,\,d^1y_{k+1}=-\sum_{i=0}^{m-1}\frac{\sum_{j=k+1}^nb_{ij}d^1x_j}{f'(y_{k+1})}=\\
 \sum_{j=k+1}^n(\sum_{i=0}^m\frac{-b_{ij}}{f'(y_{k+1})})d^1x_j.
 \end{gather*}
 The algebraic function $f'(y_{k+1})$\, is in $A_{\nu}^*$\, since its residue is equal to $\overline{f}'(\overline{y_{k+1}})$\, which is unequal to zero because $\overline{f}$\, was an irreducible polynomial and we are in characteristic zero. Also, for the same reason, the $b_{ij}$\, are either zero and we can omit the term $d^1x_j$\, or they are units in $A_{\nu}$\,. From $(**)$\, we get by taking residues an equation of differentials in the module $\Omega^1(\kappa(\nu)/\mathbb C)$\,
 $$d^1\overline{y_{k+1}}=\sum_{j=k+1}^n(\sum_{i=0}^m\frac{-\overline{b_{ij}}}{\overline{f}'(\overline{y_{k+1}})})d^1\overline{x_j}.$$ The coefficient in front of $d^1\overline{x_{k+1}}$\, cannot be zero since otherwise we had a linear dependence relation among the differentials $d^1\overline{y_{k+1}},d^1\overline{x_{k+1}},...,d^1\overline{x_n}$\, contrary to the assumption that $\overline{y_{k+1}},\overline{x_{k+2}},...,\overline{x_n}$\, form a transcendence basis of $\kappa(\nu)/\mathbb C$\,. Thus we have seen that 
 $$\overline{\partial^1y_{k+1}/\partial^1x_{k+1}}=\partial^1\overline{y_{k+1}}/\partial^1\overline{x_{k+1}}\neq 0.$$
 If we write again
 \begin{gather*}
 \omega =f\cdot d^1t_1\wedge ...\wedge d^1t_k\wedge d^1x_{k+1}\wedge ...\wedge d^1x_n;\\
 \overline{\omega}^1=\overline{ft_1...t_k}\cdot d^1\overline{x_{k+1}}\wedge ...\wedge d^1\overline{x_k};\\
 \omega =g\cdot d^1t_1\wedge ...\wedge d^1t_k\wedge d^1y_{k+1}\wedge d^1x_{k+2}\wedge ...\wedge d^1x_n;\\
 \overline{\omega}^2=\overline{gt_1...t_k}\cdot d^1\overline{y_{k+1}}\wedge d^1\overline{x_{k+2}}\wedge ...\wedge d^1\overline{x_n},
 \end{gather*}
 all we have to show is that 
 $$\overline{g\cdot \partial^1y_{k+1}/\partial^1x_{k+1}\cdot t_1...t_k}=\overline{ft_1...t_k}$$ which is, by the usual transformation rule applied to $\overline{\omega}^2$\, equivalent to $\overline{\partial^1y_{k+1}/\partial^1x_{k+1}}=\partial^1\overline{y_{k+1}}/\partial^1\overline{x_{k+1}}$\, which is what we have just proved . \\
 We have just shown, that for a particular choice of lift $y_i$\, of $\overline{y_i}, i=k+1,...,n$\, the top rational differential form $\overline{\omega}$\, is independent of the choice of transcendence basis. But by step 1 we also know that the differential form $\overline{\omega}$\, does not depend on the lift $y_i, i=1,...,n$\,.
 \item The final point is to show independence of the choice of $\mathbb Q$- basis $t_1,...,t_k$\,. Let $s_1,...,s_k$\, be another $\mathbb Q$-basis, take arbitrary lifts $x_{k+1},...,x_n$\, of an arbitrary transcendence basis $\overline{x_{k+1}},...,\overline{x_n}$\, and write 
 \begin{gather*}
 \omega=f\cdot d^1t_1\wedge ...\wedge d^1t_k\wedge d^1x_{k+1}\wedge ...\wedge d^1x_n;\\
 \overline{\omega}^1=\overline{ft_1...t_k}\cdot d^1\overline{x_{k+1}}\wedge ...\wedge d^1\overline{x_n};\\
 \omega =g\cdot d^1s_1\wedge ...\wedge d^1s_k\wedge d^1x_{k+1}\wedge ...\wedge d^1x_n;\\
 \overline{\omega}^2=\overline{gs_1...s_k}\wedge d^1\overline{x_{k+1}}\wedge ...\wedge d^1\overline{x_n}.
 \end{gather*} The two top differential forms $\overline{\omega}^1$\, and $\overline{\omega}^2$\, determine two canonical $b$-divisors $\mathcal K^{\overline{\omega}^1}$\, and $\mathcal K^{\overline{\omega}^2}$\, of the function field $\kappa(\nu)/\mathbb C$\,. We want to prove that they are equal as $b$-divisors. Since they always differ by $\overline{(\phi)}$\, for some $\phi\in \kappa(\nu)$\, if $\overline{\omega}^1=\phi\cdot \overline{\omega}^2,$\, this shows that $\overline{(\phi)}=0$\, which can only be the case if $\phi$\, is a complex constant (what we wanted to show).\\
 To this end, let $E$\, be an arbitrary prime divisor of $\kappa(\nu)/\mathbb C$\,. Choose $\overline{x_{k+1}}'\in \kappa(\nu)$\, with $\nu_E(\overline{x_{k+1}}')>0$ and $\overline{x_{k+2}}',...,\overline{x_n}'$\, such that their residues in $\kappa(\nu_E)$\, form a transcendence basis of $\kappa(\nu_E)/\mathbb C$\,. Choose lifts $x'_{k+1},...,x'_n$\, to $A_{\nu}^*$\, in the special form as described in \prettyref{lem:L52}. Write 
 \begin{gather*}
 \omega=h\cdot d^1t_1\wedge ...\wedge d^1t_k\wedge d^1x_{k+1}'\wedge ...\wedge d^1x_n';\\
 \omega =k\cdot d^1s_1\wedge ...\wedge d^1s_k\wedge d^1x_{k+1}'\wedge ...\wedge d^1x_n'\,\,\mbox{with}\\
 h\cdot \mbox{Jac}(\underline{t}\cup \underline{x'}/\underline{t}\cup \underline{x})=h\cdot \mbox{Jac}(\underline{x'}/\underline{x})=f;\\
 k\cdot \mbox{Jac}(\underline{s}\cup\underline{x'}/\underline{s}\cup\underline{x})=k\cdot \mbox{Jac}(\underline{x'}/\underline{x})=g.
 \end{gather*}
 Now, we form the composed valuation $\mu:=\nu\circ \nu_E$\, and we know from \prettyref{thm:T49} that we can calculate $\mu(\omega)$\, as
 \begin{gather*}\mu(\omega)=\mu(h\cdot t_1\cdot...\cdot t_k\cdot x_{k+1}')=\\
 \mu(f\cdot \mbox{Jac}(\underline{x}/\underline{x}')\cdot t_1\cdot...\cdot t_k\cdot x_{k+1}')=\\
 \mu(k\cdot s_1\cdot ...\cdot s_k\cdot x_{k+1}')=\\
 \mu(g\cdot \mbox{Jac}(\underline{x}/\underline{x'})\cdot s_1\cdot ...\cdot s_k\cdot x_{k+1}').
 \end{gather*}
 Now, $\nu(f\cdot t_1\cdot ...\cdot t_k)=\nu(g\cdot s_1\cdot ...\cdot s_k)=0$\, by assumption and it follows that $\nu(\mbox{Jac}(\underline{x'}/\underline{x}))=0$\, since we can also use these two presentations of $\omega$\, to calculate $\nu(\omega)$\,. Then, taking residues modulo $\mathfrak{m}_{\nu}$\, and applying \prettyref{lem:L52}, we get
 \begin{gather*}\nu_E(\overline{\omega}^1)=\nu_E\overline{(f\cdot t_1\cdot ...\cdot t_k\cdot \mbox{Jac}(\underline{x}/\underline{x}')\cdot x_{k+1}'})=\\
 \nu_E(\overline{f\cdot t_1\cdot ...\cdot t_k\cdot x_{n+1}'}\cdot  \mbox{Jac}(\overline{\underline{x}}/\overline{\underline{x}'}))\,\,\text{and}\\
 \nu_E(\overline{\omega}^2)= \nu_E(\overline{g\cdot s_1\cdot ...\cdot s_k\cdot x_{k+1}'\cdot \mbox{Jac}(\underline{x}/\underline{x'}}))=\\
 \nu_E(\overline{g\cdot s_1\cdot ...\cdot s_k\cdot x_{k+1}'}\cdot \text{Jac}(\overline{\underline{x}}/\overline{\underline{x}'}))
 \end{gather*}
  and both quantities are, as we have seen above, equal to $\mu(\omega)$\,. \\
  Thus for an arbitrary prime divisor $E$\, in $\kappa(\nu)/\mathbb C$\,, we have $\nu_E(\overline{\omega}^1)=\nu_E(\overline{\omega}^2)$\, or $\mathcal K^{\overline{\omega}^1}(E)=\mathcal K^{\overline{\omega}^2}(E)$\, and thus the two canonical b-divisors are the same.
\end{enumerate}
\end{proof} 
We want to give a second, more conceptual proof of the preceeding theorem.
\begin{proof}
We first extend $\nu$\, to an Abhyankar place $\nu'\in R_{Ab}(\overline{K(X)}/\mathbb C)$\, and then fix an arbitrary  rational $\nu_0'\in R^{k}_{Ab}(k_{\nu'}/\mathbb C)$\, and form the composed valuation $\mu':=\nu'\circ \nu_0'$\,. \\
By the embedding theorem of Kaplansky (\prettyref{thm:T2}) we find an embedding of valued fields over $\mathbb C$ $$\overline{K(X)}\hookrightarrow \mathbb C((\Gamma_{\mu'})).$$ We have an order preserving exact sequence of ordered abelian groups
$$0\longrightarrow \Gamma_{\nu'_0}\longrightarrow \Gamma_{\mu'}\longrightarrow \Gamma_{\nu'}\longrightarrow 0.$$ 
 We consider $\omega$ as an element in $\Lambda^n\Omega_{top}^1(\mathbb C((\Gamma_{\mu'}))/\mathbb C)$\, via the inclusions 
 \begin{align*} \Lambda^n(\Omega^1(K(X)/\mathbb C))\hookrightarrow \Lambda^n(\Omega^1(K(X)/\mathbb C))\otimes_{K(X)}\overline{K(X)}\stackrel{\cong}\longrightarrow \Lambda^n(\Omega^1(\overline{K(X)}/\mathbb C))\\
  \hookrightarrow \Lambda^n(\Omega^1(\overline{K(X)}/\mathbb C))\otimes_{\overline{K(X)}}\mathbb C((\Gamma_{\mu'}))\stackrel{\cong}\longrightarrow \Lambda^n(\Omega^1_{top}(\mathbb C((\Gamma_{\mu'}))/\mathbb C)).
 \end{align*}
  The restriction $\overline{\omega}$\, is then considered via the corresponding chain of isomorphisms as an element of $\Lambda^k\Omega_{top}^1(\mathbb C((\Gamma_{\nu_0}))/\mathbb C)$\,.\\
We prove first that the restricted differential form is independent of the choice of $y_1,...,y_k$\,. We fix a $\mathbb Q$-basis $\gamma_1,...,\gamma_k$\, of $\Gamma_{\nu_0'}$\, and add to them $\gamma_{k+1},...,\gamma_n\in \Gamma_{\nu'\circ \nu_0'}$\, so as to obtain a $\mathbb Q$-basis of $\Gamma_{\nu'\circ \nu_0'}$\,.\\
 Observe that the images of $\gamma_{k+1},...,\gamma_n$\, in $\Gamma_{\nu}$\, form also a $\mathbb Q$-basis. We put $z_i:=z^{\gamma_i}, i=1,...,n$\,.  We compare each  choice of $y_1,...,y_n$\, with the choice $z_1,...,z_n$\,. \\
 We run the proof by induction on the number $l$ of different elements in the two choices. Write
 $$\omega:=f\cdot d^1z_1\wedge ...\wedge d^1z_k\wedge d^1y_{k+1},...\wedge d^1y_n= g\cdot d^1y_1\wedge ...\wedge d^1y_n.$$
 We have to show that 
 $$\overline{fy_{k+1}...y_n}\cdot d^1\overline{z_1}\wedge...\wedge d^1\overline{z_k}= \overline{gy_{k+1}...y_n}\cdot d^1\overline{y_1}\wedge ...\wedge d^1\overline{y_k}.$$
  For $l=0$\, there is nothing to prove. Assume we have equality for $l=k_0<k$\,. Let now say $y_1=z_1,...,y_{k_0}=z_{k_0}$  and $y_{k_0+1},...,y_k$\, be arbitrary. By the inductive hypothesis we then have 
  $$\overline{gy_{k+1}...y_n}\cdot d^1\overline{z_1}\wedge..\wedge d^1\overline{z_{k_0}}\wedge d^1\overline{y_{k_0+1}}...\wedge d^1\overline{y_k}= \overline{fy_{k+1}...y_n}\cdot d^1\overline{z_1}\wedge ...\wedge d^1\overline{z_n}.$$
  We now exchange $y_{k_0+1}$\, and $z_{k_0+1}$\, and write 
  $$\omega=h\cdot d^1z_1\wedge ...\wedge d^1z_{k_0}\wedge d^1y_{k_0+1}...\wedge d^1y_k\wedge ...\wedge d^1y_n.$$
   We have $h\cdot \partial^1y_{k_0+1}/\partial^1z_{k_0+1}=g$\, by applying \prettyref{lem:L52}.\\
    We have to show 
  \begin{align*}\overline{h\cdot \partial^1y_{k_0+1}/\partial^1z_{k_0+1}\cdot y_{k+1}...\cdot y_k}d^1\overline{z_1}\wedge ...\wedge d^1\overline{z_{k_0}}\wedge d^1\overline{y_{k_0+1}}\wedge ...\wedge d^1\overline{y_k} \\
  =\overline{h\cdot y_{k+1}\cdot...\cdot y_n}\cdot d^1\overline{z_1}\wedge ...\wedge d^1\overline{z_{k_0-1}}\wedge d^1\overline{y_{k_0}}\wedge ...\wedge d^1\overline{y_k}.
  \end{align*}
  This comes down to the equality $$\overline{\partial y_{k_0+1}/\partial z_{k_0+1}}=\partial\overline{y_{k_0+1}}/\partial\overline{z_{k_0+1}}.$$
  We separate this statement into the following lemma.
  \begin{lemma}
  \mylabel{lem:L7} With notation as above, let $y\in \mathbb C[[\Gamma_{\nu\circ \nu_0}]]_{(p_{\Gamma_{\nu_0}})}$ be a Laurent series. \\
  Then $\overline{\partial^1y/\partial^1z_i}=\partial^1\overline{y}/\partial^1\overline{z_i} $\, for $i=1,...,k$\,, that is, taking partial derivatives commutes with reduction to the residue field $k_{\nu}$\,.
  \end{lemma} 
    \begin{proof} Write $$\partial^1y/\partial^1z_i=\sum_{\underline{n}\in \mathbb N^n}a_{\underline{n}}\underline{z}^{\underline{n}}.$$
   One obtains the residue if one cancels all terms containing a power of some $z^j, k+1\leq j\leq n$\,. Now, a summand contains some $z^j, k+1\leq j\leq n$\, iff the partial derivative contains some $z_j, k+1\leq j\leq n$\,.  If a summand $A$ does not contain any $z_j, k+1\leq j\leq n$\,, say $A=\prod_{i=1}^kz_i^{n_i}$\,, then 
   $$\overline{\partial^1A/\partial^1z_i}=\partial^1\overline{A}/\partial^1\overline{z_i}=n_i\cdot(\prod_{i=1}^k\overline{z_i}^{n_i})/\overline{z_i}.$$
 \end{proof}
 This finishes the inductive step and we have shown that for any lifts $y_1,...,y_k$\,of algebraically independent elements in $k_{\nu}$\, we have 
 $$\overline{g\cdot y_{k+1}...\cdot y_n}\cdot d^1\overline{y_1}\wedge ...\wedge d^1\overline{y_k}=\overline{f\cdot y_{k+1}\cdot ...\cdot y_n}\cdot d^1\overline{z_1}\wedge ...\wedge d^1\overline{z_k},$$
 what was to be proven.\\
   We finally show that the canonical form $\overline{\omega}$\, does not depend on the choice of the $\mathbb Q$-basis $z_{k+1},...,z_n$\, of $\nu$\,. We may again apply the exchange property, so we may assume that the two $\mathbb Q$-basis differ only by one element, say $z_{k+1}$ and $z_{k+1}'$\,. By the first two parts of the proof, we are free to choose $z_1,...,z_k$\, without changing $\overline{\omega}$\,.\\
     Choose a $\mathbb Q$-basis $z_1,...,z_k\in k_{\nu'}$ for $\nu_0'$\, and consider them via the canonical embedding $\mathbb C((\Gamma_{\nu_0'}))\hookrightarrow \mathbb C((\Gamma_{\mu'}))$\, as elements of the latter field. Then the elements $z_1,...,z_n$\, as well as $$z_1,...,z_k,z_{k+1}',z_{k+2},...,z_n$$
  form $\mathbb Q$-basis for the canonical valuation of $k((\Gamma_{\mu'}))$\,. Write 
  $$\omega= f\cdot d^1z_1\wedge ...\wedge d^1z_n = f'\cdot d^1z_1\wedge ...d^1z_k\wedge d^1z_{k+1}'\wedge d^1z_{k+2}\wedge ...\wedge d^1z_n.$$ 
  Put $$\overline{\omega}^1=\overline{f\cdot z_{k+1}\cdot ...\cdot z_n}\cdot d^1z_1\wedge ...\wedge d^1z_k$$\, 
  and 
  $$\overline{\omega}^2=\overline{f'\cdot z_{k+1}'\cdot z_{k+2}\cdot ...\cdot z_n}\cdot d^1z_1\wedge ...\wedge d^1z_k.$$
We have 
$$\mu(\omega)=\mu(f\cdot z_1\cdot ...\cdot z_n)=\mu(f'\cdot z_1\cdot ...\cdot z_k\cdot z_{k+1}'\cdot z_{k+2}\cdot ...\cdot z_k)$$
 by \prettyref{thm:T48} and \prettyref{thm:T49} and 
 $$\nu_0(\overline{\omega}^1)=\nu_0(\overline{f\cdot z_1\cdot ...\cdot z_k\cdot z_{k+1}\cdot ...\cdot z_n})$$
  and $$\nu_0(\overline{\omega}^2)=\nu_0(\overline{f'\cdot z_{k+1}'\cdot z_{k+2}\cdot ...\cdot z_n\cdot z_1\cdot ...\cdot z_k})$$ again by \prettyref{thm:T48} and \prettyref{thm:T49}. As 
  $$\mu(\frac{f\cdot z_1\cdot ...\cdot z_n}{f'\cdot z_1\cdot ...\cdot z_k\cdot z_{k+1}'\cdot z_{k+2}\cdot ...\cdot z_n})=0,$$
   the leading term in the generalized Laurent series expansion is a constant. The homomorphism $k[[\Gamma_{\mu}]]_{p_{\Gamma_{\nu_0}}}\longrightarrow k((\Gamma_{\nu_0}))$\, is obtained by deleting all terms $a_{\gamma}z^{\gamma}$\, in a Laurent series with $\gamma\notin \Gamma_{\nu_0}$\,. In particular, if a Laurent series starts with a constant term, the image Laurent series starts with this same constant term . I.e., the Laurent series expansion in $k((\Gamma_{\nu_0}))$\, of 
   $$\frac{\overline{f\cdot z_1\cdot ...\cdot z_n}}{\overline{f'\cdot z_1\cdot ...\cdot z_k\cdot z_{k+1}'\cdot z_{k+2}\cdot ...\cdot z_n}}$$
 starts with a nonzero constant term. It follows 
 $$\nu_0(\frac{\overline{f\cdot z_1\cdot ...\cdot z_n}}{\overline{f'\cdot z_1\cdot ...\cdot z_k\cdot z_{k+1}'\cdot z_{k+2}\cdot ...\cdot z_n}})=0$$
  and then $$\nu_0(\overline{\omega}^1)=\nu_0(\overline{\omega}^2).$$ 
  As $\nu_0\in R_0(k_{\nu}/\mathbb C)$\, was arbitrary, we can take for $\nu_0$ each valuation in $R^{k,k}(k_{\nu}/\mathbb C)$\,, that is valuations composed by a flag of prime divisors. It then follows $\nu_0(\overline{\omega}^1)=\nu_0(\overline{\omega}^2)$\, for each $\nu_0\in R^{1,1}(k_{\nu})$\,. Let $\overline{\omega}^2=h\cdot \overline{\omega}^1$\,. For the corresponding canonical $b$-divisors  we then have $$\mathcal K^{\overline{\omega}^2} =\overline{\mbox{div}(h)}+\mathcal K^{\overline{\omega}^1}.$$ But the two $b$-divisors coincide, since they have at each prime divisor the same value and it follows $\overline{\mbox{div}(h)}=0$\,. This can only be the case if $h\in \mathbb C/\{0\}$\, is a constant .
\end{proof} 
 Now assume that $X$ is smooth and $\nu$ is represented by a smooth Cartier divisor on $X$. Then from the cotangential sequence $$0\longrightarrow \mathcal I_H/\mathcal I_H^2\longrightarrow \Omega_X\mid_H\longrightarrow \Omega_H\longrightarrow 0$$
   we deduce, that, if we choose an open subset $U\subset X$ with an etale morphism $p:U\longrightarrow \mathbb A^n$\, such that $H=p^{-1}(x_1=0)$\, we get that $\Omega_X$ and $\Omega_H$\, are free on $U$ and $U\cap H,$\, respectively, with generators 
   $$d^1p^{\sharp}x_1,d^1p^{\sharp}x_2,..,d^1p^{\sharp}x_n\,\, \text{and}\,\, d^1p^{\sharp}\overline{x_2},...,d^1p^{\sharp}\overline{x_n},$$ respectively.\\
   The function $t=p^{\sharp}x_1$\, is a local generator of $\mathcal I_H$\, and the class of $t$ is sent in the above sequence to $d^1t=d^1p^{\sharp}(x_1)$\,. Now if 
   $$\omega =f\cdot d^1p^{\sharp}x_1\wedge d^1p^{\sharp}x_2\wedge ... \wedge d^1p^{\sharp}x_n$$ and $f$ has a simple pole along $H$, then in the resulting well known isomorphism $\Lambda^{n-1}\Omega_H\cong \Lambda^n\Omega_X\otimes \mathcal O_H(H)$\,, the rational section $t\cdot \omega\mid_H$ of $\mathcal O_X(K_X+H)$\, corresponds to the rational top differential form  $\overline{f\cdot t}\cdot d^1p^{\sharp}\overline{x_2}\wedge ...\wedge d^1p^{\sharp}\overline{x_n}$\,.\\
   Thus in our above situation, in case $X$ and $H$ are smooth, the divisor of the above constructed top differential form $\overline{\omega}\in \Lambda^{n-1}\Omega^1(k(H)/\mathbb C)$\,is nothing but $K_X^{\omega}+H\mid_H$\,.\\
   Of course, this does not hold for arbitrary $X$ and $H$.\\ 
\section{Log discrepancies for Abhyankar places}
\begin{definition}\mylabel{def:D16} Let $(X,D)$\, be a log pair such that $K_X+D$\, is an $\mathbb R$- Cartier $\mathbb R$- divisor and $\nu\in R_{Ab}(K(X)/\mathbb C)$\, an Abhyankar place. Choose a  rational top differential form $\omega\in \Lambda^n\Omega^1(K(X)/\mathbb C)$\,. Define 
$$a(X,D,\nu):= \nu(\omega)-\nu(K_X^{\omega}+D)\in \Gamma_{\nu}\otimes_{\mathbb Z}\mathbb R.$$
Let $\nu\in R^1_{Ab}(K(X)/\mathbb C)$\, be an Abhyankar place of rank one. Let $H$ be an arbitrary $\mathbb R$-Cartier divisor. Define the log canonical threshhold of $H$ with respect to $(X,D)$\, at the place $\nu$\, as 
$$\mbox{lct}((X,D),H,\nu)):=\sup\{r\in\mathbb R\mid a(X,D+rH,\nu)\geq 0\}.$$
\end{definition}
\begin{remark}
\begin{enumerate}[1]
\item  As to the first definition, this element $a(X,D,\nu)\in \Gamma_{\nu}\otimes_{\mathbb Z}\mathbb R$\, is independent of the choice of the rational top differential form $\omega$\,. Each other $\omega'\in \Lambda^n\Omega^1(K(X)/\mathbb C)$\, can be written as $\omega'=f\cdot \omega$\,. From the definition of the valuation of a rational top differential form, we have $\nu(\omega')=\nu(f)+\nu(\omega)$\, and from the definition of a rational section of $\mathcal O_X(K_X)$\, corresponding to a form $\omega$\, we have $\nu(K_X^{\omega'}+D)=\nu(K_X^{\omega}+(f)+D)=\nu(f)+\nu(K_X^{\omega}+D)$\,. Thus,
$$\nu(\omega')-\nu(K_X^{\omega'}+D)=\nu(\omega)+\nu(f)-\nu(K_X^{\omega}+D)-\nu(f)=\nu(\omega)-\nu(K_X^{\omega}+D).$$
\item As to the definition of the log canonical threshold $\mbox{lct}((X,D),H,\nu)$\,, as $\nu$\, has rank one, there is an order preserving embedding 
$\Gamma_{\nu}\hookrightarrow \mathbb R$\, that is unique up to scaling with a factor $s\in \mathbb R$\,. We can thus calculate $a(X,D+rH,\nu)\in \mathbb R$\,.\\
Observe, also one might define $\nu(rH)$\, for a valuation of rank greater than one, (there is always an order embedding into a real extension of the real numbers, the supremum might not exist as a real number.
\end{enumerate}
\end{remark}
\begin{definition}\mylabel{def:D200} Let $(X,D)$\, be a log pair. With notations as above, if 
\begin{enumerate}[1]
\item $a(X,D,\nu)>0$\, for all $\nu\in R_{Ab}(K(X)/\mathbb C)$\,, then $(X,D)$\, is called globally kawamata log terminal (globally klt);
\item $a(X,D,\nu)\geq 0$\, for all $\nu\in R_{Ab}(K(X)/\mathbb C)$\,, then $(X,D)$\, is called globally log canonical (globally lc).
\end{enumerate}
\end{definition}
We start this section with a fundamental 
\begin{lemma}\mylabel{lem:L11} Let $(X,\Delta)$\, be a log pair and $\nu\in R_{Ab}(K/\mathbb C)$\, be an Abhyankar place of dimension $n-k<n$\,. Then, there is a log resolution $p:(\widetilde{X},\widetilde{\Delta})\longrightarrow (X,\Delta)$\, with $p^*(K_X^{\omega}+\Delta)=K_{\widetilde{X}}^{\omega}+\widetilde{\Delta}$\, for some rational top differential form $\omega$\, (crepant pull back) and a $\mathbb Z$-basis $r_1,...,r_k\in \mathfrak{m}_{\nu}\subset A_{\nu}$\, such that
\begin{enumerate}[a]
\item If $c_{\widetilde{X}}(\nu)=\widetilde{\eta}$\,, then $r_1,...,r_k$\, are in $\mathcal O_{\widetilde{X},\widetilde{\eta}}$\, and $\mathfrak{m}_{\widetilde{\eta}}=(r_1,r_2,...,r_k)$\,.
\item We have $\kappa(\nu)=\kappa(\widetilde{\eta})$\,.
\item If $\Delta=\sum_id_i\Delta_i$\, and $\nu$\, is not composed with $\Delta_i$\,, then $\widetilde{\eta}$\, does not lie on the strict transform on $\widetilde{X}$\, of $\Delta_i$\,. The same statement holds for an arbitrary finite collection $S=\{D_1,...,D_n\}$\, of prime divisors.
\item If $E_i$\, is the Zariski closure of the principal divisor $(r_i)$\, in a neighbourhood of $\widetilde{\eta}$\,, where $r_i$ is regular and $F_1,...,F_m$\, are the irreducible components of $\mbox{Exc}(p)$\,, then $\mbox{Supp}\widetilde{\Delta}\cup \bigcup_{j=1}^nF_j\cup \bigcup_{i=1}^kE_i$\, are in simple normal crossing position.
\end{enumerate}
\end{lemma}
\begin{proof} Choose a $\mathbb Z$-basis $r_1,...,r_k\in \mathfrak{m}_{\nu}$\,. Let $p_1: X_1\longrightarrow X$\, be a birational complete model such that, if $c_{X_1}(\nu)=\eta_1\in X_1$\,, then 
$$r_1,...,r_k\in \mathfrak{m}_{\eta_1}\subset \mathcal O_{X_1,\eta_1}.$$
 This is always possible since a valuation ring $A_{\nu}$\, is the inductive limit of the system of all normal $(A,\mathfrak{m})\subset A_{\nu}$\, ordered by domination of normal local rings. Furthermore, for each local ring $(A, \mathfrak{m}),$\, such that there is a morphism $\Spec A\longrightarrow X,$\, there is a complete model $X_1,$\, a scheme point $\eta_1\in X_1$\,, a birational morphism of complete normal schemes $p_1: X_1\longrightarrow X$\, such that $(\mathcal O_{X_1,\eta_1},\mathfrak{m}_{\eta_1})\cong (A,\mathfrak{m})$\,.\\
  For the same reason, since $\kappa(\nu)/\mathbb C$\, is finitely generated, for sufficiently high $X_1\longrightarrow X,$\, we must have $\kappa(\nu)=\kappa(\eta_1)$.\\
 (In any case, $\kappa(\nu)/\kappa(\eta_1)$\, is a finite field extension. Lift generators 
$$\overline{\alpha_1},...\overline{\alpha_r}\in \kappa(\nu)\,\, \mbox{to}\,\, \alpha_1,...,\alpha_r\in A_{\nu}$$ and find a local ring $(A,\mathfrak{m})$ under $A_{\nu}$\, with $\alpha_1,...,\alpha_r\in A$\, and then apply the above arguement.) \\
Since for a given prime divisor $D$, the set of all valuations whose center lies on each birational model on the strict transform of $D$ is presicely the set of all valuations composed with $\nu_D$, if $\nu$\, is not composed with any prime divisor in the set $S$, we can achieve by further blowing up that the center of $\nu$\, on $X_1$\, does not lie on the strict transforms of any of the $D_i$\,. In particular, this holds for the prime components of $\Delta$\, and the over $X$ horizontal components of the principal divisors $\mbox{div}(r_i)$\,. Thus, possibly up to one horizontal component, we can assume that all prime components of $\mbox{div}(r_i)$\, are exceptional over $X$\,.\\ 
By further blowing up $X_1$\, we may even assume that $X_1$\, is smooth. Let $G_i$\, be the Zariski closure in $X_1$\, of the principal divisor $(r_i), i=1,...,k$\, in a neighbourhood $U$ of $\eta_1$\, such that $r_i$\, are regular functions on $U$\,. Let $(X_1,\Delta_1)$\, be the crepant pull back of the log pair $(X,\Delta)$\, via the birational morphism $p_1$\,. Let $H_1,...,H_l$\, be the exceptional divisors of $p_1$\,. Then, there is a log resolution $p_2:(\widetilde{X},\widetilde{\Delta})\longrightarrow (X_1,\Delta_1)$\, where $\widetilde{\Delta}$\, is again the crepant pull back, such that 
$$\mbox{Supp}(\widetilde{\Delta})\cup\bigcup_{j=1}^lp_{2,*}^{-1}(H_j)\cup \mbox{Exc}(p_2)\cup\bigcup_{i=1}^kp_2^*G_i$$ 
are in simple normal crossing position.\\
Let $\widetilde{\eta}=c_{\widetilde{X}}(\nu)\in \widetilde{X}$\,. The ring $\mathcal O_{\widetilde{X},\widetilde{\eta}}$\, is regular and thus an UFD. We have $r_i\in \mathfrak{m}_{\widetilde{\eta}}$\,. Let 
$$r_i=\prod_{j=1}^as_j^{n_{ij}}, n_{ij}\in \mathbb N_0$$
 be the decomposition into irreducibel elements. Then,  $p_2^*G_i$\, is the Zariski closure of $(r_i)$\, in a neighbourhood $\widetilde{U}$\, of $\widetilde{\eta}$\, (take $\widetilde{U}=p_2^{-1}(U),$\,) where $r_i$\, are now regular functions on $\widetilde{U}$\,and the support of $p_2^*G_1\cup p_2^*G_2\cup ...\cup p_2^*G_k$\, is in simple normal crossing position. Thus $a\leq k$\,. Since $(r_1,r_2,...,r_k)$\, is a $\mathbb Z$-basis for $\nu$\, and the $\nu(r_i)$\, are $\mathbb Z$-linear combinations of the $\nu(s_j)$\,, the elements $(s_1,s_2,...,s_a)$\, are also a $\mathbb Z$-basis for $\nu$\, and we necessarily must have $a=k$\,. Since each $s_i\in \mathfrak{m}_{\widetilde{\eta}}$\, and
 $$\mbox{Supp}\bigcup_{i=1}^kp_2^*G_i=\mbox{Supp}\bigcup_{i=1}^aE_i,$$
  where $E_i=(s_i), i=1,...,k,$\, the $(s_i), i=1,...,k$\, are smooth around $\widetilde{\eta}$\, and they cross normally, by basic commutative algebra, $s_1,...,s_k$\, form a regular sequence and we must have $\mathfrak{m}_{\widetilde{\eta}}=(s_1,s_2,...,s_k)$\,.\\
The crepant model 
$$p:(\widetilde{X},\widetilde{\Delta})\longrightarrow (X,\Delta)\,\, \text{together with the}\,\, \mathbb Z-\text{basis}\,\, (s_1,s_2,...,s_k)$$ satisfies all the requirements of the lemma.
\end{proof}
We summarize, using notation as above, what we have shown so far.
\begin{enumerate}[1]
\item If $\nu$\, is composed with a discrete divisorial valuation $\nu_1$, then  $\nu_1$ has divisorial center on $\widetilde{X}$\, and if $E_1$ is the corresponding prime divisor, then $E_1\cup \mbox{Exc}(f)\cup f_*^{-1}(D)$\, is a simple normal crossing divisor on $\widetilde{X}$\,. If $\nu$\, is not composed with any discrete divisorial valuation, we require that $\rm{Exc}(f)$\, is an SNC divisor and as $E_1$\, we take an arbitrary $f$-exceptional divisor passing through $\widetilde{\eta}$.
\item There are $k-1$ irreducible $f$-exceptional divisors $E_2,...,E_k$\, passing through $\widetilde{\eta}$ such that if   $(t_i=0)$\, is a local equation of  $E_i$\, for $i=1,...,k$\, on $\widetilde{X}$\, in $\mathcal O_{\widetilde{X},\widetilde{\eta}}$\,, then $\nu(t_1),\nu(t_2),...,\nu(t_k)$\, form a $\mathbb Z$-basis  of $\Gamma_{\nu}$\,.\\
There is no exceptional divisor different from $E_1,...,E_k$\, passing through $\widetilde{\eta}$.
\item We may achieve that $\mbox{div}_{\widetilde{\eta}}(f^*\Delta)=\sum_{i=1}^ka_iE_i$\, for some $a_i\in \mathbb R$\,. 
\end{enumerate}
For later usage, we make the following 
\begin{definition}
\mylabel{def:D4}
 Let $\nu\in R_{Ab}^{-,k}(L/\mathbb C)$\, be an arbitrary Abhyankar place of the function field $L$. A special center of $\nu$\, is a regular local ring $(A,\mathfrak{m},\kappa)$\,   such that $\nu$\, is centered above $A$, i.e., $A\subset A_{\nu}$\, and $A\longrightarrow A_{\nu}$\, is a local homomorphism, we have $\kappa=\kappa(\nu)$\, and there is a regular sequence $(r_1,...,r_k)$\, in $\mathfrak{m}$ generating $\mathfrak{m}$ such that $(r_1,...,r_k)$\, is a $\mathbb Z$-basis for $\nu$\,. If $\nu$\, is composed with a discrete valuation $\nu_1$\,, then we require that $r_1$\, is a local parameter for $\nu_1$\,. $(r_1,...,r_k)$\, is called the corresponding special $\mathbb Z$-basis.\\
If $\nu$\, is as above and $(X,D)$\, is an arbitrary log pair such that $L$ is the function field of $X$, then a special center $(A,\mathfrak{m})$\, with corresponding special $\mathbb Z$-basis $(r_1,...,r_k)$\, is called adapted to $(X,D)$\, if the following holds.
\begin{enumerate}[1]
\item $A$ lies above $X$, i.e., there is a morphism $\Spec A\longrightarrow X$\, inducing the identity on generic points.
\item If $\nu$\, is not composed with any discrete divisorial valuation, the strict transforms of any prime components of $D$ are required not to pass through the closed point of $\Spec A$\,.\\
If $\nu$\, is composed with some discrete divisorial valuation $\nu_1,$\, then the strict transform of any prime component of $D$ being different from the prime divisor on $X$ corresponding to $\nu_1$ (in case $\nu_1$ has divisorial center on $X$) do not pass through the closed point of $\Spec A$.
\item In case $\nu$\, is not composed with any discrete divisorial valuation, each $(r_i)_x$\, is exceptional over $X$.\\
In case $\nu$\, is composed with the discrete divisorial valuation $\nu_1$\,, then each $(r_i)_x, i>1$\, is exceptional over $X$.
\end{enumerate}
\end{definition}
Thus the above lemma basically says that for each Abhyankar place $\nu$\, and each log pair $(X,D)$\, there is a special center of $\nu$\, adapted to $(X,D)$\,.\\
\begin{proposition}
\mylabel{prop:P3}
 Let $\nu$\, be an Abhyankar place and $(X,D)$\, be a log variety of the function field $(L/\mathbb C)$\,. Let $(A,\mathfrak{m},\kappa)$\, be a special center of $\nu$\, adapted to $(X,D)$\, with corresponding special $\mathbb Z$-basis $(r_1,...,r_k)$\,. Let $\nu_i$\, be the discrete divisorial valuation corresponding to the Zariski closure of $\text{div}_A(r_i)$\,. Write $$a(X,D,\nu)=\sum_in_i\nu(r_i).$$
Then $n_i=a(X,D,\nu_i)$\,.
\end{proposition}
\begin{proof} Choose a nonsingular compactification $p:X'\longrightarrow X$ of $\Spec A\longrightarrow X$\,. Choose a  specialization $\mathcal O_{X',x}$\, of $A$, $x\in X'$\, a closed point  such that $r_i\in \mathcal O_{X',x}$\, and complete $(r_1,...,r_k)$\, to a regular sequence $(r_1,...,r_k,r_{k+1},...,r_n)$\, of $\mathfrak{m}_x$\, generating $\mathfrak{m}_x$\,, $n=\dim(X)$\,. Then $r_1,...,r_n$\, form a transcendence basis of $L/\mathbb C$\,. If $\omega$\, is any rational top differential form of $L/\mathbb C$\,, then writing 
$$\omega=f\cdot dr_1\wedge dr_2\wedge ...\wedge dr_k\wedge dr_{k+1}\wedge ...\wedge dr_n,$$ 
 as a rational section of $\mathcal O_{X'}(K_{X'}),$\, $\mbox{div}(\omega)$\, is equal at the local ring $A$ to $\rm{div}(f)$\, since the $n$-form $dr_1\wedge ...\wedge dr_n$\, is regular and nonvanishing in $A$. Thus $\rm{div}(r_1\cdot ...\cdot r_k\cdot f)$\, at $A$ is the divisor of $\omega$\, plus $\sum_{i=1}^kE_i$\,.\\
 Now consider the Weil divisor $K_X^{\omega}$\, on $X$. The strict transforms of its horizontal components at the image point of $\Spec A\longrightarrow X$\, are precisely the horizontal components of $(\omega)$\, at $A$, i.e., of $(f)$\, at $A$. It follows from the definition of a special center adapted to $(X,D)$\, that $(fr_1...r_k)- p^*(K_{X}^{\omega}+D)$\, has no horizontal components at $A$ except possibly $E_1$\,.\\
  Now the residues of $r_{k+1},...,r_n$\, modulo $\mathfrak{m}_A$ form a transcendence basis for $\kappa(A)$\, which follows immediately from the definition of a regular sequence. By definition of a special center we have $\kappa= \kappa(\nu)$\, and the residues of the $r_i, i=k+1,...,n$\, form a transcendence basis for $\kappa(\nu)$\,. Hence by \prettyref{thm:T48} and \prettyref{thm:T49}, 
$$a(X,D,\nu)=\nu(f\cdot r_1\cdot ...\cdot r_k)-\nu(p^*(K_X^{\omega}+D)).$$
 Because of the above remark, we may write 
 $$(fr_1....r_k)_A-p^*(K_X^{\omega}+D)_A=\sum_{i=1}^ka_iE_i.$$
  We thus get 
  $$a(X,D,\nu)=\sum_ia_i\nu(E_i)=\sum_ia_i\nu(r_i),$$
   hence $a_i=n_i$\,. For each $1\leq i\leq k$\,, the residues of the elements $r_1,...,\widehat{r_i},...,r_n$\,in $\kappa(E_i)$\, form a transcendence basis which again follows from the fact that the $r_i$\, form a regular sequence in $\mathcal O_{X',x}$\,. Thus, $\nu_{E_i}(r_j)=0$\, for $j\neq i$\, and it follows that 
 $$\nu_{E_i}(f\cdot r_1\cdot ...\cdot r_k)-\nu_{E_i}(p^*(K_X^{\omega}+D))=a_i\,\, \text{for}\,\, i=1,...,k.$$
    But again by \prettyref{thm:T49} it follows that 
    $$a(X,D,\nu_i)=\nu_i(r_i\cdot f)-\nu_i(p^*(K_X^{\omega}+D))$$ and by the above remark $\nu_{E_i}(f\cdot r_1\cdot ...,\cdot r_k)=\nu_{E_i}(f\cdot r_i)$\, which proves the assertion.
\end{proof}
We prove the following important consequence of the above lemma.
\begin{theorem}
\mylabel{thm:T10}
 Let $(X,D)$\, be a  log variety. Then the following implications hold true.
 \begin{enumerate}[1]
 \item If $(X,D)$ is klt, then it is globally klt.
 \item If $(X,D)$ is lc, then it is globally lc.
 \item If $(X,D)$\, is plt and $D=\sum_id_iD_i+\sum_jD_j, 0< d_i< 1$\, then $\nu_{D_j}$\, are the only Abhyankar lc places of $(X,D)$\,. 
 \item If $(X,D)$\, is dlt and $Z\subset X$\, is such that $(X,D)\mid_{X/Z}$\, is log smooth , $\codim_X(Z)\geq 2$\, and $a(X,D,\nu_E)>0\,\, \forall E$\, prime and $c_X(E)\subset Z$\, then there are no Abhyankar lc places over $Z$.
 \end{enumerate}
 \end{theorem}
\begin{proof} Let $\nu$\,  be an Abhyankar place. Choose a special center $(A,\mathfrak{m},\kappa)$\, with corresponding special $\mathbb Z$-basis $(r_1,...,r_k)$\, adapted to $(X,D)$\, and let $\nu_i$\, be the discrete divisorial valuations corresponding to $r_i$ as above. By the previous proposition, if we write $a(X,D,\nu)=\sum_in_i\nu(r_i)$\,, then $n_i=a(X,D,\nu_i)$\,.
\begin{enumerate}[1]
\item If $(X,D)$\, is klt, then all $a(X,D,\nu_i)>0$, hence $a(X,D,\nu)>0$\,.
\item If $(X,D)$\, is lc, then all $a(X,D,\nu_i)\geq 0$\, and also $a(X,D,\nu)\geq 0$\,.
\item If $(X,D)$\, is plt and $r(\nu)\geq 2$\, then one $\nu_i$\, is exceptional over $X$ and $a(X,D,\nu_i)>0$\,. Hence also $a(X,D,\nu)>0$\,.
\item If $(X,D)$\, is dlt, and $\nu$\, is centered above $Z$ and $r(\nu)\geq 2$\, again one $\nu_i$\, must be exceptional and $a(X,D,\nu_i)>0$\, implies $a(X,D,\nu)>0$\,.
\end{enumerate} 
\end{proof}
\begin{remark}
\mylabel{rem:10} $(3)$\, of the above theorem implies in particular, that the locus $\mbox{NKLT}(X,D)$\,, consisting of all $\nu\in R_{Ab}(K(X)/\mathbb C)$\, with $a(X,D,\nu)\leq 0$\,  need not be closed inside the Riemann-variety $R_{Ab}(K(X)/\mathbb C)$\,, since there are Abhyankar places in the topological closure $\overline{\nu_{D_j}}$\, which we have shown not to belong to $\mbox{NKLT}(X,D)$\,.
\end{remark}
 \begin{corollary}\mylabel{cor:C10} With notation as above, the Abhyankar place $\nu\in R_{Ab}(K/\mathbb C)$\, is a log canonical center iff for all $i=1,2,...,k$\, $\nu_{E_i}$\, is a log canonical center. \\
 Conversely, if $p:(\widetilde{X},\widetilde{\Delta})\longrightarrow (X,\Delta)$\, is a log resolution of the lc log pair $(X,\Delta)$\, and $\bigcup_{i=1}^kE_i\subset \widetilde{X}$\, is an exceptional SNC-divisor, such that each $E_i$\, is an lc place of $(X,\Delta)$\, then there are Abhyankar places $\nu\in R_{Ab}^k(K/\mathbb C)$\, with $\overline{c_{\widetilde{X}}(\nu)}=\bigcap_{i=1}^kE_i$\, such that $\nu$\, is a (generalized) lc place of $(X,\Delta)$\,.
 \end{corollary}
 \begin{proof} The first part is clear from the last line of the proof of \prettyref{prop:P3}. For the second part, take e.g. the incomplete flag 
 $$\mathcal F: E_1\supset E_1\cap E_2\supset ....\supset E_1\cap E_2\cap ...\cap E_k,$$
 where $E_i$ are lc places in simple normal crossing position.
 The corresponding discrete algebraic rank $k$ valuation $\nu_{\mathcal F}$\, has center $V=\bigcap_{i=1}^kE_i$\, and if $\mathfrak{m}_{\eta}$\, is the maximal ideal in the local ring $\mathcal O_{\widetilde{X},\eta_V}$\, and $r_i=0$\, is a local equation for $E_i$\, at $\eta_V$\, (generic point of $V$), then $(r_1,...,r_k)$\, is a regular sequence that generates $\mathfrak{m}_{\eta}$\,. It is easily seen to be a $\mathbb Z$-basis of $\nu_{\mathcal F}$\, and we are in the situation of \prettyref{prop:P3} in order to conclude that $\nu_{\mathcal F}$\, is an lc place of $(X,\Delta)$\,. For $k=n,$\, one can also construct easily rank one Abhyankar lc places. Take the regular system of parameters $(r_1,r_2,...,r_n),$\, then $A=\mathbb C[r_1,...,r_n]_{(r_1,...,r_n)}$\, is a regular local ring with $\mathcal O_{\widetilde{X},\eta}$\, lying etale over $A$ and thus both local rings have the same completion $\mathbb \mathbb C[[r_1,...,r_n]].$\, Then continue as in 'Examples of zero dimensional valuations', i.e., choose $n$ $\mathbb Q$-linearely independend real numbers $\alpha_i\in \mathbb R_{>0}$\, and put $\nu(r_i)=\alpha_i$\,, and extend to $\mathbb C[[r_1,...,r_n]]$\, and then restrict to $\mathcal O_{\widetilde{X},\eta}$\,. We obtain a rank one Abhyankar place on $K/\mathbb C$\, with center $\eta=\bigcap_{i=1}^nE_i$\, on $\widetilde{X}$\,.
 \end{proof}
 As a further application, we prove
 \begin{proposition}\mylabel{prop:P1960}(Monotonicity for Abhyankar places)\\
 Let $(X,D)$\, be a normal $\mathbb Q$-factorial lc log pair and $\phi:(X,D)\longrightarrow (Y,D_Y)$\, be a flipping contraction. Let $\phi^+: (X^+,D^+)\longrightarrow (Y,D_Y)$\, be its log flip. Then for all $\nu\in R_{Ab}(K(X)/\mathbb C)$\,
 \begin{gather*}a(X,D,\nu)\leq a(X^+,D^+,\nu)\\
 \text{If}\,\, \nu\in \overline{\text{Exc}(\phi)}=\overline{\text{Exc}(\phi^+)}\,\,\text{then}\\
 a(X,D,\nu)< a(X^+,D^+,\nu).
 \end{gather*}
 \end{proposition}
 \begin{proof} The proof runs along the same lines as in \cite{Matsuki}[Lemma 9-1-3, pp. 320-321]. We choose a common log resolution $X\stackrel{\sigma}\longleftarrow V\stackrel{\sigma^+}\longrightarrow X^+$\, of $(X,D)$\, and $(X^+,D^+)$\, and choose a sufficiently divisible $l\in \mathbb N$\, such that $l(K_X+D)$\, is Cartier and $l(K_X^++D^+)$\, is $\phi^+$-very ample. Write 
 $$\sigma^*(l(K_X+D))=M+\sum_ir_iF_i,$$
  where $M$\, is the $\phi\circ \sigma$-movable part and $\sum_iF_i$\, with $r_i> 0$\, is the $\phi\circ \sigma$-fixed part. By the calculations in \cite{Matsuki}[Lemma 9-1-3, pp.320-321], we have $M=(\sigma^+)^*(l(K_X^++D^+))$\, and $\sigma(\cup_iF_i)=\text{Exc}(\phi)$\,. Thus, as $\mathbb Q$-b-Cartier divisors, we have
  $$\overline{K_X+D}>\overline{K_X^++D^+}$$ and  for $\nu\in R_{Ab}(K(X)/\mathbb C)$\, we have 
  \begin{gather*}a(X,D,\nu)=\nu(\omega)-\nu(K_X^{\omega}+D)\leq \nu(\omega)-\nu(K_{X^+}^{\omega}+D^+)\\
  =a(X^+,D^+,\nu).
  \end{gather*}
  If $\nu\in \overline{\text{Exc}(\phi)}$\,, then its center on $V$ is in $\cup_iF_i$\, because 
  $$\overline{\text{Exc}(\phi)}=\overline{\sigma^{-1}(\text{Exc}(\phi))}=\cup_i\overline{F_i}.$$
   Hence there is $i_0$\, such that $c_V(\nu)\in F_{i_0}$\, and $\nu(F_{i_0})>0$\,. Then,
\begin{gather*}  \nu(\overline{K_X^{\omega}+D})=\nu((\overline{K_X+D})_V)=\nu(\sigma^{+*}(K^{\omega}_{X^+}+D^+)+\sum_ir_iF_i)\\
  = \nu(K^{\omega}_{X^+}+D^+)+\sum_ir_i\cdot\nu(F_i)\geq \nu(K_{X^+}^{\omega}+D^+)+r_{i_0}\cdot\nu(F_{i_0})\\
  >\nu(K_{X^+}^{\omega}+D^+),
  \end{gather*}
  which implies $a(X,D,\nu)<a(X^+,D^+,\nu)$\, in this case. Moreover if $\nu\notin \overline{\text{Exc}(\phi)}$\,, then, as in $R(K(X)/\mathbb C)\backslash \overline{\text{Exc}(\phi)}$\, we have $\overline{(K_X^{\omega}+D)}=\overline{(K^{\omega}_{X^+}+D^+)}$\, as $\mathbb Q$-b-Cartier divisors, this implies $a(X,D,\nu)=a(X^+,D^+,\nu)$\,.
  \end{proof}
\section{Adjunction for non-klt-centers}
If $(X,D)$\, is a log pair such that $D=S+\sum_ia_iD_i, D_i\neq S$\,,$S$ irreducible and normal, the usual adjunction problem is to define a divisor $D_S$\, on $S$\, such that the adjunction formula $K_X+D\mid_S=K_S+D_S$\, holds. A further problem is to compare the singularity types of the pairs $(X,D)$\, and $(S,D_S)$\,.\\ 
A first step towards a generalization of this problem is to observe that $a(X,D,\nu_S)=0$\,. We will thus consider the following problem.\\
 \begin{enumerate}[i]
 \item $(X,D)$\, is  a log pair;
 \item  $\nu\in R_{Ab}^{-,k,}(X)$\, is an Abhyankar place with $0< k< \dim(X);$
 \item $c_X(\nu)=x$\, is such that the residue field $k_x$\, of $\mathcal O_{X,x}$\, is equal to $k_{\nu};$\,
 \item $\overline{\{x\}}$\, is normal and
 \item $a(X,D,\nu)=0$\,.
 \end{enumerate}
  For instance, $\nu$\, could be given by any incomplete flag $$\underline{S}:S_{1}\subset S_2\subset ...\subset S_k\subset S_{k+1}=X$$ such that $S_i$\, is irreducible and of codimension one in $S_{i+1}$\, and $S_{i+1}$\, is nonsingular at the generic point of $S_i$\,. This defines  the well known  Abhyankar places $\nu_{\underline{S}}\in R_{Ab}^{k,k,}(K(X)/\mathbb C)$\,.\\
The problem is as in the classical case to restrict the log pair $(X,D)$\, to a log pair $(X_{\nu},D_{\nu})$\, with $X_{\nu}=\overline{x}=\overline{c_X(\nu)}$\,. \\
The first thing we have to do is to restrict $K_X^{\omega}+D$\, as a $b$-Cartier log-divisor to a $b$-Cartier log-divisor of $k_{\nu}$\,.\\
We start with a lemma. 
\begin{lemma}
\mylabel{lem:L8} Let $X/\mathbb C$ be a variety , $D$ a Cartier divisor on $X$ and $\nu\in R_{Ab}^{-,k}(X/\mathbb C)$\, be an Abhyankar place centered above $X$ such that $\nu(D)=0$\,. Then we may define a $b$-Cartier divisor $D_{\nu}'$\, in $k_{\nu},$\, such that for any divisorial discrete valuation $\nu_1$\, of $k_{\nu}$\, with $\nu_1(D'_{\nu})=a$\,, if we form the composed valuation $\mu=\nu\circ \nu_1$\, of $R(K(X)/\mathbb C)$\,, then $\mu(D)=(0,a)\in \Gamma_{\mu}$\,.\\
If furthermore the residue field of the center $x$ of $\nu$ on $X$ is equal to $k_{\nu}$\, and $D$ is effective, then there is a Cartier divisor $D_{\nu}$\, on $\overline{c_X(\nu)}:=\overline{\{x\}}$\, such that $\overline{D_{\nu}}=D_{\nu}'$\, as $b$-Cartier divisors on $k_{\nu}$\,.
\end{lemma}
\begin{proof} We first construct a valuation function $$\widetilde{D}: R^{1,1}_{Ab}(k_{\nu}/\mathbb C)\longrightarrow \mathbb Z$$
 that satisfies the requirements of the theorem and then show that it is locally on $R(k_{\nu}/\mathbb C)$\, given by the $b$-Cartier closure of a single function in $k_{\nu}$\,.\\
  Let $\nu_0\in R^{1,1}_{Ab}(k_{\nu}/\mathbb C)$\, be given. We form the composed valuation $\mu:=\nu\circ \nu_0$\,. We have an exact sequence of ordered abelian groups
   $$0\longrightarrow \mathbb Z\longrightarrow \Gamma_{\mu}\longrightarrow \Gamma_{\nu}\longrightarrow 0.$$ Since $\nu(D)=0$\,, we may consider $\mu(D)$\, by the above standard exact sequence as an element of $\mathbb Z$\,  and put $$\widetilde{D}(\nu_0):=\mu(D)\in \mathbb Z.$$
    This defines the valuation function $\widetilde{D}$\,. Now let $\nu_0\in R^{1,1}_{Ab}(k_{\nu}/\mathbb C)$\, be given. Let $\mbox{RSpec}(A)$\, be a Zariski open neighbourhood of $\mu_0:=\nu\circ \nu_0$\, in $R(K(X)/\mathbb C)$\, such that the $b$-Cartier divisor $\overline{D}$\, is given on $\mbox{RSpec}(A)$\, by the divisor of a single function $\overline{(f)}$\,.\\
    If $A=k[a_1,...,a_l]\subset K(X)$\, we even have $a_i\subset A_{\mu_0}\subset A_{\nu}$\,.\\ We put $$\overline{A}:=k[\overline{a_1},...\overline{a_l}]\subset k_{\nu}$$
     where $\overline{a_i}$\, are the residues of $a_i$\, in $k_{\nu}$\,.\\
      Since $\nu(D)=0$\, we have $f\in A_{\nu}$\,. Let $\overline{f}\in k_{\nu}$\, be the residue. We claim that $\widetilde{D}=\overline{(\overline{f})}$\, on $\mbox{RSpec}\overline{A}$\,.\\
Indeed, let $\nu_1\in \mbox{RSpec}(\overline{A})$\, be given. Then $\mu_1:=\nu\circ \nu_1\in \mbox{RSpec}(A)$\, and $\mu_1(D)=\mu_1(f)$\,. We have again an exact sequence of value groups 
$$0\longrightarrow \mathbb Z\longrightarrow \Gamma_{\mu_1}\longrightarrow \Gamma_{\nu}\longrightarrow 0.$$ The first inclusion is the canonical inclusion $\Gamma_{\nu_1}\hookrightarrow \Gamma_{\mu_1}$\, via $$k_{\nu}^*/A_{\nu_1}^*\cong A_{\nu}^*/A_{\mu_1}^*\subset K(X)^*/A_{\mu_1}^*,$$
where the first inclusion is the inverse of taking residues. So the class of $f$ in $\Gamma_{\mu_1}$\, goes to the class of $\overline{f}$\, in $\Gamma_{\nu_1}\cong \mathbb Z$\,.\\
To prove the last statement, let $\mathcal O_{X,x}$\, be the local ring of the center $c_X(\nu)$\, of $\nu$\, on $X$. Let $g=0$\, be a local equation for $D$ at $x$. Since $D$ is supposed to be effective, $g\in \mathcal O_{X,x}$\, and we cannot have $g\in \mathfrak{m}_{X,x}$, the maximal ideal, since then also $g\in \mathfrak{m}_{\nu}$\, which contradicts $\nu(g)=0$\,. Thus we may restrict $D$ to $\overline{\{x\}}$\, and put $D_{\nu}:=D\mid_{\overline{\{x\}}}$\,. We have to show that the $b$-Cartier closure $\overline{D_{\nu}}$\, equals $D'_{\nu}$\,. Let $y\in \overline{\{x\}}$\, be a scheme point. Choose an open affine neighbourhood $\Spec A$\, of $y$ such that $D$ is defined on $\Spec A$ by a single equation $h=0$\,. Then $\mathfrak{m}_{X,x}\cap A$\, is the ideal defining $\overline{\{x\}}\cap \Spec A$\, and $D_{\nu}$\, is defined by the residue class of $h$ modulo $\mathfrak{m}_{X,x}\cap A$\,. Let $\nu_1$\, be a discrete divisorial valuation of $k_{\nu}$\, centered above $y$. If we put $\mu_1:=\nu\circ\nu_1,$\, then $h=0$\, is also an equation for $D$ at $\mu\in R(K(X)/\mathbb C).$\\
 We have seen that the residue class of $h$ modulo $\mathfrak{m}_{\nu}$\, is the value $\nu_1(\overline{h})=\nu_1(\widetilde{D})$\,. But $\mathfrak{m}_{X,x}\cap A\subset \mathfrak{m}_{\nu}$\, so also $\nu_1(D_{\nu})=\nu_1(\overline{h})$\, what was to be shown.
\end{proof}
Let again $(X,\Delta)$\, be an lc log pair and $Z\subset X$\, be an lc center. Choose as above an Abhyankar lc place $\nu\in R^k_{Ab}(K/\mathbb C)$\, with $\overline{c_X(\nu)}= Z$\,. Assume that we can manage that $\kappa(\nu)=\kappa(Z)$\,. By \prettyref{thm:T49}, we have constructed a generalized Poincar$\acute{e}$ residue map
\begin{gather*}Res_{\nu}:\,\,\Lambda^n\Omega^1(K/\mathbb C)_{\nu=0}\longrightarrow \Lambda^{n-k}\Omega^1(\kappa(\nu)/\mathbb C)\\
 \omega\mapsto \overline{\omega},
 \end{gather*}
  where the first set denotes the set of all top rational differential forms $\omega$\, with $\nu(\omega)=0$\,. Now, fix $\omega$\, with $\nu(\omega)=0$\,. The rational top differential form $\overline{\omega}$\, defines a canonical b-divisor $\mathcal K^{\overline{\omega}}$\, of the function field $\kappa(\nu)/\mathbb C$\,.  Now, as $a(X,\Delta,\nu)=\nu(\omega)=0$\, it follows $\nu(K_X^{\omega}+\Delta)=0$\, and we may use the restriction of $\mathbb Q-$, resp., $\mathbb R-$ Cartier divisors to define $(K_X^{\omega}+\Delta)\mid_Z$\, as a $\mathbb Q-(\mathbb R)-$ Cartier divisor on $Z$ (no component of the divisor $K_X^{\omega}+\Delta$\, may contain in its support the generic point of $Z$). We define 
  $$\Delta_Z:= (K_X^{\omega}+\Delta)\mid_Z-\mathcal K^{\overline{\omega}}_Z.$$
 If $Y\in \mbox{Mod}(\kappa(Z)/\mathbb C)$\, is any complete birational model above $Z$ with morphism $p:Y\longrightarrow Z$\,, then we define 
 $$\Delta_Y:=(p^*(K_X^{\omega}+\Delta))\mid_Y-\mathcal K^{\overline{\omega}}_Y.$$ Since $\mathcal K^{\overline{\omega}}$\, is a $b$-divisor, we have $p_*\Delta_Y=\Delta_Z$\,. This way, we have defined a $b$-divisor $\Delta^{\kappa(Z)}$\, of the function field $\kappa(Z)/\mathbb C,$\, that we call the adjunction $b$-divisor of the pair $(X,\Delta)$\, with respect to $\nu$\, relative to $\kappa(Z)$\,. 
 \begin{lemma}\mylabel{lem:L14} With notation as above, the adjunction $b$-divisor $\Delta^{\kappa(Z)}$\, does not depend on $\nu\in R_{Ab}(K/\mathbb C)$\, with $\overline{c_X(\nu)}=Z$\, and $(\kappa(\nu)=\kappa(Z))$\, and $a(X,\Delta,\nu)=0$\,.
 \end{lemma}  
 \begin{proof}  Let us assume that we are given a second Abhyankar place $\nu'$\, with $a(X,\Delta,\nu')=0,$\, $\overline{c_X(\nu')}=Z$\, and $\kappa(\nu')=\kappa(Z)$\,. We fix a transcendence basis $\overline{x_{k+1}},...,\overline{x_n}$\, of $\kappa(\nu)=\kappa(\nu')=\kappa(Z)/\mathbb C$\, and  lifts $x_{k+1},...,x_n$\, to $\mathcal O_{X,Z} \subset A_{\nu},A_{\nu'}\subset K(X)$\,. Our goal is know to find $s_1,...,s_k\in K(X)$\, that are at the same time a $\mathbb Q$-basis for $\Gamma_{\nu}$\, and $\Gamma_{\nu'}$\,. Assume that we have  reduced to the case that $\nu$\, and $\nu'$\, are not composed with a fixed nontrivial valuation. Fix then 
 $$\alpha_1,...,\alpha_k\in \Gamma_{\nu}\,\, \text{and}\,\, \alpha_1',...,\alpha_k'\in \Gamma_{\nu'}$$ that are $\mathbb Q$-basis (in the ordinary sense) respectively. Then find $r_i ,r_i'$\, with $\nu(r_i)=\alpha_i$\, and $\nu'(r_i')=\alpha_i', i=1,...,k$\,. By the Approximation theorem (see \cite{ZaSam}[Vol.II, chapter 10, Theorem 18, p.45, Theorem 18',p.47]), for each $i=1,...,k$\, we find an algebraic function $s_i$\, such that at the same time  $\nu(s_i-r_i)>\alpha_i$\, and $\nu'(s_i-r_i')>\alpha_i'$\,. By the strong triangle inequality, we have 
 $$\nu(s_i)\geq \min\{\nu(s_i-r_i), \nu(r_i)\}=\alpha_i\,\, \mbox{and}\,\, \nu'(s_i)\geq \min\{\nu(s_i-r_i'),\nu'(r_i')\}=\alpha_i'$$ and we have equality since $\nu(s_i-r_i)>\nu(r_i)=\alpha_i$\, and $\nu'(s_i-r_i')>\nu'(r_i')=\alpha_i'$\,. Thus, for $i=1,...,k$\, we have $\nu(s_i)=\alpha_i$\, and $\nu'(s_i)=\alpha_i'$\, and $s_1,...,s_k$\, is at the same time a $\mathbb Q$-basis for $\nu$\, and $\nu'$\,.\\
 Now, assume that $\nu=\mu\circ \xi$\, and $\nu'=\mu\circ \xi'$\, with $\xi,\xi'\in R_{Ab}(\kappa(\mu)/\mathbb C)$\, such that $\xi$\, and $\xi'$\, are not composed with the same valuation. We apply the above considerations to $\xi, \xi'$\, to find a common $\mathbb Q$-basis $\overline{y_1},...,\overline{y_l}$\, in $\kappa(\mu)$\,. We choose lifts $y_1,...y_l$\, to $A_{\mu}\subset K(X)$\,. By the standard arguement, there are exact sequences
 \begin{gather*}0\longrightarrow \Gamma_{\xi}\longrightarrow \Gamma_{\mu\circ \xi}\longrightarrow \Gamma_{\mu}\longrightarrow 0\,\,\text{and}\\
 0\longrightarrow \Gamma_{\xi'}\longrightarrow \Gamma_{\mu\circ\xi'}\longrightarrow \Gamma_{\mu}\longrightarrow 0.
 \end{gather*}
 We can then extend $y_1,...,y_l$\, to a common $\mathbb Q$-basis $s_1,...,s_k$\, for $\mu\circ \xi$\, and $\mu\circ \xi'$\, as was to be shown.\\
 \\
 Turning to our original problem, we have  fixed algebraic functions 
 $$s_1,...,s_k\in K(X)\,\, \mbox{and}\,\, x_{k+1},...x_n\in K(X)$$
  such that the residues $\overline{x_{k+1}},...,\overline{x_n}$\, form a transcendence basis for $\kappa(\nu)=\kappa(\nu')/\mathbb C$\,. Thus we may write each top rational differential form as 
 $$\omega=f\cdot d^1s_1\wedge ...\wedge d^1s_k\wedge d^1x_{k+1}\wedge ...\wedge d^1x_n.$$
 We now determine $f\in K(X)$\, such that 
 $$\nu(f\cdot s_1\cdot ...\cdot s_k)=0\,\, \mbox{and}\,\, \nu'(f\cdot s_1\cdot ....\cdot s_k)=0.$$
  By the approximation theorem, (see \cite{ZaSam}[Vol.II, chapter 10, Theorem 18, p.45, Theorem 18', p. 47]) we find $f\in K(X)$\, such that at the same time 
 $$\nu(f-\frac{1}{s_1\cdot ...\cdot s_k})>-\sum_{i=1}^k\alpha_i\,\,\text{and}\,\,\nu'(f-\frac{1}{s_1\cdot ...\cdot s_k})>-\sum_{i=1}^k\alpha_i'.$$
 Then, from the same arguement as above, we get that 
 $$\nu(f)=\nu(\frac{1}{s_1\cdot ...\cdot s_k})\,\,\text{and}\,\,\nu'(f)=\nu(\frac{1}{s_1\cdot ...\cdot s_k})$$ and thus 
 $$\nu(\omega)=\nu(f\cdot s_1\cdot ...\cdot s_k)=0\,\,\text{and}\,\, \nu'(\omega)=\nu'(f\cdot s_1\cdot ...\cdot s_k)=0.$$
  By \prettyref{thm:T11}, the rational top differential form of $\kappa(Z)/\mathbb C$\,
  $$ \overline{\omega}=\overline{f\cdot s_1\cdot ...\cdot s_k}\cdot d^1\overline{x_{k+1}}\wedge ...\wedge d^1\overline{x_n} $$ does not depend on $\nu$\, or $\nu'$\,.\\
 Now, our situation was, that we are given a log pair $(X,\Delta)$,\, an lc-center $Z\subset X$ and two lc-places $\nu,\nu'\in R_{Ab}(K(X)/\mathbb C)$\, with 
 $$\overline{c_X(\nu)}=\overline{c_X(\nu')}=Z\,\, \mbox{and}\,\, \kappa(\nu)=\kappa(\nu')=\kappa(Z).$$ For simplicity, we assume that $Z\subset X$\, is normal. By assumption, we have $a(X,\Delta,\nu)=a(X,\Delta,\nu')=0$\,. With the above rational top differential form $\omega$\, we have 
 $(K_X^{\omega}+\Delta)\mid_Z=K_Z^{\overline{\omega}}+\Delta_Z$\, and as $K_X^{\overline{\omega}}$\, does not depend on the choosen Abhyankar place, neither does $\Delta_Z$\,.
 \end{proof}
\subsection{Log canonical centers}
We propose the following 
\begin{definition} Let $(X,\Delta)$\, be a normal lc log variety. A generalized log canonical center $C\subset X$\, is an integral subvariety such that there is an Abhyankar place $\nu\in R_{Ab}(K(X)/\mathbb C)$\, with $a(X,\Delta, \nu)=0$\, and $\overline{c_X(\nu)}=C$\,.\\
An integral subvariety $C\subset X$\, is called a generalized minimal log canonical center iff it is a generalized log canonical center and no proper integral subvariety $C'\subsetneq C$\, is a generalized log canonical center of $(X,\Delta)$\,.
\end{definition} 
To demonstrate the method, we prove the following fact that is well known for classical minimal lc centers.
\begin{proposition} Let $(X,\Delta)$\, be a normal lc log variety and $C\subset X$\, a generalized minimal lc center that is assumed to be normal. Assume moreover that we can find $\nu\in R_{Ab}(K(X)/\mathbb C)$\, with $a(X,\Delta,\nu)=0$\, and $\overline{c_X(\nu)}=C$\, and , in addition, $\kappa(\nu)=\kappa(C)$\,.\\
Then,  the log variety $(C,\Delta_C)$\, defined by adjunction is Kawamata log terminal.
\end{proposition}
\begin{proof} Suppose to the contrary that $(C,\Delta_C)$\, is not klt. Then, ther exists a proper birational modification $p:\widetilde{C}\longrightarrow C$\, and a prime divisor $\overline{F}$\, of the function field $\kappa(C)/\mathbb C$\, having divisorial support on $\widetilde{C}$\, such that $a(C,\Delta_C,\nu_{\overline{F}})\leq 0$\,. We form the composed valuation $\nu\circ \nu_{\overline{F}}$\, which is possible since $\kappa(\nu)=\kappa(C)$\,. Recall that in order to determine the log variety $(C,\Delta_C)$\, , we choose a rational top differetial form $\omega$\, with $\nu(\omega)=0$\, and apply the Poincar$\acute{\mbox{e}}$\, residue map to obtain a rational top differential form $\overline{\omega}$\, of $\kappa(C)/\mathbb C$\, and put $K_C^{\overline{\omega}}+\Delta_C= (K^{\omega}+\Delta)\mid_C$\,. By our valuation formula, for $\mu=\nu\circ \nu_{\overline{F}}$\, because of $\nu(\omega)=0$\, we have 
$$\mu(\omega)=\nu_{\overline{F}}(\overline{\omega})\,\, \text{and similarely}\,\, \mu(K_X^{\omega}+\Delta)=\nu_{\overline{F}}((K_X^{\omega}+\Delta)\mid_C).$$
These equalities are of course valid under the canonical inclusion of value groups $\Gamma_{\nu_{\overline{F}}}\hookrightarrow \Gamma_{\mu}$\,.
Thus we get 
$$a(X,\Delta,\mu)=a(C,\Delta_C,\overline{F})\leq 0.$$ But by \ref{thm:T10}, $(X,\Delta)$\, is globally lc and thus $a(X,\Delta,\mu)=0$\,. Now by standard valuation theory $c_C(\nu_{\overline{F}})=c_X(\nu\circ \nu_{\overline{F}})$\, and $c_C(\nu_{\overline{F}})=p(\overline{F})\subsetneq C$\,, where $p$ denoted the proper birational morphism $p: \widetilde{C}\longrightarrow C$\, with $\overline{F}\subset \widetilde{C}$\,. Thus 
$C'=c_X(\mu)\subsetneq C$\, is a generalized lc center of the log variety $(X,\Delta),$\, contradicting the minimality of $C$. 
\end{proof}
 \section{Outlook}
 In this section, we will scetch generalizations and applications of our main theorems. The intention of this paper was of course to introduce methods from general valuation theory to modern birational geometry, in particular the log minimal model program. The theory of lc-places and centers is generalized to arbitrary Abhyankar places. One hope is that introducing generalized lc centers and lc places and using on them adjunction makes the whole theory more flexible, in particular for running inductive arguements.\\ One possible further application is the development of the LMMP for log canonical pairs which requires nowadays the theory of quasi log varieties. If one is given a nonexceptional lc-center $Z\subset (X,D)$\,, the idea is to construct an Abhyankar  place $\nu$\, lying generically finite over $Z$ (meaning that $\kappa(\nu)/\kappa(Z)$\, is a finite algebraic extension), use adjunction to define an lc pair $(Z,D_Z)$\, or at least an lc pair 
 $$(W, D_W),\,\, W\subset (X',D')\,\,, X'\longrightarrow X,\,\, \overline{c_{X'}(\nu)}= W$$ and $W\longrightarrow Z$\, surjective and generically finite. Usually, we need to consider on a log resolution $(X',D')\longrightarrow (X,D)$\,  an SNC-divisor $E=\cup_{i=1}^kE_i,$\, where each $E_i$\, is an lc place of $(X,D)$\, with center in $Z$\, plus a surjective proper morphism $E\longrightarrow Z$\, and use the theory of quasi log varieties (see \cite{Fujino1} und \cite{Fujino2}. This possibly very much simplifies the theory because e.g. it avoids the necessity of proving vanishing theorems etc. on quasi log varieties.\\
 \\
 \\
It is possible to generalize our valuation formula to algebraic function fields over a base field of characteristic $p$\,. Given $\omega\in \Lambda^{max}(\Omega^{(1)}(K(X)/k)$\, and, say $\nu\in R_{Ab}(K(X)/k)$\, of dimension zero, write $\omega=f\cdot d^1x_1\wedge ...\wedge d^1x_n$\, where
\begin{enumerate}[1]
\item $x_1,...,x_n$\, form a $p$-basis for $K(X)/k$\,;
\item the $\mathbb Z$-module  $N:=\mathbb Z\langle \nu(x_1),...,\nu(x_n)\rangle\subset \Gamma_{\nu}$\, has finite index in $\Gamma_{\nu}$\, such that $\sharp\{\Gamma_{\nu}/N\}$\, is a power of $p=\mbox{char}(k).$\,
\end{enumerate}
Then, one proves that, if one defines $\nu(\omega)=\nu(f)+\nu(x_1)+...+\nu(x_n)$\,, the value $\nu(\omega)$\, is independent of the choice of $x_1,...,x_n$\, with the properties $(1)$\, and $(2)$\,. Then, the theory is developed as in characteristic $p=0$\,.  
\bibliography{ValuationLMMP}

\begin{thebibliography}{10}

\bibitem{Fujino1}
Osamu Fujino.
\newblock Introduction to the log minimal model program for log canonical
  pairs.
\newblock {\em arXiv.org:0907.1506v1[math.AG]}, 2009.

\bibitem{Fujino2}
Osamu Fujino.
\newblock Introduction to the theory of quasi-log varieties.
\newblock {\em arXiv.org:0711.1203v8}, 2009.

\bibitem{SGA}
Alexander Grothendieck.
\newblock {\em $S\acute{e}minaire\,\, de\,\, G\acute{e}om\acute{e}trie\,\,
  Alg\acute{e}brique\,\, I$}.
\newblock Lecture Notes in Mathematics, Springer Verlag Berlin,Heidelberg, New
  York, 1961.

\bibitem{Promotion}
Stefan Guenther.
\newblock Higher {K}aehler {D}ifferential {C}alculus in {C}ommutative {A}lgebra
  and {A}lgebraic {G}eometry.
\newblock {\em Dissertation, Universitaet Osnabrueck}, 2014.

\bibitem{Guenther}
Stefan Guenther.
\newblock Valuation {T}heory, {R}iemann varieties and the {S}tructure of
  {I}ntegral {P}reschemes.
\newblock {\em arXiv.org:1610.07832, math/AG}, 2016.

\bibitem{Kaplansky}
Irving Kaplansky.
\newblock Maximal fields with valuations.
\newblock {\em Duke mathematical journal 9, pp.303-321}, 1942.

\bibitem{Kollar}
Janos Kollar.
\newblock {\em Singularities of the Minimal Model program}.
\newblock Cambridge University press, 2013.

\bibitem{Matsuki}
Kenji Matsuki.
\newblock {\em Introduction to the {M}ori {P}rogram}.
\newblock Springer-Verlag New York Berlin Heidelberg, 2002.

\bibitem{ZaSam}
Pierre~Samuel Oscar~Zariski.
\newblock {\em Commutative {A}lgebra, I,II}.
\newblock Springer Verlag New York, Berlin, London, 1960.

\bibitem{Bochner}
William Ted~Martin Salomon~Bochner.
\newblock {\em Several complex variables}.
\newblock Princeton University press, 1948.

\bibitem{Hoeven}
Joris van~der Hoeven.
\newblock Operators on generalized power series.
\newblock {\em Dep. of {M}athematics, bat. 425, Universit$\acute{e}$
  {P}aris-{S}ud, 91405 {O}rsay {CEDEX}, {F}rance}, 2008.

\bibitem{Val}
Michel Vaquie.
\newblock Valuations.
\newblock {\em Hauser,H. et. al. , Resolution of singularities, Birkhauser.
  Basel, Progress in Mathematics 181, pp. 539-590}, 2000.

\end{thebibliography}
\bibliographystyle{plain}
\noindent
\emph{E-Mail-adress:}verb!stef.guenther2@vodafone.de!
\end{document}